\documentclass[12pt, a4paper, twoside, notitlepage]{article}
\usepackage[utf8]{inputenc}
\usepackage{appendix}
\usepackage{url}
\usepackage{caption}
\usepackage[T1]{fontenc}
\usepackage{lmodern}
\usepackage[margin=2.4cm]{geometry}
\usepackage{amsmath,amssymb,amsthm,mathtools}
\usepackage{booktabs}
\usepackage{array}
\usepackage{colortbl}
\usepackage{enumitem}
\usepackage{graphicx}
\usepackage{hyperref}
\usepackage{cleveref}
\usepackage{todo}
\crefname{theorem}{Thm.}{Thms.}
\Crefname{theorem}{Thm.}{Thms.}
\crefname{proposition}{Prop.}{Props.}
\Crefname{proposition}{Prop.}{Props.}
\crefname{remark}{Rem.}{Rems.}
\Crefname{remark}{Rem.}{Rems.}
\crefname{lemma}{Lem.}{Lems.}
\Crefname{lemma}{Lem.}{Lems.}
\crefname{corollary}{Cor.}{Cors.}
\Crefname{corollary}{Cor.}{Cors.}
\crefname{definition}{Def.}{Defs.}
\Crefname{definition}{Def.}{Defs.}
\crefname{assumption}{Ass.}{Asss.}
\Crefname{assumption}{Ass.}{Asss.}
\crefname{property}{Prop.}{Props.}
\Crefname{property}{Prop.}{Props.}
\usepackage{xcolor}
\usepackage{tikz}
\usetikzlibrary{arrows.meta,calc,patterns}
\usepackage{tcolorbox}
\tcbuselibrary{theorems,breakable}

\definecolor{accent}{HTML}{2D5F8A}
\definecolor{accent2}{HTML}{8A2D5F}
\definecolor{defgreen}{HTML}{44AA99}
\definecolor{thmamber}{HTML}{CC8844}
\definecolor{lemgold}{HTML}{AA8800}
\definecolor{prfpurple}{HTML}{8866CC}
\definecolor{remgrey}{HTML}{8888AA}
\definecolor{corgold}{HTML}{DD9966}
\definecolor{warnred}{HTML}{CC4444}

\hypersetup{
  colorlinks=true, linkcolor=accent, urlcolor=accent, citecolor=accent2,
  pdftitle={Conservation as selection principle: discrete exterior calculus for the incompressible Navier–Stokes and Euler equations},
  pdfauthor={Peter Korn},
  pdfkeywords={discrete exterior calculus; incompressible Navier–Stokes; Euler equations; structure-preserving discretisation; selection principle}
}

\theoremstyle{definition}
\newtheorem{definition}{Definition}[section]

\newtheorem{proposition}{Proposition}[section]
\newtheorem{lemma}{Lemma}[section]
\newtheorem{assumption}{Assumption}[section]
\newtheorem{property}{Property}[section]
\newtheorem{convention}{Convention}[section]

\theoremstyle{plain}
\newtheorem{theorem}{Theorem}[section]
\newtheorem{corollary}{Corollary}[section]
\theoremstyle{remark}
\newtheorem{remark}{Remark}[section]

\makeatletter
\let\c@proposition\c@definition
\let\c@lemma\c@definition
\let\c@assumption\c@definition
\let\c@property\c@definition
\let\c@convention\c@definition
\let\c@theorem\c@definition
\let\c@corollary\c@definition
\let\c@remark\c@definition
\makeatother

\newtcolorbox{eqbox}{
  colback=white!5, colframe=accent!40!black,
  left=8pt, right=8pt, top=6pt, bottom=6pt,
  breakable, sharp corners, boxrule=0.8pt
}

\usepackage{titlesec}
\titleformat{\section}{\Large\bfseries\color{accent}}{\thesection.}{0.5em}{}
\titleformat{\subsection}{\large\bfseries\color{accent}}{\thesubsection}{0.5em}{}

\usepackage{csquotes}
\usepackage{mathrsfs} 
\usepackage{authblk}

\newcommand{\dd}{\mathrm{d}}
\newcommand{\RR}{\mathbb{R}}
\newcommand{\ip}[2]{\langle #1,\, #2 \rangle}

\newcommand{\Lie}{\mathcal{L}}

\newcommand{\nrm}[1]{\lVert #1 \rVert}
\newcommand{\twdg}{\;{\wedge}\;}
\newcommand{\tdd}{\tilde{\dd}}
\newcommand{\ddt}{\frac{\mathrm{d}}{\mathrm{d}t}}
\newcommand{\tD}{\tilde{D}}
\newcommand{\tLie}{{\Lie}}

\newcommand{\bv}{{v}}
\newcommand{\bu}{{u}}
\newcommand{\bw}{{w}}
\newcommand{\bom}{\omega}
\newcommand{\tU}{\widetilde{U}}

\let\caron\v

\DeclareUnicodeCharacter{0161}{\caron{s}}%
\DeclareUnicodeCharacter{0160}{\caron{S}}%
\DeclareUnicodeCharacter{010D}{\caron{c}}%
\DeclareUnicodeCharacter{010C}{\caron{C}}%
\DeclareUnicodeCharacter{017E}{\caron{z}}%
\DeclareUnicodeCharacter{017D}{\caron{Z}}%
\renewcommand{\v}{{v}}
\renewcommand{\u}{{u}}
\newcommand{\w}{{w}}

\newcommand{\Iv}{I_{{\v}}}
\newcommand{\Iw}{I_{{\w}}}

\newcommand{\bM}{M}
\newcommand{\bD}{D}

\newcommand{\KK}{\mathcal{K}}
\newcommand{\KKs}{\mathcal{K}^*}
\newcommand{\OO}{\mathcal{O}}

\newcommand{\bP}{P}
\newcommand{\bQ}{Q}
\newcommand{\bA}{A}
\newcommand{\bff}{\mathrm{f}}

\newcommand{\bPhi}{\boldsymbol{\Phi}}

\newcommand{\I}{\mathcal{I} }

\renewcommand{\div}{\mathrm{div}\,}

\newcommand{\curl}{\mathrm{curl}\,}

\newcommand{\bn}{\boldsymbol{n}}
\newcommand{\bt}{\boldsymbol{t}}

\newcommand{\br}{\boldsymbol{r}}

\newcommand{\Qh}{\mathcal{Q}_h}

\newcommand{\ekin}{e^{\mathrm{kin}}}
\newcommand{\Ekin}{E^{\mathrm{kin}}}

\renewcommand{\epsilon}{\varepsilon}

\newcommand{\codiff}{\delta_h}
\newcommand{\Deltah}{\Delta_h}
\begin{document}

\title{
{\LARGE\bfseries\color{accent} 
Exact conservation as selection principle: discrete exterior calculus for the incompressible Navier--Stokes and Euler equations
}
}
\author{Peter Korn\thanks{Max-Planck Institute for Meteorology, Imperial College London. Email: \href{mailto:peter.korn@mpimet.mpg.de}{peter.korn@mpimet.mpg.de}}}
\date{}%

\maketitle

\begin{abstract}
We formulate a new discrete-exterior calculus based discretisation of the incompressible Euler and Navier--Stokes equations that preserves the geometric structure of the continuum, and establish a rigorous convergence and structure theory for a new discretisation. The discretisation operates on prismatic Delaunay--Voronoi meshes over closed Riemannian manifolds. The geometry of Euler and Navier--Stokes equations  is maintained via a discrete Lie derivative that is built from an extrusion-based contraction for the nonlinear term in vector-invariant form.
Conservation of energy and Kelvin circulation links the discrete scheme to the continuum: at the discrete level, energy conservation is a stability property, and in the vanishing-resolution limit it becomes both a constructive route into the conservative weak-solution theory of the continuum equations and a selection principle on the limits the scheme can reach. This correspondence appears in four regimes.
\emph{Smooth solutions}: convergence at rate $\OO(h^{\min(r_{\rm rec},\,r_\star)}\,|\log h|)$ in dimensions
$d=2,3$, uniformly in viscosity $\nu \ge 0$; first order on general meshes, second order under centroid proximity and reconstruction symmetry.
\emph{Leray--Hopf weak regime}: subsequential $L^2$ limits of the discrete Navier--Stokes system are weak solutions of the viscous equations.
\emph{Inviscid measure-valued regime}: limits are conservative measure-valued Euler solutions, with concentration defect vanishing above the Onsager threshold $\alpha > 1/3$ provided the discrete solutions admit a uniform $C^{0,\alpha}$ bound; the scheme reaches the energy-conserving side of the Onsager landscape but not the dissipative side.
\emph{Dissipative regime}: no subsequence converges to an energy-dissipating Euler solution at any H\"older regularity, an exclusion that follows from discrete energy conservation.
\end{abstract}



\medskip
\newpage
 \tableofcontents
\newpage

\section{Introduction and Overview}
This paper formulates a new discretisation of the incompressible Euler and Navier--Stokes
equations that preserves the geometric structure of the continuum via a discrete Lie derivative
and  develops a convergence and structure theory for the scheme.
Two invariants carry that structure. First, conservation of kinetic energy
is, at the discrete level, a stability property: the discretised
advection redistributes energy across scales, but neither creates nor destroys it. Second, conservation of
Kelvin circulation, a discrete theorem holding on admissible
meshes, is shared by only a few discretisations and separates the present scheme from other
energy-conserving methods. The discretisation also
conserves helicity, the topological invariant measuring vortex-tube
linkage, to truncation order.
Both invariants translate to the continuum limit where they
decide which continuum solutions the scheme can approximate: the
conservation structure constructs the conservative weak and
measure-valued solutions of the Euler and Navier--Stokes equations
and excludes the dissipative ones. In the vanishing-resolution limit, this stability property of the
scheme becomes a selection principle on the continuum.

The motivation is concrete: the incompressible Navier--Stokes equations
form the nucleus of the dynamical core of non-hydrostatic Boussinesq ocean
models, and the discrete operator framework analysed here originates in
the unstructured-grid ocean model of~\cite{korn2017}. In that setting the
gap between exact and approximate conservation governs the fidelity of the
solutions over long times: advection operators that conserve energy and
circulation only to truncation order act as a spurious, mesh-dependent
source or sink, and over the long integrations characteristic of turbulent
ocean and climate simulation these errors accumulate into energetic drift
or nonlinear instability, forcing the artificial dissipation that
contaminates the resolved-scale statistics one wishes to
predict~\cite{arakawalamb1977}. With exact conservation, energy enters and
leaves the discrete system only through the physical channels of viscosity
and forcing, never through discretised advection. The selection
principle developed below is the rigorous counterpart of this requirement,
valid on any bounded interval $[0,T]$ with $T$ arbitrarily large and
independent of the initial data.

We work throughout in vector-invariant form,
\begin{align*}
    \partial_t \v + (\mathrm{curl}\,\v)\times\v
      + \nabla\bigl(p + \tfrac{|\v|^2}{2}\bigr)
      = \nu\Delta\v, \qquad
    \nabla\cdot \v = 0, \qquad \nu \ge 0,
\end{align*}
in which the nonlinearity is carried by the Lamb vector
$(\mathrm{curl}\,\v)\times\v$ and pressure and kinetic energy gradient
are absorbed into the Bernoulli function $p + |\v|^2/2$. The vector-invariant form
is chosen for structural reasons. On prismatic Delaunay--Voronoi
meshes the discrete Lamb vector is antisymmetric,
$\ip{\bv}{\Iv(\bom)}_1 = 0$, i.e., the discrete Lamb form annihilates
its own velocity argument, and the discrete exterior derivative
commutes with de~Rham interpolation,
$\tD_1\mathcal{R}_h = \mathcal{R}_h\,d$, i.e., interpolating a
continuous form commutes with taking its exterior derivative. These
two algebraic facts underlie every conservation law and convergence
result established below, with no assumption of smoothness, mesh
symmetry, or an inf--sup condition. The same antisymmetry yields the discrete
helicity balance~\cite{moffatt1969}.

\paragraph{Exact discrete conservation: a constructive route and a selection principle.}
The four regimes in which we analyse the DEC scheme -- smooth,
Leray--Hopf weak, measure-valued inviscid, and dissipative -- are one
conservation structure unfolding through regularity. In the first three the limit identity is constructive,
yielding solutions of continuum equations from
discrete dynamics; in the fourth it is excluding
convergence to energy-dissipating Euler solutions. The results
therefore join three perspectives, reviewed
below: the convergence theory of Navier--Stokes discretisations,
which provides convergence rates for smooth solutions and weak limits in their
absence; the weak-solution and regularity theory of the continuum
equations, which supplies the landscape -- Leray--Hopf solutions, the
Onsager threshold, measure-valued solutions -- against which the
limits are classified; and structure-preserving discretisation, which
supplies the conservation laws but, to date, no convergence theory
for the nonlinear problem.

\paragraph{Convergence theory for Navier--Stokes discretisations.}
For smooth solutions, convergence rates for conforming and mixed
finite element methods are classical
(see e.g.~\cite{john2018finite}), but these typically
require an inf--sup stable pair and do not address the inviscid limit
$\nu \to 0$ uniformly. The present scheme achieves convergence rates
uniform in $\nu \ge 0$, because the inviscid and viscous discrete
systems share the same nonlinear structure. In the absence of
smoothness assumptions, Guermond~\cite{guermond2006,guermond2007}
proved that a broad class of finite-element approximations converges
to weak solutions of the Navier--Stokes equations. Our Leray--Hopf
existence result (\Cref{thm:weak_convergence}) is of this type but is
obtained directly from the discrete system: uniform a~priori bounds
from the discrete energy inequality and Aubin--Lions compactness
identify the limit as a weak solution.

\paragraph{Weak solutions and the regularity landscape.}
The global existence of weak solutions to the incompressible
Navier--Stokes equations was established by Leray~\cite{Leray1934}
and Hopf~\cite{Hopf1951}, and uniqueness remains open in three
dimensions. For the inviscid Euler equations, the regularity
landscape is shaped by the \emph{Onsager conjecture}: weak solutions
in the H\"older class $C^{0,\alpha}$ conserve energy if
$\alpha > 1/3$, while for $\alpha \le 1/3$ energy-dissipating
solutions exist. The positive direction was established by
Constantin, E, and Titi~\cite{CET1994} (see also
Duchon--Robert~\cite{DuchonRobert2000}); the negative direction was
resolved through the convex-integration programme initiated by
De~Lellis and
Sz\'ekelyhidi~\cite{DeLellisSzekelyhidi2009,DeLellisSzekelyhidi2013},
building on Scheffer~\cite{scheffer1993} and
Shnirelman~\cite{shnirelman1997}, and completed by
Isett~\cite{Isett2018}, with the refinement of Buckmaster, De~Lellis,
Sz\'ekelyhidi, and Vicol~\cite{BDLSV2019}. Below $1/3$, weak
solutions are maximally non-unique; above $1/3$, energy is conserved
but uniqueness remains open. At Lipschitz regularity
($\alpha \ge 1$), Brenier, De~Lellis, and
Sz\'ekelyhidi~\cite{BDLS2011} proved weak--strong uniqueness among
conservative measure-valued solutions. The gap
$1/3 < \alpha < 1$, where energy is conserved but uniqueness is
open, isolates a central open problem of inviscid fluid dynamics.

\paragraph{Structure-preserving discretisation.}
The principle that numerical schemes should preserve geometric and
physical structures of the continuous equations has a long history.
In the exterior calculus setting, the foundational framework is
Finite Element Exterior Calculus (FEEC) of Arnold, Falk, and
Winther~\cite{arnold2006,arnold2010}, which constructs discrete
de~Rham complexes on simplicial meshes and establishes
approximation properties for the Hodge Laplacian. FEEC provides the
topological infrastructure (exact sequences, commuting projections,
discrete Poincar\'e inequalities) but does not address the nonlinear
advection structure of the Navier--Stokes equations. The construction
analysed here extends the FEEC programme to incompressible fluid
dynamics by equipping the discrete complex with a nonlinear extrusion
operator and a Lie derivative. 
Closely related is the body of work on compatible finite element methods
for geophysical fluid dynamics~\cite{cotter2023}, the finite element
descendant of the Arakawa C-grid~\cite{arakawalamb1977} that underpins
operational dynamical cores; the present scheme is, at the level of its
topological skeleton, the lowest-order, mass-lumped member of that family.
We make this correspondence precise in \Cref{subsect:cfem}, including
the relation to the classical $H(\div)$ and $H(\curl)$ spaces.

DEC fluid mechanics was pioneered by Elcott et al.~\cite{elcott2007},
who introduced discrete Lie derivatives for fluids on surfaces.
Mullen et al.~\cite{mullen2009} and Pavlov et al.~\cite{pavlov2011}
developed the variational and conservation structure, establishing
energy and Kelvin circulation conservation on simplicial meshes;
no error estimates were proved. We close that gap here.
On simplicial meshes the contraction loops over simplicial faces constructs vortex stretching automatically; on prismatic cells the horizontal/vertical split breaks that symmetry and stretching must be recovered through the extrusion cross-terms. The sharper contrast is with the reconstruction-based advection of the operational ocean/C-grid lineage on these same meshes (\cite{korn2017,cotter2023}): there transport and stretching are separate, hand-constructed terms, whereas the Lie transport here produces both from one contraction.

Other structure-preserving approaches (mimetic finite differences
\cite{BrezziLipnikovShashkov2005} and their variational descendants,
the virtual element methods on general polytopal
meshes~\cite{beiraoVeiga2013basic}, including divergence-free
virtual elements for the Navier--Stokes
equations~\cite{beiraoVeiga2018NS}, the energy/momentum-conserving EMAC
formulation \cite{charnyi2017}, and helicity-conserving dual-field
discretisations \cite{zhang2022}) each preserve some invariants but
differ in scope or mechanism, and none combines simultaneous
energy and circulation conservation with a convergence theory for the
nonlinear problem. 
The relation of DEC to the polytopal family
parallels its relation to compatible finite elements: lowest-order
virtual elements are equivalent to mimetic finite differences, whose
local inner products are fixed by consistency only up to a
stabilisation term; on Delaunay--Voronoi meshes the primal--dual
orthogonality selects the diagonal Hodge star as a consistent inner
product requiring no stabilisation, and this selection is rigid: the
DEC star is the unique diagonal star accurate on constants, and
no diagonal star is exact on a non-orthogonal pairing
(\Cref{prop:hodge_rigidity}).

\Cref{tab:prior_art} summarises the comparison: the
combination of simultaneous preservation of energy and Kelvin
circulation with a complete regularity-landscape convergence
theory, on the prismatic Delaunay--Voronoi meshes used by operational
ocean models, distinguishes the scheme from the prior structure-preserving
discretisations.

\begin{table}[h!]
\centering
\renewcommand{\arraystretch}{0.9}
\small
\setlength{\tabcolsep}{2.0pt}
\begin{tabular}{@{}l c c c l l@{}}
\toprule
Scheme & Energy & Kelvin & Helicity & Conv.\ proof & Mesh class \\
\midrule
FEEC \cite{arnold2006,arnold2010} & --- & --- & --- & Hodge Laplace & simplicial \\
CFEM / C-grid \cite{cotter2023,thuburncotter2015} & exact & --- & --- & linear/wave analysis & quad/tri FE, prismatic \\
EMAC \cite{charnyi2017} & exact & --- & --- & 2D FEM only & simplicial FEM \\
Elcott et al.\ \cite{elcott2007} & exact & exact & --- & none & simplicial surfaces \\
Mullen/Pavlov \cite{mullen2009,pavlov2011} & exact & exact & --- & none & simplicial \\
Zhang et al.\ \cite{zhang2022} & exact & --- & exact & linear analysis only & simplicial dual-field \\
Mimetic FD \cite{BrezziLipnikovShashkov2005} & approx. & --- & --- & linear problems only & polyhedral \\
Div-free VEM \cite{beiraoVeiga2018NS} & approx. & --- & --- & viscous, smooth rates & polygonal \\
Galerkin/pseudospectral & approx. & --- & --- & classical & tensor product \\
\midrule
\textbf{This work} & \textbf{exact} & \textbf{exact} & $\OO(h^{\min(r_{\rm rec},r_\star)})$ & \textbf{full reg.\ landscape} & \textbf{prismatic D--V} \\
\bottomrule
\end{tabular}
\caption{\it\small Position of the present work relative to existing structure-preserving
discretisations of incompressible Euler/Navier--Stokes. ``Exact'' means the
invariant is conserved as a discrete algebraic identity.
``Conv.\ proof'' refers to rigorous convergence to continuous solutions of
the nonlinear equations. ``Prismatic D--V'' abbreviates polygonal prismatic
Delaunay--Voronoi meshes, the class used by ICON~\cite{korn2017} and
MPAS~\cite{ringler2010}.}
\label{tab:prior_art}
\end{table}

\paragraph{Contributions and proof strategy.}
The scheme's defining structural feature is that its nonlinearity is a discrete Lie transport, built from an extrusion contraction (\cref{def:contraction,def:Lie}), realising transport and vortex stretching through a single operator on prismatic meshes. The convergence proof then rests on two algebraic facts: antisymmetry of the discrete Lamb form and de~Rham commutativity, which together yield conservation laws (\Cref{thm:energy,thm:kelvin,thm:helicity}) 
and stability estimates. These properties imply also 
 global well-posedness of both inviscid and viscous discrete systems on fixed
meshes (\Cref{thm:global,thm:NS_global}). Smooth-regime convergence
(\Cref{thm:convergence,thm:NS_convergence}) follows the
consistency--stability paradigm. Consistency exploits the fact that
the Leray projection annihilates gradients, leaving the
extrusion error at rate $r_{\rm rec}$ combined with the Hodge-star
accuracy $r_\star$. Stability rests on three ingredients: Lamb
antisymmetry removes the cubic self-interaction; a sharp pointwise
bound on the projection residual (\Cref{lem:proj_error}), obtained by
exploiting Delaunay--Voronoi orthogonality to identify the discrete
Hodge Laplacian with a cell-centred finite-volume Laplacian, controls
the reconstruction error; and a polarised antisymmetry identity
eliminates the bilinear Leray remainder. In
the absence of smoothness, uniform a~priori estimates from the
discrete energy law yield Aubin--Lions compactness, and the energy
inequality identifies the limit as a Leray--Hopf solution; for the
inviscid system, discrete energy conservation passes to the
limit, yielding conservative measure-valued solutions and excluding
dissipative ones. Pressure is not an evolution variable of the
discrete state: the weak formulations are tested against discretely
divergence-free fields, so pressure never enters the convergence
proofs, and it is recovered a~posteriori from the Bernoulli function
via a discrete Poisson problem (\Cref{rem:pressure_LH}).

\paragraph{Numerical validation.} The scheme introduced here is the
non-hydrostatic member of a family whose hydrostatic version coincides
structurally with the discrete operator framework of the operational
ocean and coupled Earth-system model ICON, with which it shares the mesh, Hodge star, coboundary operators and reconstruction framework; the nonlinear Lie-transport advection analysed here is the new construction. ICON's validation on production-scale meshes~\cite{korn2017,korn2022james,hohenegger2023gmd} establishes that the meshes, operators and conservation structure are accurate and stable at scale, 
not the convergence rates and four regularity regimes (\Cref{sect:cont_limit})
proved below. Direct numerical validation of those is the subject of a
separate computational paper.

\paragraph{Structure.}
\Cref{sect:functional_framework} develops the DEC framework: mesh geometry, exterior derivatives, Hodge star, de~Rham and Whitney maps, interpolation and Hodge star error estimates.
\Cref{section_EulerIncomp} formulates the discrete Euler equations and establishes their conservation laws (energy, circulation, helicity) and the discrete Lamb vector, Lie derivative, and wedge product.
\Cref{section_WellPosedness} proves discrete well-posedness for both Euler and Navier--Stokes.
\Cref{sect:cont_limit} establishes convergence, from smooth solutions (with explicit rates) 
through the Leray--Hopf weak-solution theory to the Onsager threshold for the inviscid system.
\Cref{sect:bounded_domains} extends all results to bounded domains with Dirichlet boundary conditions.

\subsection*{Summary of Main Results}

The paper makes three contributions. 
First, it formulates a DEC scheme on prismatic Delaunay--Voronoi meshes whose nonlinearity is generated by a contraction 
$\Iv$ (\cref{def:contraction}) from which discrete wedge product, Lamb vector and Lie derivative are all built
 (\cref{prop:extrusion}). Vorticity dynamics are a discrete Lie transport: 
  flux-form transport in 2D, in 3D transport with vortex stretching; 
  the dimensional distinction is carried by the prismatic mesh geometry, not by a change of the algebraic formula. 
Conservation of energy, Kelvin circulation, and helicity to truncation order follows from this structure.
Second, both the inviscid
and the viscous discrete systems are globally well-posed. Third, the conservation properties determine the
scheme's continuum limits across the regularity landscape:
explicit convergence rates for smooth solutions, construction of
Leray--Hopf and conservative measure-valued solutions, and the exclusion of
energy-dissipating Euler limits. 

All results hold on closed oriented Riemannian $d$-manifolds
($d=2,3$) with Delaunay--Voronoi meshes (\cref{ass:mesh_reg}); they
extend to bounded domains with Dirichlet boundary conditions
(\cref{sect:bounded_domains}). The scheme's four convergence modes
are summarised in the following Main Theorem and depicted in
\Cref{fig:landscape}. 

\begin{theorem}[Main Theorem]\label{thm:main}
Let $\bv^h$ be the semi-discrete DEC solutions of the incompressible
Euler $(\nu=0)$ or Navier--Stokes $(\nu>0)$ equations on $[0,T]$, on a
Delaunay--Voronoi mesh of spacing $h$ satisfying
\Cref{ass:mesh_reg}, and with initial data
$\mathcal{R}_h\bu_0^\flat$. The discrete energy law and Lamb
antisymmetry produce four convergence modes, distinguished  by the
regularity of the continuous solution:
\begin{description}[nosep]
  \item[(I) Smooth regime.]
    If the continuous equations admit a solution
    $\bu\in C^1([0,T];W^{r_{\rm rec},\infty})$ with initial datum
    $\bu_0$, then
    \[
      \sup_{t\in[0,T]}\nrm{\bv^h(t) - \mathcal{R}_h\bu^\flat(t)}_{L_h^2}
      \le C(T)\,h^{\min(r_{\rm rec},\,r_\star)}\,|\log h|
    \]
    uniformly in $\nu\ge 0$, for $d=2,3$. The rate is first order on
    general and second order on meshes with centroid
    proximity and reconstruction symmetry
    (\cref{thm:convergence},\cref{thm:NS_convergence}).
  \item[(II) Weak regime (viscous, $\nu>0$).]
    Without smoothness assumptions, subsequences of the Whitney
    reconstructions $\mathcal{W}_h\bv^h$ converge in $L^2(0,T;L^2)$ to
    Leray--Hopf weak solutions of the Navier--Stokes equations
    (\cref{thm:weak_convergence}). If a strong solution exists on
    $[0,T]$, weak--strong uniqueness (\cref{thm:weak_strong}) upgrades
    this to full-sequence convergence at the rate of (I).
  \item[(III) Measure-valued regime (inviscid, $\nu = 0$).]
    Subsequences of $\mathcal{W}_h\bv^h$ converge in the
    measure-valued sense to a conservative measure-valued
    Euler solution $(\bw,\sigma)$ (\cref{thm:CMV}). Under the
    additional hypothesis of a uniform $C^{0,\alpha}$ bound on the
    Whitney reconstructions with $\alpha>1/3$, the concentration
    defect vanishes (\cref{thm:sigma_vanish}); for Lipschitz weak
    Euler solutions $\bu\in L^\infty(0,T;W^{1,\infty})$ existing on
    $[0,T]$, full-sequence convergence holds with $\sigma=0$
    (\cref{thm:Holder_unconditional}).
  \item[(IV) Dissipative regime.]
    \emph{No subsequence of $\mathcal{W}_h\bv^h$ can converge to an
    energy-dissipating Euler solution at any H\"older regularity,
    regardless of mesh size, time step, or initial data}
    (\cref{prop:no_dissipative}). This exclusion follows from the
    discrete energy conservation inherited by any subsequential
    limit.
\end{description}
Parts~(III) and~(IV) bracket the gap $1/3<\alpha<1$: there the
concentration defect vanishes under the $C^{0,\alpha}$ bound of (III),
but full-sequence convergence would require weak--strong uniqueness
for the Euler equations, which is open.
\end{theorem}

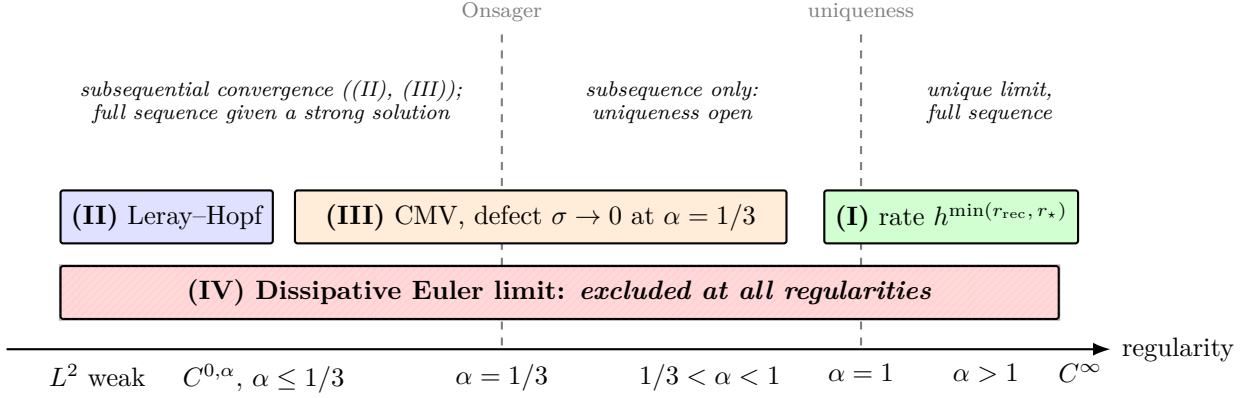
\begin{figure}[h]
\centering
\begin{tikzpicture}[
  >=Latex,
  font=\footnotesize,
  axis/.style={thick, -Latex},
  band/.style={draw, thick, rounded corners=1pt, minimum height=7mm},
  thr/.style={dashed, thick, gray},
]
\draw[axis] (-0.3,0) -- (14.3,0) node[right] {regularity};
\node[below] at (0.9,  -0.05) {$L^2$ weak};
\node[below] at (3.1,  -0.05) {$C^{0,\alpha},\,\alpha\le 1/3$};
\node[below] at (6.25, -0.05) {$\alpha=1/3$};
\node[below] at (9.0,  -0.05) {$1/3<\alpha<1$};
\node[below] at (11.0, -0.05) {$\alpha=1$};
\node[below] at (12.65,-0.05) {$\alpha>1$};
\node[below] at (13.9, -0.05) {$C^\infty$};

\draw[thr] (6.25,0.05) -- (6.25, 4.2);
\draw[thr] (11.0,0.05) -- (11.0, 4.2);
\node[gray, font=\scriptsize, anchor=south] at (6.25, 4.2) {Onsager};
\node[gray, font=\scriptsize, anchor=south] at (11.0, 4.2) {uniqueness};

\node[band, fill=red!15, minimum width=132mm, anchor=west] at (0.4, 0.75) 
  (IV) {};
\draw[pattern=north east lines, pattern color=red!50, opacity=0.35]
  (IV.south west) rectangle (IV.north east);
\node at (IV.center) {\textbf{(IV) Dissipative Euler limit: \emph{excluded at all regularities}}};

\node[band, fill=blue!12, minimum width=19mm, anchor=west] at (0.4, 1.75)
  (II) {\textbf{(II)} Leray--Hopf};

\node[band, fill=orange!15, minimum width=65mm, anchor=west] at (3.5, 1.75)
  (III) {\textbf{(III)} CMV, defect $\sigma\to 0$ at $\alpha=1/3$};

\node[band, fill=green!18, minimum width=33mm, anchor=west] at (10.5, 1.75)
  (I) {\textbf{(I)} rate $h^{\min(r_{\rm rec},\,r_\star)}$};

\node[anchor=south, text width=52mm, align=center, font=\scriptsize\itshape] 
  at (3.2, 2.85) {subsequential convergence ((II), (III)); \\[-1pt] full sequence given a strong solution};

\node[anchor=south, text width=38mm, align=center, font=\scriptsize\itshape] 
  at (8.5, 2.85) {subsequence only: \\[-1pt] uniqueness open};

\node[anchor=south, text width=30mm, align=center, font=\scriptsize\itshape] 
  at (12.7, 2.85) {unique limit,\\[-1pt] full sequence};

\end{tikzpicture}
\caption{\small\it Regularity landscape and four convergence modes of
the DEC scheme (\Cref{thm:main}). The horizontal axis ranks
solution classes by regularity, with Onsager threshold $\alpha=1/3$
and Lipschitz uniqueness threshold $\alpha=1$ marked by dashed lines.
The bands stack the scheme's behaviour against this landscape:
(II) subsequential convergence to Leray--Hopf solutions in the
viscous case, full-sequence convergence with rates when a
strong solution exists; (III) subsequential convergence to conservative
measure-valued Euler solutions in the inviscid case, with the
concentration defect vanishing at the Onsager threshold; (I)
smooth-regime convergence for $C^1\cap W^{r_{\rm rec},\infty}$
solutions; and (IV), the bottom red band, the structural exclusion of
energy-dissipating Euler limits across the regularity axis.}
\label{fig:landscape}
\end{figure}

\newpage
\medskip\noindent\textbf{Notation.}
The following symbols are used throughout.

\smallskip
\begin{center}
\small
\begin{tabular}{@{}lll@{}}
  \toprule
  Symbol & Meaning & Defined in \\
  \midrule
  $\KK$, $\KKs$ & primal / dual cell complex & \Cref{subsect:grid}\\
  $\bD_k$, $\tD_k$ & primal / dual exterior derivative (coboundary) & \Cref{def:ext-deriv}\\
  $\bM_k$ & diagonal Hodge star ($k$-forms) & \Cref{def:M1}\\
  $\Iv(\cdot)$ & extrusion (discrete interior product) & \Cref{def:contraction}\\
  $\bQ(\cdot,\cdot)$ & $\bQ(\bv,\bv) = \Iv(\tD_1\bv)$, the discrete Lamb vector & \Cref{def:contraction}\\
  $\mathcal{R}_h$ & de~Rham map (continuous $\to$ cochain) & \Cref{def:deRham}\\
  $\mathcal{W}_h$ & Whitney map (cochain $\to$ continuous) & \Cref{def:Whitney}\\
  $\bP_h$ & discrete Leray projector onto $V_h$ & eq.~\eqref{eq:Leray}\\
  $V_h$ & $\ker(\bD_2\bM_1)$, discrete div-free subspace & eq.~\eqref{eq:Vh}\\
  $\nrm{\cdot}_{L_h^2}$ & $\ip{\cdot}{\cdot}_1$-weighted discrete norm & \Cref{def:disc_norms}\\
  $\codiff$ & discrete codifferential, $\bM_1^{-1}\tD_1^T\bM_2$ & eq.~\eqref{eq:codiff}\\
  $\Deltah$ & discrete curl-curl Laplacian, $\codiff\,\tdd_1$ & eq.~\eqref{eq:Deltah}\\
  $\Ekin$ & discrete kinetic energy, $\tfrac{1}{2}\ip{\bv}{\bv}_1$ & eq.~\eqref{eq:energy_cons}\\
  \bottomrule
\end{tabular}
\end{center}

\section{Discrete Geometric Framework}\label{sect:functional_framework}

The central principle of Discrete Exterior Calculus is the separation
between topology and metric.
The exterior derivative is defined combinatorially, as the coboundary
operator of the cell complex; it commutes with the
de~Rham interpolation without approximation error.
All approximation error in the scheme is concentrated in the Hodge star, which
must compare integrals of a $k$-form over a primal cell with integrals of its
dual over the associated dual cell.

A Delaunay--Voronoi mesh is the natural setting for this construction:
the orthogonality turns the Hodge star into a diagonal matrix with entries $(\bM_k)_{ii} = |\sigma_i^*|/|\sigma_i|$, which is a first-order ($\OO(h)$) approximation on general meshes.
On centroidal Voronoi tessellations (CVTs) and on icosahedral or
cubed-sphere meshes widely used in geoscientific modelling, the
circumcentre of each Delaunay cell coincides, to higher order in $h$, with
the centroid of its dual Voronoi cell. This \emph{centroid proximity}
condition (\Cref{prop:centroid_proximity}) is the geometric
property that raises Hodge star accuracy from $\OO(h)$ to
$\OO(h^2)$ and with it the scheme's convergence rate.
The prismatic structure matches physical anisotropy.
Atmospheric and oceanic flows have anisotropic aspect ratios
($\Delta z / \Delta x \sim 10^{-3}$) and require vertical layer
resolution independent of horizontal mesh refinement. Prismatic
extrusion of a two-dimensional Delaunay--Voronoi layer, rather than a
fully three-dimensional tetrahedralisation, preserves this anisotropy
at the level of the discrete framework and keeps the horizontal
$d=2$ DEC structure intact.
This structure coincides with the spatial discretisation used by
operational atmosphere--ocean models (ICON~\cite{korn2017}, MPAS~\cite{ringler2010}),
so the mathematical theory developed here applies directly to those codes.

We reuse standard DEC objects and introduce
new constructions only when needed.
Standard objects reused as-is: the primal/dual coboundary operators
$\bD_k$, $\tD_k$; the diagonal Hodge stars $\bM_k$ on Delaunay--Voronoi
meshes (\cite{desbrun2005,Hirani}); the de~Rham map $\mathcal{R}_h$ and
Whitney map $\mathcal{W}_h$ (\cite{dodziuk1976,bossavit1998,arnold2006}); and
Bossavit's approximation bound for the diagonal Hodge star
(\Cref{lem:hodge_error}, cf.~\cite{bossavit2000generalized,hiptmair2001}).
Standard objects with a non-standard choice of reconstruction: the
extrusion-based velocity reconstruction $\bar\u_j(\bv)$
(\Cref{def:averaging_recon}) uses the reconstruction framework of
\cite{korn2017}, replacing Perot reconstructions.
New constructions specific to this
paper: the extrusion-based discrete contraction $\Iv$
(\Cref{def:contraction}), whose matrix form is chosen so that
antisymmetry $\ip{\bv}{\Iv(\tD_1\bv)}_1 = 0$ holds as an algebraic
identity (\Cref{prop:extrusion}, part~2) for any linear
reconstruction; the discrete Lie derivative $\tLie_\bv$ defined via
Cartan's magic formula from this contraction (\Cref{def:Lie}); and
the resulting discrete vorticity equation on dual 2-cochains.

\subsection{Mesh Geometry}\label{subsect:grid}

All discrete operators are defined in terms of the combinatorial structure
and geometric measures (lengths, areas, volumes) of the cell complex,
without reference to coordinates.
The domain $\Omega$ may be a closed oriented Riemannian manifold
of dimension $d = 2$ or~$3$, or a compact manifold with boundary.
Horizontal layers are Delaunay--Voronoi tessellations; prisms are formed by vertical extrusion.
We adopt the notation of DEC (see e.g.\ \cite{Desbrun, Hirani, Marsden});
operators on the dual complex carry a tilde
($\tdd_k$, $\tLie_\bv$); primal operators ($\dd_k$, $\bD_k$) do not.
Each primal $k$-cell has a unique dual $(3{-}k)$-cell:
\begin{center}
\begin{tabular}{cclcl}
  \toprule
  $k$ & Primal & Geometry & Dual & Geometry \\
  \midrule
  0 & $v_i\in\mathcal{V}$ & vertices & $K_i^*\in\KKs$ & Voronoi cells \\
  1 & $e_k\in\mathcal{E}$ & edges & $f_k^*\in\mathcal{F}^*$ & dual faces \\
  2 & $f_j\in\mathcal{F}$ & faces & $e_j^*\in\mathcal{E}^*$ & dual edges \\
  3 & $K_i\in\KK$ & prisms & $v_i^*\in\mathcal{V}^*$ & circumcentres \\
  \bottomrule
\end{tabular}
\end{center}
The Delaunay--Voronoi construction guarantees that every primal edge $e_k$ is orthogonal to its dual face $f_k^*$,
and that every primal face $f_j$ is orthogonal to its dual edge $e_j^*$.
We define the following geometric measures for length, area and volume:
\[
  \ell_k = |e_k|,\quad A_k^* = |f_k^*|,\quad
  A_j = |f_j|,\quad \ell_j^* = |e_j^*|,\quad
  |K_i|,\quad |K_m^*|.
\]

\begin{assumption}[Mesh regularity]\label{ass:mesh_reg}
The family of Delaunay--Voronoi meshes $\{\KK_h\}_{h>0}$ satisfies:
\begin{enumerate}[nosep]
\item \textit{Quasi-uniformity:} there exists $C_{\rm qu}>0$ such that $h/h_{\min}\le C_{\rm qu}$ for all $h$, where $h=\max_j\ell_j^*$ and $h_{\min}=\min_j\ell_j^*$.
\item \textit{Shape-regularity:} horizontal cells can be decomposed into triangles such that the ratio between inscribed circle radius and triangle diameter is bounded below by $\sigma_0>0$, uniformly in $h$.
\item \textit{Delaunay property:} the circumcentre of every primal cell lies inside the cell, ensuring orthogonality between primal edges and dual faces.
\item \textit{Bounded valence:} the number of cells sharing any vertex is uniformly bounded.
\end{enumerate}
These conditions guarantee that Hodge stars $\bM_k$ are uniformly equivalent to the identity,
and that discrete operators satisfy the approximation estimates needed for convergence.
\end{assumption}

The following optional mesh property is \emph{not} assumed in general
but will be invoked  \emph{additionally} where it yields improved rates.
\begin{property}[Centroid proximity]\label{prop:centroid_proximity}%
For each degree $k = 0, 1, 2$, let $\sigma$ denote a primal
$k$-cell and $\sigma^*$ its dual $(d{-}k)$-cell.
The Hodge star $(\bM_k)_{\sigma\sigma} = |\sigma^*|/|\sigma|$
compares the integral of a $k$-form over $\sigma$
with the integral of its Hodge dual over
$\sigma^*$.  The approximation error depends on the offset between
the midpoint (or centroid) of each primal and dual cell.
We require:
\begin{equation}\label{eq:centroid_proximity}
  \max_{j}\,\bigl|\bar{\mathbf{x}}_{e_j^*} - \bar{\mathbf{x}}_{f_j}\bigr|
  + \max_{k}\,\bigl|\bar{\mathbf{x}}_{f_k^*} - \mathbf{x}_{e_k}\bigr|
  \le C_{\rm cg}\,h^2,
\end{equation}
where, for primal face $f_j$:
$\bar{\mathbf{x}}_{f_j}$ is the centroid of $f_j$ and
$\bar{\mathbf{x}}_{e_j^*}$ is the midpoint of the dual edge $e_j^*$;
and for primal edge $e_k$:
$\mathbf{x}_{e_k}$ is the midpoint of $e_k$ and
$\bar{\mathbf{x}}_{f_k^*}$ is the centroid of the dual face $f_k^*$.
The first term controls the $k{=}1$ Hodge;
the second the $k{=}2$ Hodge.

On a prismatic mesh, condition~\eqref{eq:centroid_proximity} is
satisfied when:
\begin{enumerate}[label=(\alph*),nosep]
\item the horizontal tessellation is a centroidal Voronoi
  tessellation, including spherical meshes
  \cite{ringler2010}; this ensures $|\bar{\mathbf{x}}_{f_j}^{\rm horiz}  - x_j^*| = 0$ for horizontal faces
  and $|\bar{\mathbf{x}}_{f_k^*} - \mathbf{x}_{e_k}| = \OO(h^2)$
\item the vertical layer interfaces are $C^2$-smooth functions
  of horizontal position -- this ensures the vertical component
  of $|\bar{\mathbf{x}}_{e_j^*} - \bar{\mathbf{x}}_{f_j}|$ is
  $\OO(h^2)$%
\end{enumerate}
When \Cref{prop:centroid_proximity} holds, the diagonal Hodge*
achieves second-order accuracy (\cref{lem:hodge_error}),
improving several convergence rates from $\OO(h)$ to $\OO(h^2)$.
\end{property}

\noindent
Throughout the paper, $r_\star$ denotes the \emph{Hodge accuracy exponent}:
\begin{equation}\label{eq:rstar_def}
  r_\star := 1 \text{ under \Cref{ass:mesh_reg} only,} \ 
  \text{ or }\  r_\star := 2 \text{ under \Cref{ass:mesh_reg} and \Cref{prop:centroid_proximity}.} 
\end{equation}
\subsection{Approximation Spaces}\label{subsection_ApproxSpaces}
The degrees of freedom mirror the de~Rham complex: velocity is a dual
1-cochain (one scalar value per dual edge), vorticity a dual 2-cochain
(one value per dual face), and pressure a primal 3-cochain (one value per
primal prism).
This mimetic correspondence ensures that the exterior derivative acting on
the velocity cochain gives the vorticity cochain, without
interpolation error.
The following norms measure these cochains in a way that respects the
geometric weighting by cell volumes.
\begin{definition}[Discrete norms]
\label{def:disc_norms}
For a dual 1-cochain $\bv\in C^1(\KKs)$ we introduce the following norms.
\begin{itemize}[nosep]
\item \textit{$\ell^2$-norm}:
$\nrm{\bv}_{\ell^2}^2 := \sum_j v_j^2$.

\item \textit{$L^2$-type norm}:
$\nrm{\bv}_{L_h^2}^2 := \ip{\bv}{\bv}_1 = \sum_j(\bM_1)_{jj}v_j^2$.

\item \textit{$H^1$-type norm}:
$\nrm{\bv}_{H_h^1}^2 := \nrm{\bv}_{L_h^2}^2 + \nrm{\tD_1\bv}_{L_h^2}^2$.

\item \textit{$L^\infty$-type norm} ($L^2$-Lebesgue-inclusion convention):
$\nrm{\bv}_{L_h^\infty} := \max_j |v_j|/\sqrt{(\bM_1)_{jj}}$.

\item \textit{Pointwise-reconstruction norm} (natural norm
for pointwise reconstructed velocities):
\begin{equation}\label{eq:grad_norm_def}
  \nrm{\bv}_{\rm rec} := \max_j |v_j|/\ell_j^*.
\end{equation}
\end{itemize}
The two infinity-type norms are not equivalent: on a quasi-uniform
mesh $\nrm{\bv}_{\rm rec} = \OO(h^{(d-4)/2})\,\nrm{\bv}_{L_h^\infty}$,
because $\sqrt{(\bM_1)_{jj}} = \OO(h^{(d-2)/2})$ while
$\ell_j^* = \OO(h)$. The averaging
reconstruction~(\Cref{def:averaging_recon}) and the extrusion
weights~(\Cref{def:contraction}) are continuous in the pointwise-%
reconstruction norm with mesh-regularity-only constants:
$|\bar\u_j(\bv)|\le C_R\,\nrm{\bv}_{\rm rec}$
and $|w_{jk}(\bv)|\le C_w\,\nrm{\bv}_{\rm rec}$. The
$L_h^\infty$-norm above relates to $L_h^2$ by
$\nrm{\bv}_{L_h^2}\le |\Omega|^{1/2}\,\nrm{\bv}_{L_h^\infty}$ but
does not bound reconstructed velocities pointwise with an
$h$-independent constant.
\end{definition}

The transition from the continuous to the discrete world is carried out by means of the de~Rham mapping
\begin{definition}[de~Rham map]
\label{def:deRham}
The {de~Rham map}
$\mathcal{R}_h:C^\infty(\Omega;\Lambda^k)\to C^k(\KKs)$
maps a smooth differential $k$-form to a $k$-cochain by integration.
\begin{enumerate}[label=(\roman*), nosep]
\item For continuous velocity (a 1-form) the discrete representation is given by
\begin{equation*}
  (\mathcal{R}_h\bu^\flat)_j = \int_{e_j^*}\bu\cdot d\ell
  \qquad\forall\;\text{dual edges}\;e_j^*.
\end{equation*}
\item For continuous vorticity (a 2-form) we have
\begin{equation*}
(\mathcal{R}_h\omega^\flat)_k = \int_{f_k^*}\bom\cdot d{A}
\qquad\forall\;\text{dual faces}\;f_k^*.
\end{equation*}
\item For continuous pressure (a 3-form) we have 
\begin{equation*}
(\mathcal{R}_h\ p)_K = \int_{K} p\, d{V}
\qquad\forall\; \text{ primal prisms } K.
\end{equation*}
\end{enumerate}
\end{definition}
\begin{definition}[Primal and dual $k$-forms]
Let $\sigma^k\in\KK$ denote a primal $k$-cell and $\tilde\sigma^k\in\KKs$ a dual $k$-cell.
\begin{enumerate}[label=(\roman*), nosep]
\item A \emph{primal $k$-form} $\alpha^k$ assigns to each primal $k$-cell
$\sigma^k$ the integral $\alpha^k(\sigma^k) = \int_{\sigma^k}\alpha$.
\item A \emph{dual $k$-form}
$\tilde\alpha^k$ assigns to each dual $k$-cell
$\tilde\sigma^k$ the integral $\tilde\alpha^k(\tilde\sigma^k) = \int_{\tilde\sigma^k}\tilde\alpha$.
\end{enumerate}
The space of primal $k$-forms is denoted $C^k(\KK)$; the space of dual $k$-forms
is denoted $C^k(\KKs)$.
\end{definition}
The staggering is summarised in the following table:
\begin{center}
\begin{tabular}{llll}
  \toprule
  Variable & Space & Lives on & Count \\
  \midrule
  $\v(t)$ (velocity) & $C^1(\KKs)$ &
    dual edges $e_j^*$ & $|\mathcal{F}|$ \\
  $\bom(t)$ (vorticity) & $C^2(\KKs)$ &
    dual faces $f_k^*$ & $|\mathcal{E}|$ \\
  $p(t)$ (pressure) & $C^0(\KKs)$ &
    dual vertices $v_i^*$ = {primal cell centres} & $|\KK|$ \\
  \bottomrule
\end{tabular}
\end{center}
\medskip
The transition from discrete to continuous is carried out by means of the Whitney reconstruction
\begin{definition}[Whitney reconstruction]
\label{def:Whitney}
The \textit{Whitney map}
$\mathcal{W}_h:C^k(\KKs)\to L^2(\Omega;\Lambda^k)$
reconstructs a continuous $k$-form from a $k$-cochain.
For a dual 1-cochain $\bv$, the Whitney reconstruction
$\mathcal{W}_h\bv$ is defined as follows:
on a prismatic Delaunay--Voronoi mesh, each prism $K_i$ is decomposed into
simplices (tetrahedra) by connecting the circumcentre
$v_i^* = x_i^*$ to all faces of $K_i$.  This yields a
simplicial refinement $\KK_h^{\rm simp}$ of $\KK_h$ that
inherits the Delaunay--Voronoi duality: each dual edge $e_j^*$
lies along a simplex edge, and each dual face $f_k^*$ is
a union of simplex faces.  On such a refinement,
the classical Whitney 1-form
\cite{whitney1957,dodziuk1976} is well defined:
for each dual edge $e_j^*$ with endpoints $v_a^*,v_b^*$,
the Whitney basis form is
$\lambda_{a}\,d\lambda_{b} - \lambda_{b}\,d\lambda_{a}$,
where $\lambda_{a},\lambda_{b}$ are the barycentric
coordinates of the simplex containing $e_j^*$.
The reconstruction is
\begin{equation*}%
  \mathcal{W}_h\bv = \sum_j v_j\,W_j,
  \qquad W_j = \lambda_{a_j}\,d\lambda_{b_j} - \lambda_{b_j}\,d\lambda_{a_j},
\end{equation*}
where $a_j, b_j$ indicate the endpoints of the $j$-th dual edge.
This satisfies the interpolation property
\begin{equation*}
  \int_{e_j^*}(\mathcal{W}_h\bv) = v_j \quad\forall\;j.
\end{equation*}
\end{definition}
The simplicial refinement 
serves to provide the barycentric coordinates needed
for the Whitney basis forms.
Classical Whitney approximation theory applied to the
simplicial refinement $\KK_h^{\rm simp}$
\cite{dodziuk1976,arnold2006} gives, for a smooth 1-form $\alpha$,
\begin{equation}\label{eq:Whitney_approx}
  \nrm{\alpha - \mathcal{W}_h\mathcal{R}_h\alpha}_{L^2}
  \le C\,h\nrm{\alpha}_{H^1},
\end{equation}
where the refinement inherits the shape-regularity of $\KK_h$
(\cref{ass:mesh_reg}).
The Whitney map is norm-bounded, i.e.,\ there exists $C_W > 0$ depending
on mesh regularity such that
$\nrm{\mathcal{W}_h\bv}_{L^2} \le C_W\nrm{\bv}_{L_h^2}$ for all
$\bv\in C^1(\KKs)$.  This follows from the finite overlap of the
Whitney basis forms and the norm equivalence on shape-regular
simplices (see \cite{arnold2006}, Thm.~5.6).
Because $\mathcal{W}_h$ is built from classical Whitney forms on the
shape-regular simplicial refinement $\KK_h^{\rm simp}$, every Whitney
consistency estimate invoked in the convergence
proofs---\eqref{eq:Whitney_approx}, the norm bound above, and the
$L^2$- and enstrophy-equivalences
$\nrm{\mathcal{W}_h\cdot}_{L^2}\sim\nrm{\cdot}_{L_h^2}$,
$\nrm{d\mathcal{W}_h\cdot}_{L^2}\sim\nrm{\tD_k\cdot}_{\bM_{k+1}}$ of
\ref{H:Whitney_interp}---is a classical result on a simplicial
complex, applied to $\KK_h^{\rm simp}$ rather than to the polygonal
prisms directly; the $\OO(h^{r_\star})$ discrepancies are the
Hodge-star quadrature errors between the prismatic and refined
metrics.
\subsection{Differential and Other Operators}\label{subsection_Operators}

The DEC discretisation rests on two operators.
The exterior derivative $\bD_k$ is defined combinatorially via
Stokes' theorem on the cell complex and is therefore exact: it
commutes with the de~Rham interpolation without approximation error.
By contrast, the Hodge star $\bM_k$ must compare integrals of a
$k$-form over a primal cell with integrals of its Hodge dual over the
complementary dual cell, and all approximation error in the scheme is
concentrated here.
The consistency of the Hodge star, its accuracy as a function of mesh
parameters, is what governs the convergence rate throughout the paper.
The following definitions make this precise.
\begin{definition}[Exterior derivatives]\label{def:ext-deriv}
$i)$ the primal exterior derivative is defined by
$$\dd_k : C^k(\KK)\to C^{k+1}(\KK)$$ 
\[\text{with }\quad
  (\dd_k\,\alpha)(\sigma^{k+1})
  = \alpha(\partial\,\sigma^{k+1})
  = \sum_{\sigma^k\prec\sigma^{k+1}}[\sigma^{k+1}:\sigma^k]\;\alpha(\sigma^k),
\]
where $\sigma^k\prec\sigma^{k+1}$ means that $\sigma^k$ belongs to the boundary of $\sigma^{k+1}$.
The primal exterior derivative is in matrix form denoted by $D_k$.\\
$ii)$ the dual exterior derivative is defined by
$\tdd_k : C^k(\KKs)\to C^{k+1}(\KKs)$ with matrix representation $\tD_k$.
\end{definition}

A central property is that the de~Rham map commutes with the exterior derivative,
\begin{equation}\label{eq:deRham_commutativity}
\tD_1\mathcal{R}_h\bu^\flat = \mathcal{R}_h(\dd\bu^\flat)
  = \mathcal{R}_h\omega^\flat
  \qquad(\text{by Stokes' theorem}).
\end{equation}
\phantomsection\label{prop:derham_commutativity}%
This identity is the reason why the discrete vorticity $\tD_1\bv$ is the
de~Rham image of the continuous vorticity, not merely an approximation of it.

The following estimates quantify the approximation error of the de~Rham
interpolant and of the extrusion operator.
These are the only truncation bounds needed in the convergence proofs;
all other operators, such as the exterior derivative and the discrete
pressure gradient, are exact.
\begin{lemma}[Interpolation error estimates]
\label{lem:interp_error}%
Let $\bu\in W^{r_\star,\infty}(\Omega)$
be a smooth divergence-free velocity field, and let
$\bv_{\mathrm{exact}} := \bP_h\mathcal{R}_h\bu^\flat$
denote the Leray-projection of the de~Rham interpolant.
Then the projection residual satisfies
\begin{equation}\label{eq:interp_error_v}
  \nrm{\mathcal{R}_h\bu^\flat - \bv_{\mathrm{exact}}}_{L_h^2}
  \le C_{\mathrm{int}}\,h^{r_\star}\nrm{\bu}_{W^{r_\star,\infty}},
\end{equation}
where $C_{\mathrm{int}}$ depends only on the Hodge* constant $C_\star$
and the mesh regularity constants from \Cref{ass:mesh_reg}.
The rate is determined by the Hodge* approximation error
(\cref{lem:hodge_error}); no discrete Poincar\'e constant enters
(see the proof in \Cref{app:approximation}).
The individual operators satisfy the following estimates:

 \begin{enumerate}[label=(\roman*), nosep]
\item For any smooth function $p$, the discrete gradient is exact:
\[
\nrm{\tD_0\mathcal{R}_h p - \mathcal{R}_h(\dd p)}_{L_h^2}
  = 0.
\]

\item For any smooth 1-form $\bu^\flat$, the discrete curl is exact:
\[
\nrm{\tD_1\mathcal{R}_h\bu^\flat - \mathcal{R}_h(\dd\bu^\flat)}_{L_h^2}
  = 0.
\]

\item For smooth $\bu\in W^{r_\star,\infty}(\Omega)$ with $\nabla\cdot\bu = 0$,
the discrete divergence is approximately consistent:
\begin{equation}\label{eq:div_consistency}
  \nrm{\bD_2\bM_1\mathcal{R}_h\bu^\flat}_{\ell^2}
  \le C\,h^{r_\star}\,\nrm{\bu}_{W^{r_\star,\infty}},
\end{equation}
where the $\OO(h^{r_\star})$ residual arises from the Hodge* approximation
(\cref{lem:hodge_error}).
\end{enumerate}
\end{lemma}
\begin{proof}
The proof is given in \Cref{app:approximation}.
\end{proof}

\paragraph*{Hodge*.} Because every primal $k$-cell is orthogonal to its dual $(3{-}k)$-cell, all Hodge* operators are diagonal and positive definite.
\begin{definition}[Hodge*-Operator]\label{def:hodge}
i) The primal-to-dual Hodge* is defined by
\[
\star_k : C^k(\KK)\to C^{3-k}(\KKs),\text{ with matrix entries }   (\star_k)_{ii} := \frac{|\sigma_i^{*(3-k)}|}{|\sigma_i^k|}
\]
ii) The inverse primal-to-dual Hodge* is given by 
\[
\star_k^{-1} : C^{3-k}(\KKs)\to C^k(\KK), \text{ with matrix entries }
  (\star_k^{-1})_{ii} := \frac{|\sigma_i^k|}{|\sigma_i^{*(3-k)}|}.
\]
\end{definition}

The next lemma estimates the error from approximating the continuous Hodge* by
the diagonal Delaunay--Voronoi Hodge*.
\begin{lemma}[Hodge* approximation error]
\label{lem:hodge_error}%
Let $\alpha$ be a smooth $k$-form, $\sigma^k$ a primal $k$-cell with
dual $(d{-}k)$-cell $(\sigma^k)^*$, and let
$\Phi_\sigma := \int_{(\sigma^k)^*}\star\alpha$ be the Hodge dual flux.
Then:
\begin{equation}\begin{split}
 &\bigl|(\bM_k)_{\sigma\sigma}\,(\mathcal{R}_h\alpha)_\sigma - \Phi_\sigma\bigr|
  \le C_\star\,h^{r_\star}\,\nrm{\alpha}_{W^{r_\star,\infty}}\,|\sigma^k|, \\
 &\nrm{\bM_k\mathcal{R}_h\alpha - \boldsymbol\Phi}_{L_{h,k}^2}
  \le C_\star\,h^{r_\star}\,\nrm{\alpha}_{W^{r_\star,\infty}},
\end{split}\end{equation}
where $r_\star = 1$ under \Cref{ass:mesh_reg} only,
$r_\star = 2$ under \Cref{prop:centroid_proximity},
$\nrm{\mathbf{w}}_{L_{h,k}^2}^2 := \sum_\sigma(\bM_k^{-1})_{\sigma\sigma}w_\sigma^2$
is the $(\bM_k^{-1})$-weighted norm on primal $k$-cochains,
and $C_\star$ depends only on mesh regularity.
\end{lemma}
\begin{proof}
The proof is given in \Cref{app:approximation}.
\end{proof}

The diagonal star is moreover rigid: exactness on constants
characterises both the weights and the orthogonality of the mesh pair.

\begin{proposition}[Rigidity of diagonal Hodge star]\label{prop:hodge_rigidity}
Let each primal face $f_j$ with unit normal $\hat n_j$ be paired with
a transversal dual edge $e_j^*$ with unit direction $\hat e_j$
(oriented as in \Cref{subsect:grid}), and let
$D : C^1(\KKs)\to C^2(\KK)$ be a diagonal operator,
$(D\bv)_j = m_j v_j$ with $m_j > 0$. Identifying a neighbourhood of
$f_j\cup e_j^*$ with $\mathbb{R}^d$, the exactness-on-constants
property
\[
  (D\,\mathcal{R}_h\bu^\flat)_j \;=\; \int_{f_j}\star\,\bu^\flat
  \qquad\text{for every constant field }\bu\in\mathbb{R}^d
  \text{ and every }j
\]
holds if and only if the pairing is orthogonal, $\hat e_j = \hat n_j$,
and $m_j = A_j/\ell_j^*$; that is, $D = \bM_1$. In particular:
$(i)$~on a Delaunay--Voronoi pair the DEC Hodge star is the unique
diagonal star exact on constants; $(ii)$~on a non-orthogonal pairing
no diagonal star is exact on constants; exactness can be restored
only by a full matrix.
\end{proposition}
\begin{proof}
For constant $\bu$ one has
$(\mathcal{R}_h\bu^\flat)_j = \ell_j^*\,(\bu\cdot\hat e_j)$ and
$\int_{f_j}\star\,\bu^\flat = A_j\,(\bu\cdot\hat n_j)$. Exactness at
$j$ for all $\bu\in\mathbb{R}^d$ is the equality of the linear forms
$m_j\ell_j^*\,\hat e_j = A_j\,\hat n_j$. Taking norms gives
$m_j = A_j/\ell_j^*$; dividing by the norms gives
$\hat e_j = \hat n_j$. The converse is immediate. For
$\hat e_j \ne \hat n_j$ the two forms are not proportional for any
$m_j > 0$, which proves $(ii)$; the existence of full consistent local
matrices on general pairings, unique up to the stabilisation family.
(see e.g.~\cite{BrezziLipnikovShashkov2005,beiraoVeiga2013basic}).
\end{proof}

\begin{lemma}[Hodge* bilinear form approximation]\label{lem:hodge_bilinear}%
For $k = 0, 1, 2$, let $\alpha$ be a smooth $(d-k)$-form, $\mathcal{R}_h\alpha$ its
de~Rham interpolant (a $(d-k)$-cochain on primal $(d-k)$-cells, equivalently
a $k$-cochain on the dual complex via the standard primal-dual correspondence),
and $b\in C^k(\KKs)$ an arbitrary dual $k$-cochain.
Then the diagonal Hodge* $\bM_k$ approximates the dual-cell pairing to
$\OO(h^{r_\star})$:
\[
\Bigl|\ip{b}{\mathcal{R}_h\alpha}_k
       - \textstyle\sum_{\sigma}\, b_\sigma\!\int_{\sigma^{*(d-k)}}\!\star\alpha\Bigr|
  \le C_k\,h^{r_\star}\,\nrm{\alpha}_{W^{r_\star,\infty}}\,\nrm{b}_{\bM_k},
\]
where $\sigma$ runs over primal $(d-k)$-cells (equivalently, dual $k$-cells)
and $\sigma^{*(d-k)}$ denotes the corresponding dual $k$-cell on which $b$ is
indexed; the constant $C_k$ depends on the mesh regularity.
\end{lemma}
\begin{proof}
By \Cref{lem:hodge_error} (covering all $k=0,1,2$), the pointwise error
satisfies
\[
  \bigl|(\bM_k)_{\sigma\sigma}\,(\mathcal{R}_h\alpha)_\sigma
    - \textstyle\int_{(\sigma^k)^*}\!\star\alpha\bigr|
  \le C\,h^{r_\star}\,\nrm{\alpha}_{W^{r_\star,\infty}}\,|\sigma^k|.
\]
Summing over all $k$-cells and applying Cauchy--Schwarz yields
\[
  \Bigl|b^T\bM_k\,\mathcal{R}_h\alpha
    - \textstyle\sum_\sigma b_\sigma\!\int_{(\sigma^k)^*}\!\star\alpha\Bigr|
  \le C\,h^{r_\star}\,\nrm{\alpha}_{W^{r_\star,\infty}}
      \textstyle\sum_\sigma|b_\sigma|\,|\sigma^k|
  \le C\,h^{r_\star}\,\nrm{\alpha}_{W^{r_\star,\infty}}\,\nrm{b}_{\bM_k},
\]
where the last step uses $\sum_\sigma|b_\sigma||\sigma^k|
\le (\sum_\sigma(M_k)_{\sigma\sigma}b_\sigma^2)^{1/2}
(\sum_\sigma |\sigma^k|^2/(M_k)_{\sigma\sigma})^{1/2}
= \nrm{b}_{\bM_k}\OO(1)$
by Cauchy--Schwarz and quasi-uniformity.
\end{proof}

\begin{definition}[Inner Products]\label{def:M1}
$i)$ The inner product on dual 1-forms is defined by
\[
  \ip{{v}}{{w}}_1 := {\v}^TM_1\,{\w},
  \qquad
  M_1 := \star_2^{-1},
  \qquad
  (M_1)_{jj} = \frac{A_j}{\ell_j^*},
\]
representing the discrete inner product $\int v\wedge\star w$.\\
$ii)$ The inner product on dual 0-forms is defined by
\[
  \ip{B_1}{B_2}_0 := B_1^T\,M_0\,B_2,
  \qquad
  M_0 := \star_3^{-1},
  \qquad
  (M_0)_{ii} = |K_i|.
\]
This represents the volume-weighted $L^2$-inner product at primal cell centres.\\
$iii)$ The inner product on dual 2-forms is defined by%
\[
  \ip{\omega_1}{\omega_2}_2 := \omega_1^T\,\bM_2\,\omega_2,
  \qquad
  \bM_2 := \star_1^{-1}.
\]
\end{definition}

\paragraph*{Extrusion and Contraction.}  The interior product $\iota_{{\u}}\omega$ is discretised via
{extrusion}: the contraction on a dual 1-cell $e_j^*$ equals the flux
of $\omega$ through the parallelogram swept by extruding $e_j^*$
infinitesimally along the velocity field.
\begin{definition}[Discrete contraction]\label{def:contraction}
The discrete contraction
$\Iv:C^2(\KKs)\to C^1(\KKs)$ is defined for any dual 1-cochain $\bw$ by
\begin{equation}\label{eq:contraction}
\ip{\bw}{\Iv(\bom)}_1
    := \tfrac{1}{2}\bigl(  \ip{\tU\,\bom}{\bw}   -  \ip{\tU^T\bv}{\tD_1\bw}  \bigr),
 \end{equation}
or equivalently in matrix form
\[
    \bM_1\,\Iv(\bom)
    = \tfrac{1}{2}\bigl(\tU\,\bom - \tD_1^T\,\tU^T\bv\bigr),
\]
where the velocity-weighted incidence matrix $\tU$ has the
same sparsity as $\tD_1$ and entries
\begin{equation}\label{eq:Utilde_def}
  \tU_{jk}
  := D_{1,jk}\;\bar{\u}_j(\bv)\cdot\hat{e}_k.
\end{equation}
where $\bar{\u}_j(\bv)\in\mathbb{R}^3$ is a reconstructed
velocity vector at primal face $f_j$, {linear} in
the cochain $\bv$.  
We also write
\begin{equation}\label{eq:Q_diagonal}
\bQ(\bv,\bv):=\Iv(\tD_1\bv).
\end{equation}
The form $\bQ$ admits a unique extension to a bilinear map
$\bQ:C^1(\KKs)\times C^1(\KKs)\to C^1(\KKs)$ defined by the matrix formula
\begin{equation}\label{eq:Q_bilinear}
  \bM_1\,\bQ(\bv_1,\bv_2)
  := \tfrac{1}{2}\bigl(\tU(\bv_1)\,\tD_1\bv_2
                       \;-\;\tD_1^T\,\tU(\bv_1)^T\bv_2\bigr).
\end{equation}
This extension agrees with~\eqref{eq:Q_diagonal} when $\bv_1=\bv_2=\bv$
(both terms then equal $\bv^T\tU(\bv)\tD_1\bv$), and it is the
extension used throughout the convergence analysis. Note that
$\bQ(\bv_1,\bv_2)\ne\Iv[\bv_1](\tD_1\bv_2)$ in general off-diagonal:
the second-term factor in the matrix form of $\Iv[\bv_1](\tD_1\bv_2)$ is
$\bv_1$, not $\bv_2$, so $\Iv$ is not jointly bilinear in
$(\bv_1,\bv_2)$ whereas the bilinear $\bQ$ above is, by construction.
With this extension,
$\bQ(\bv^h,\bv^h)-\bQ(\bar\bv,\bar\bv) = \bQ(\bar\bv,e)+\bQ(e,\bar\bv)+\bQ(e,e)$
holds as an algebraic identity for $e:=\bv^h-\bar\bv$.
\end{definition}
\smallskip
The contraction~\eqref{eq:contraction} requires a velocity
reconstruction $\bar{\u}_j(\bv)$ at each primal face~$f_j$,
linear in the cochain~$\bv$.
The algebraic conservation properties such as energy identity, Lamb
antisymmetry, Kelvin circulation depend only on this linearity
and on the matrix structure~\eqref{eq:contraction};
they hold for any linear reconstruction.

\begin{definition}[Averaging reconstruction]\label{def:averaging_recon}
At each dual vertex $v_i^*$ (circumcentre of prism~$K_i$),
define the Gram matrix
\[
G_i := \sum_{e_n^*\prec K_i^*} \hat{t}_n\otimes\hat{t}_n,
\]
where the sum runs over the dual edges $e_n^*$ emanating
from~$v_i^*$, and $\hat{t}_n$ is the unit tangent to~$e_n^*$.
On a shape-regular mesh, $G_i$ is symmetric positive definite
with condition number bounded by the mesh regularity constant.

The \emph{averaging reconstruction} at~$v_i^*$ is
\[
\u(v_i^*)
  := G_i^{-1}\sum_{e_n^*\prec K_i^*}
    \frac{\v(e_n^*)}{\ell_n^*}\;\hat{t}_n.
\]
The face velocity is obtained by averaging over the two
adjacent cells:
$\bar{\u}_j = \tfrac{1}{2}(\u(v_a^*) + \u(v_b^*))$.
\end{definition}

\begin{proposition}[Reconstruction accuracy]\label{prop:recon_accuracy}%
Let $\bu\in W^{s,\infty}(\Omega)$ be a smooth velocity field
with de~Rham interpolant $\bv = \mathcal{R}_h(\bu^\flat)$.
Then:

\noindent\textup{(i)}
The averaging reconstruction reproduces constants and is
therefore at least first-order accurate:
\[
|\u(v_i^*) - \bu(v_i^*)| \le C_R\,h\,\nrm{\bu}_{W^{1,\infty}},
\]
with $C_R$ depending only on mesh regularity.

\noindent\textup{(ii)}
On meshes satisfying the \emph{reconstruction symmetry}
condition
\begin{equation}\label{eq:recon_symmetry}
  \sum_{e_n^*\prec K_i^*}
  \ell_n^*\,(\hat{t}_n)_j\,(\hat{t}_n)_k\,(\hat{t}_n)_l = 0
  \qquad\text{for all }j,k,l,
\end{equation}
the reconstruction reproduces linear polynomials and
hence is second-order accurate:
\[
|\u(v_i^*) - \bu(v_i^*)|
  \le C_R'\,h^2\,\nrm{\bu}_{W^{2,\infty}}.
\]
\end{proposition}

\begin{proof}
\textit{(i)} The statement for constants is obvious, the first-order bound follows by Taylor expansion.
\textit{(ii)}  The leading error from~(i) is
$\delta\u = G_i^{-1}\sum_n [(\nabla\bu)^T\hat{t}_n \cdot
(x_n^{\rm mid} - v_i^*)]\,\hat{t}_n$.
Since $x_n^{\rm mid} - v_i^* = \frac{\ell_n^*}{2}\hat{t}_n$
on a Delaunay--Voronoi mesh, the $l$-th component involves the third moment
$T_{jkl}: = \sum_n\ell_n^*(\hat{t}_n)_j(\hat{t}_n)_k(\hat{t}_n)_l$.
Under~\eqref{eq:recon_symmetry}, $T_{jkl} = 0$, so
$\delta\u = \OO(h^2)$.
\end{proof}

We denote by $r_{\rm rec}$ the \emph{reconstruction accuracy exponent}:
\begin{equation}\label{eq:r_rec_def}
  r_{\rm rec} := 1 \text{ under \Cref{ass:mesh_reg} only,} \qquad
  r_{\rm rec} := 2 \text{ under \Cref{ass:mesh_reg} and reconstruction symmetry~\eqref{eq:recon_symmetry}.}
\end{equation}

\begin{convention}[Two mesh cases]\label{conv:cases}
Throughout the paper, all convergence rates are stated for the
following two mesh classes:
\begin{itemize}[nosep]
  \item[\textbf{(A)}] general Delaunay--Voronoi meshes
    (\Cref{ass:mesh_reg} only): $r_{\rm rec} = r_\star = 1$;
  \item[\textbf{(B)}] Delaunay--Voronoi meshes with centroid proximity
    (\Cref{prop:centroid_proximity}) and reconstruction
    symmetry~\eqref{eq:recon_symmetry}: $r_{\rm rec} = r_\star = 2$.
\end{itemize}
Other configurations (e.g.\ centroid proximity without reconstruction
symmetry, or vice versa) are not considered. Under (A) and (B) the
convergence rate is $h^{r_\star}|\log h|^{\beta_d}$ for $d = 2, 3$:
first order in case (A), second order in case (B).
\end{convention}


\paragraph*{Wedge Product.} The discrete wedge product is defined
via the Cartan identity
$\iota_{{\u}}(\alpha\wedge\beta)
= (\iota_{{\u}}\alpha)\wedge\beta
 + (-1)^p\,\alpha\wedge(\iota_{{\u}}\beta)$,
ensuring algebraic compatibility with the contraction and exterior derivative.

\begin{definition}[Discrete wedge product]\label{def:wedge}
The {extrusion-based discrete wedge product}
\[
  \twdg\;:\;
  C^p(\KKs)\times C^q(\KKs)
  \;\longrightarrow\;
  C^{p+q}(\KKs),
  \qquad p+q\le 3,
\]
is defined as follows:

 \begin{enumerate}[label=(\roman*), nosep]
\item  \emph{Wedge between 1-forms.}
For dual 1-forms $\tilde\alpha,\tilde\beta\in C^1(\KKs)$,
the wedge product is the dual 2-cochain defined via the contraction as
\begin{equation}\label{eq:wedge_11}
  \bigl(\tilde\alpha\twdg\tilde\beta\bigr)(f_k^*):
  = \sum_{e_j^*\prec f_k^*}
    w_{jk}(\tilde\alpha)\;\tilde\beta(e_j^*)
  = \bigl(I_{\tilde\alpha}\,\tilde\beta\bigr)(f_k^*),
\end{equation}
where $I_{\tilde\alpha}$ denotes the contraction
(\cref{def:contraction}).
Equivalently, from the contraction form~\eqref{eq:contraction},
$w_{jk}(\tilde\alpha) = \frac{1}{2}(\bM_1^{-1}\tU_{\tilde\alpha})_{jk}$,
so $w_{jk}$ is the extrusion weight obtained by
reconstructing a velocity field from $\tilde\alpha$,
projecting onto $\hat{e}_k$, and dividing by $(M_1)_{jj}$
to pass from $\bM_1\Iv$ to $\Iv$.
\item \emph{Wedge between a 1-form and a 2-form.}
For a dual 1-form $\tilde\alpha\in C^1(\KKs)$ and a dual 2-form
$\tilde\beta\in C^2(\KKs)$, the wedge product is defined as the
dual 3-cochain
\[
\bigl(\tilde\alpha\twdg\tilde\beta\bigr)(K_m^*):
  = \sum_{e_j^*\prec K_m^*}
    \tilde\alpha(e_j^*)\;\overline{\tilde\beta}_{j,m},
\]
where $\overline{\tilde\beta}_{j,m}$ is the average of
$\tilde\beta(f_k^*)$ over the dual 2-faces $f_k^*$ of $K_m^*$ that
share the dual edge $e_j^*$, weighted by the
fraction of $f_k^*$ lying inside $K_m^*$.
\end{enumerate}
\end{definition}

We define the discrete Lie derivative via the discrete Cartan formula.

\begin{definition}[Discrete Lie derivative]\label{def:Lie}
For a velocity dual 1-form $\bv\in C^1(\KKs)$ and
a dual $k$-form $\tilde\alpha^k\in C^k(\KKs)$, the
\emph{discrete Lie derivative} is defined by
$$\tLie_{\bv}\,\tilde\alpha^k
    := \tdd_{k-1}(\Iv\,\tilde\alpha^k)
    + \Iv(\tdd_k\,\tilde\alpha^k).$$
\end{definition}
Applied to $k=1$ (velocity): $\tLie_{\bv}\bv
  = \tdd_0\,\ekin + \Iv(\bom)$
with $\ekin = \Iv\bv$ and $\bom = \tdd_1\bv$.
Applied to $k=2$ (vorticity): $\tLie_{\bv}\,\tilde\alpha^2
  = \tdd_1(\Iv\,\tilde\alpha^2)
  + \Iv(\tdd_2\,\tilde\alpha^2)$.

All three operators -- wedge product, Lie derivative,
and Lamb vector -- are built from the single
extrusion-based contraction (\cref{def:contraction}).

\begin{proposition}\label{prop:extrusion}
The exterior derivative, the extrusion-based discrete wedge product $\widetilde\wedge$ and contraction
$I$ satisfy:
\begin{enumerate}%
\item  \textit{Cell complex: } The exterior derivative satisfies
$$  \dd_{k+1}\circ\dd_k = 0,\qquad
  \tdd_{k+1}\circ\tdd_k = 0.$$ 
  Equivalently in matrix form $D_{k+1}D_k = \mathbf{0}$ and
$\tD_{k+1}\tD_k = \mathbf{0}$.
  \item \textit{Energy identity:} For any dual 1-cochain
    $\bv\in C^1(\KKs)$ and $\bom = \tdd_1\bv$,
    \begin{equation}\label{eq:antisym_11}
      \ip{\bv}{\Iv(\bom)}_1 = 0.
    \end{equation}
   This identity holds for any linear velocity reconstruction.
  \item \textit{Approximate Leibniz rule (1-forms):} For smooth 1-forms
    $\alpha,\beta$, the de~Rham interpolants
    $\tilde\alpha=\mathcal{R}_h\alpha$, $\tilde\beta=\mathcal{R}_h\beta$
    satisfy, on each dual 3-cell $\sigma^*$,
    $$\bigl|\bigl(\tdd(\tilde\alpha\twdg\tilde\beta)
    - \tdd\tilde\alpha\twdg\tilde\beta
    - (-1)^p\tilde\alpha\twdg\tdd\tilde\beta\bigr)(\sigma^*)\bigr|
    \le C\,h^2\,\nrm{\alpha}_{W^{1,\infty}}\nrm{\beta}_{W^{1,\infty}},
    \qquad p=q=1,$$
    with $C$ depending only on the mesh-regularity constants; the defect
    arises from the metric dependence of the extrusion weights. This is
    the only Leibniz combination used below (\cref{rem:Perot_forms}).
  \item \textit{Discrete Cartan magic formula:}
    $\tLie_{\v}
    = \tdd\circ\Iv + \Iv\circ\tdd$.
\end{enumerate}
\end{proposition}
\begin{proof}
The proof is given in \Cref{app:basic}.
\end{proof}

\subsection{Discrete Functional-Analytic Structure}\label{subsect:cfem}

The operators of the preceding subsections assemble into a discrete
de~Rham complex with the same homological structure as the continuous
one. We record this structure, its closure into the primal--dual pair by
the metric, and the resulting identification of the velocity cochain with
the classical $H(\div)$ and $H(\curl)$ spaces; these are the
functional-analytic facts the convergence theory rests on.

\paragraph*{The commuting diagram.}
The dynamical variables live on the dual complex: pressure
$p\in C^0(\KKs)$, velocity $\bv\in C^1(\KKs)$, and vorticity
$\bom=\tD_1\bv\in C^2(\KKs)$. The de~Rham interpolation $\mathcal{R}_h$
(integration over cells, \Cref{def:deRham}) intertwines the continuous
exterior derivative with the coboundary, and because the relation is
Stokes' theorem the square commutes without approximation
error:
\[
\begin{array}{ccccccc}
\Lambda^0(\Omega) & \xrightarrow{\ \dd\ } & \Lambda^1(\Omega)
  & \xrightarrow{\ \dd\ } & \Lambda^2(\Omega)
  & \xrightarrow{\ \dd\ } & \Lambda^3(\Omega)\\[3pt]
\big\downarrow{\scriptstyle\,\mathcal{R}_h} & &
\big\downarrow{\scriptstyle\,\mathcal{R}_h} & &
\big\downarrow{\scriptstyle\,\mathcal{R}_h} & &
\big\downarrow{\scriptstyle\,\mathcal{R}_h}\\[3pt]
C^0(\KKs) & \xrightarrow{\ \tD_0\ } & C^1(\KKs)
  & \xrightarrow{\ \tD_1\ } & C^2(\KKs)
  & \xrightarrow{\ \tD_2\ } & C^3(\KKs)
\end{array}
\qquad \tD_k\,\mathcal{R}_h=\mathcal{R}_h\,\dd_k .
\]
The coboundary operators satisfy $\tD_{k+1}\tD_k=0$
(\Cref{def:ext-deriv}); this yields a discrete
Hodge--Helmholtz decomposition, on which the discrete Leray projector
$\bP_h$ is built. The Whitney map $\mathcal{W}_h$ (\Cref{def:Whitney})
furnishes the reverse ladder, $\dd\,\mathcal{W}_h=\mathcal{W}_h\,\tD$
with $\mathcal{R}_h\mathcal{W}_h=\mathrm{Id}$, so the two interpolations
relate the continuous and discrete complexes as a commuting pair.

\paragraph*{Metric and the primal--dual pair.}
The exterior derivative is metric-free; all metric information,
and all approximation error, resides in the Hodge star. The star closes
the single complex above into the primal--dual pair
\[
\begin{array}{ccccccc}
C^0(\KK) & \xrightarrow{\ \bD_1\ } & C^1(\KK)
  & \xrightarrow{\ \bD_2\ } & C^2(\KK)
  & \xrightarrow{\ \bD_3\ } & C^3(\KK)\\[3pt]
\big\updownarrow{\scriptstyle\,\star_0} & &
\big\updownarrow{\scriptstyle\,\star_1} & &
\big\updownarrow{\scriptstyle\,\star_2} & &
\big\updownarrow{\scriptstyle\,\star_3}\\[3pt]
C^3(\KKs) & \xleftarrow{\ \tD_2\ } & C^2(\KKs)
  & \xleftarrow{\ \tD_1\ } & C^1(\KKs)
  & \xleftarrow{\ \tD_0\ } & C^0(\KKs),
\end{array}
\]
with columns aligned by Hodge duality
($\star_k:C^k(\KK)\to C^{3-k}(\KKs)$, \Cref{def:hodge}) and primal and
dual coboundaries related by transposition, the discrete form of
integration by parts. On a Delaunay--Voronoi mesh the primal--dual
orthogonality makes each $\bM_k$ diagonal,
$(\bM_k)_{ii}=|\sigma_i^*|/|\sigma_i|$, so the codifferential
$\delta_h=\bD\circ\bM$ is local and explicitly invertible. The diagonal
star is an $\OO(h^{r_\star})$ approximation of the continuous Hodge star
(\Cref{lem:hodge_error}); all approximation error in the scheme is
concentrated in it.

\paragraph*{$H(\div)$ and $H(\curl)$.}
The classical conforming spaces at lowest order correspond to cochain degrees through
their degrees of freedom:
\begin{center}
\begin{tabular}{llll}
  \toprule
  Continuous space & Form degree & Degree of freedom & DEC cochain \\
  \midrule
  $H^1$       & $0$-form & point value
              & $C^0$ (vertices)\\
  $H(\curl)$  & $1$-form & $\int_{e}\bu\cdot\bt\,d\ell$ (circulation)
              & $C^1$ (edges)\\
  $H(\div)$   & $2$-form & $\int_{f}\bu\cdot\bn\,dA$ (flux)
              & $C^2$ (faces)\\
  $L^2$       & $3$-form & $\int_{K}\bu\,dV$
              & $C^3$ (cells)\\
  \bottomrule
\end{tabular}
\end{center}
The lowest-order conforming spaces are the Whitney forms
(\Cref{def:Whitney}): the lowest-order N\'ed\'elec (edge) element is the
Whitney $1$-form and the lowest-order Raviart--Thomas (face) element is
the Whitney $2$-form~\cite{nedelec1980,raviartthomas1977,bossavit1998}.
The spaces $H(\curl)$
and $H(\div)$ are the same complex related by $\star_1$: a
$1$-cochain (circulation) and its Hodge dual (flux on the complementary
dual cell) are the two ends of $\bM_1$.
The velocity is a dual $1$-cochain with circulation degrees of freedom
$(\mathcal{R}_h\bu^\flat)_j=\int_{e_j^*}\bu\cdot d\ell$
(\Cref{def:deRham}), hence on the $H(\curl)$, tangential side, placed on
dual edges; its Hodge dual $\bM_1\bv$ is the $H(\div)$, normal-flux
representation on primal faces. The two exterior operations on velocity
therefore have different status: the vorticity $\tD_1\bv$ uses no Hodge
star and is the $H(\curl)$-side curl
\eqref{eq:deRham_commutativity}, whereas the divergence
$\bD_2\bM_1\bv\in C^3(\KK)$ uses $\bM_1$ to the $H(\div)$ side
and is consistent only to $\OO(h^{r_\star})$
(\Cref{lem:interp_error}\,(iii)). 

\begin{remark}[Relation to compatible finite elements]
The complex above coincides topologically with
that of compatible finite element methods (CFEM)~\cite{cotter2023}, the finite element descendant of the Arakawa
C-grid~\cite{arakawalamb1977}; the lowest-order, mass-lumped CFEM is the
primal--dual mimetic scheme
of~\cite{thuburncotter2012,thuburncotter2015}, which coincides with the
DEC complex. The frameworks differ in the metric. In CFEM the
codifferential is defined weakly and its evaluation inverts a
globally coupled mass matrix; mass-lumping
that Galerkin Hodge star produces the local diagonal star used
here. 
Two further differences are structural: the DEC
complex realises only the lowest-order Whitney representatives, with no
analogue of the higher N\'ed\'elec/Raviart--Thomas families, so its
accuracy is governed by the Hodge-star exponent $r_\star$
(\eqref{eq:rstar_def}) rather than a local polynomial degree; and the
commuting projection of CFEM, bounded on the full energy space, is here
the de~Rham map $\mathcal{R}_h$, which commutes but is defined
only on regular forms.
\end{remark}

\begin{remark}[Relation to mimetic and virtual element inner products]\label{rem:vem}
The metric closure admits a comparison against mimetic and
virtual element inner products on general polytopal
meshes~\cite{BrezziLipnikovShashkov2005,beiraoVeiga2013basic}, whose
local matrices are fixed by polynomial consistency only up to a
stabilisation family; lowest-order virtual elements are equivalent to
mimetic finite differences. \Cref{prop:hodge_rigidity} locates the
present scheme in this numerical design space: the DEC star is the unique
diagonal member, available on orthogonal pairings, with
vanishing stabilisation. The complementary fact is that the
conservation structure established in \Cref{section_EulerIncomp} is
indifferent to the choice. The antisymmetry underlying kinetic energy
conservation is the matrix identity
\eqref{eq:antisym_11}, in which $\bM_1$ enters only through the
cancellation $\bM_1\bM_1^{-1}$ inside the contraction
(\Cref{def:contraction}); the Bernoulli gradient cancels through the
incompressibility constraint itself; 
and the commutativity $\tD_1\mathcal{R}_h = \mathcal{R}_h\,d$ is purely
combinatorial. Consequently the conservation laws of
\Cref{section_EulerIncomp} (energy and Kelvin circulation) 
remain intact when the diagonal star is replaced by a symmetric
positive-definite inner product of mimetic or virtual element type on
a general polytopal mesh, with the discrete divergence-free space
redefined as $\ker(\bD_2\bM_1)$ for $\bM_1$.
Delaunay--Voronoi orthogonality renders the
inner product diagonal (local and explicitly invertible,
the rigidity of \Cref{prop:hodge_rigidity}), and it underlies the
projection bound (\Cref{lem:proj_error}) on which the smooth-regime
rates rest. 
\end{remark}

\subsection{Standing Approximation Properties}\label{subsect:approx_hypotheses}

The scheme rests on a collection of approximation properties,
each of which is verified below. This list serves as
the FEEC analogue of Arnold--Falk--Winther~\cite{arnold2006,arnold2010};
the verification that the DEC scheme on Delaunay--Voronoi meshes
satisfies these properties enables the convergence theory.

\begin{enumerate}[label=\textup{(H\arabic*)},ref=\textup{H\arabic*}, leftmargin=3em, itemsep=0.05mm]
\item\label{H:cochain_complex}
\textit{Cochain complex.} The coboundary operators $\tD_0, \tD_1, \tD_2$
satisfy $\tD_{k+1}\circ\tD_k = 0$, so
$C^0\xrightarrow{\tD_0} C^1\xrightarrow{\tD_1} C^2\xrightarrow{\tD_2} C^3$
is a cochain complex. This is a combinatorial property of
the cell complex and holds by construction (\Cref{def:ext-deriv}).

\item\label{H:deRham_commutativity}
\textit{De~Rham commutativity.} For every smooth differential form
$\alpha$ of degree $k$, $\tD_k\mathcal{R}_h\alpha = \mathcal{R}_h(d\alpha)$
(\Cref{prop:derham_commutativity}). This is a consequence of
Stokes' theorem applied to cells and is exact, with no approximation.

\item\label{H:Hodge_accuracy}
\textit{Hodge star accuracy.} The diagonal Hodge star $\bM_k$ approximates
the metric pairing of a $k$-cochain with the dual cell integral of a
continuous $k$-form to accuracy $\OO(h^{r_\star})$ in the $L_h^2$-norm
(\Cref{lem:hodge_error}), with $r_\star = 1$ on general Delaunay--Voronoi meshes
and $r_\star = 2$ under \Cref{prop:centroid_proximity} (centroid proximity).

\item\label{H:reconstruction_accuracy}
\textit{Reconstruction accuracy.} The averaging reconstruction
$\bar\u_j(\bv)$ at primal faces is an $\OO(h^{r_{\rm rec}})$
approximation of pointwise velocity evaluation
(\Cref{prop:recon_accuracy}), with $r_{\rm rec} = 1$ on general
meshes and $r_{\rm rec} = 2$ under reconstruction symmetry (equivalently,
on meshes on which opposing faces contribute symmetrically to the
average).

\item\label{H:antisymmetry}
\textit{Lamb antisymmetry.} The discrete contraction satisfies
$\ip{\bv}{\Iv(\tD_1\bv)}_1 = 0$ for $\bv\in C^1(\KKs)$
(\Cref{prop:extrusion}, part~2). This algebraic identity
requires \emph{no} mesh regularity beyond the Delaunay property.

\item\label{H:bounded_projection}
\textit{Bounded Leray projection.} The discrete Leray projector
$\bP_h$ onto $V_h = \ker(\bD_2\bM_1)$ is $L_h^2$-orthogonal and
satisfies $\nrm{\bP_h} = 1$. For smooth divergence-free velocities,
$\nrm{(I-\bP_h)\mathcal{R}_h\bu^\flat}_{L_h^2} = \OO(h^{r_\star})$
(\Cref{lem:proj_error}). This is the discrete analogue of the
bounded-cochain-projection hypothesis of FEEC
\cite[Thm.~5.2]{arnold2006}.

\item\label{H:discrete_Poincare}
\textit{Discrete Poincar\'e inequality.} There exists
$\lambda_1^h > 0$, bounded below independently of $h$, such that
$\nrm{e}_{L_h^2}^2 \le (\lambda_1^h)^{-1}\nrm{\tD_1 e}_{\bM_2}^2$
for every $e \in V_h^\perp$ (\Cref{thm:poincare}). The uniform
lower bound follows via Whitney-lift norm-equivalence from the
FEEC Poincar\'e inequality of
Arnold--Falk--Winther~\cite{arnold2006,arnold2010}.

\item\label{H:Poincare_0}
\textit{Discrete Poincar\'e inequality.} There exists
$\mu_1^h > 0$, bounded below independently of $h$, such that
$\nrm{q}_{L_h^2}^2 \le (\mu_1^h)^{-1}\nrm{\tD_0 q}_{\bM_1}^2$
for $q\in C^0(\KKs)$ with $\sum_i(\bM_0)_{ii}q_i = 0$. The
uniform lower bound follows via Whitney-lift norm-equivalence from
standard FEEC Poincar\'e inequality at the 0-form level on
mean-zero subspace, with $\mu_1^h$ converging as $h\to 0$ to the
smallest nonzero eigenvalue $\mu_1$ of the continuous
Laplacian. 

\item\label{H:Whitney_interp}
\textit{Whitney interpolation bounds.} The Whitney map
$\mathcal{W}_h:C^k(\KKs)\to\Lambda^k_h$ satisfies the standard
interpolation estimates $\nrm{\mathcal{W}_h\cdot}_{L^2(\Omega)} \sim \nrm{\cdot}_{L_h^2}$
and $\nrm{d\mathcal{W}_h\cdot}_{L^2(\Omega)} \sim \nrm{\tD_k\cdot}_{\bM_{k+1}}$
up to $\OO(h^{r_\star})$ multiplicative errors (see \cite{dodziuk1976,bossavit1988}).

\item\label{H:Whitney_Lp}
\textit{Whitney $L^p$-equivalence.} For $1<p<\infty$ and $k=0,1,2$,
the Whitney map satisfies
$\nrm{\mathcal{W}_h e}_{L^p(\Omega)}\sim\nrm{e}_{L_h^p}$
for  $e\in C^k(\KKs)$, with equivalence constants depending
only on $p$ and the mesh regularity. The discrete $L_h^p$-norm at
level $k$ is the natural lift of the diagonal $\bM_k$-pairing to
$L^p$.
This extends the $L^2$-equivalence of
\ref{H:Whitney_interp} to general Lebesgue exponents and is a
classical Whitney-form result on shape-regular meshes
(\cite{dodziuk1976,christiansen2007}).
\end{enumerate}

\section{Discrete Incompressible Euler Equations and Their Invariants}\label{section_EulerIncomp}

This section formulates the discrete incompressible Euler equations
and establishes their conservation properties.
The continuous equations, written in the vector-invariant form
of Arnold~\cite{arnold1966}, are
\begin{equation}\label{Euler:cont}
  \frac{\partial v}{\partial t} + \iota_{{\u}}\,\omega + \dd B = 0,
  \qquad
  \dd\star v = 0,
\end{equation}
with velocity 1-form $v = \u^\flat$, vorticity 2-form $\omega = \dd v$, and Bernoulli function
$B = \tfrac{1}{2}|\u|^2 + p + \Phi$, with external potential $\Phi$.
Applying the exterior derivative to the momentum equation
produces the vorticity equation
\begin{equation}\begin{split}
 &\partial_t\omega + \dd(\iota_{{\u}}\omega) = 0,\quad
 \text{or equivalently }\quad\partial_t\omega + \Lie_{\u}\omega = 0,
\end{split}\end{equation}
where we have used Cartan's formula and $\dd\omega=0$.
The continuous Euler equations conserve four fundamental quantities that will be mirrored at the discrete level:
time-reversibility and energy conservation are kinematic
(following from the vector-invariant form alone, with energy
$E = \frac{1}{2}\int_\Omega v\wedge\star v + \int_\Omega\Phi\,\mathrm{vol}$
conserved because the Lamb vector does no work);
Kelvin circulation $\frac{d}{dt}\oint_{c(t)} v = 0$ is topological
(a consequence of $d^2=0$, holding without approximation);
and helicity $\frac{d}{dt}\int_\Omega v\wedge\dd v = 0$ (knottedness of
vortex lines in 3D) is {metric-dependent} (requiring compatibility
of wedge and Hodge star), preserved discretely only at the rate governed
by the Hodge-star accuracy. This three-tier structure (exact kinematic,
exact topological, approximate metric) is the organisational principle
of \Cref{subsection_InvariantsIncomp}.

\subsection{Discrete Euler Equations}
The discrete incompressible Euler equations are given by
 \medskip
\begin{eqbox}
\begin{align}
    &\frac{\dd\v}{\dd t}
    + \Iv(\bom)
    + \tD_0\, (p+\ekin +\Phi) = 0,\label{eq:M}\\
    &D_2\;\star_2^{-1}\;\v = 0.\label{eq:D}
\end{align}

\end{eqbox}
\medskip
The potential $\bPhi\in C^0(\KKs)$ is a time-independent dual 0-cochain, representing gravitational, centrifugal, or any conservative body force. Standard choices are $\bPhi_i = g z_i$ (gravity, $z_i$ = cell-centre height),
$\bPhi_i = -\frac{1}{2}|\Omega\times x_i|^2$ (centrifugal), or a sum of both.

\begin{remark}[Rotation]\label{rem:coriolis}
The centrifugal part of the Coriolis force is already accommodated
through $\bPhi$, being conservative. The remaining planetary
contribution is incorporated, when desired, by replacing the relative
vorticity in the Lamb term with {total} (absolute) vorticity:
$\Iv(\bom)\mapsto\Iv(\bom + \mathbf{f})$, where $\mathbf{f}\in C^2(\KKs)$
is the planetary-vorticity 2-cochain ($\mathbf{f} = 2\,\mathcal R_h\,
(\boldsymbol\Omega\!\cdot\!\mathbf{n})\,\mathrm{dA}$ on the sphere).
This is the only modification required, and it is structurally
transparent: the energy identity is unaffected, since the
antisymmetry $\ip{\bv}{\Iv(\bom)}_1 = 0$ holds verbatim for any
2-cochain in place of $\bom$, so $\ip{\bv}{\Iv(\bom+\mathbf{f})}_1 = 0$
and the planetary term contributes no kinetic energy; Kelvin
circulation conservation carries the standard planetary-vorticity
flux. 
\end{remark}

The discrete vorticity equation follows by applying $\tD_1$ to \eqref{eq:M}:
\begin{equation}
  \frac{\dd\bom}{\dd t}
  + \tD_1\,\Iv(\bom) = 0,\quad\text{equivalently }\quad \frac{\dd\bom}{\dd t}+ \tLie_{\bv}\bom=0.
\end{equation}
The vorticity equation encodes different mechanisms dependent on the spatial dimension.  
\paragraph{Vorticity in 2D and 3D}\label{rem:Lie_2D}
In two dimensions (continuous) vorticity is a scalar that is materially conserved along
particle trajectories, with $D\omega/Dt = 0$.
In three dimensions the vorticity equation reads
\[
\partial_t\bom
  +(\u\cdot\nabla)\bom
  +(\bom\cdot\nabla)\u=0,
\]
with an additional \emph{vortex-stretching} term that amplifies
and reorients vorticity via strain; this term is absent in 2D and
is responsible for the fundamental differences between 2D and 3D
turbulence. Both cases can be written succinctly as
$\partial_t\bom +\tLie_{\bv}\bom=0$.

\paragraph{The discrete Lie derivative.}
The Lie derivative $\tLie_{\bv}\bom = \tdd_1(\Iv\bom)$ captures both
advection and stretching simultaneously via the extrusion-based
contraction: $\Iv\bom$ projects the reconstructed velocity along primal
edges (coupling all velocity and vorticity components), and $\tdd_1$
distributes the result as a 2-form update. In 2D, $\Iv\bom$ is a dual
1-form whose value on each dual edge $e_j^*$ is the vorticity flux across
$e_j^*$, and $\tdd_1$ sums these fluxes around each dual 2-cell:
the discrete vorticity equation is a pure advective transport equation in
flux form, a finite-volume update, reflecting the absence of vortex
stretching in 2D. In 3D, the contraction decomposes naturally into
horizontal/vertical parts on prismatic meshes, with the cross-terms
$\Iv^{hv}(\bom^v) + \Iv^{vh}(\bom^h)$ encoding the stretching-and-tilting
mechanism. The form-based Lie derivative treats transport and stretching
as a single geometric object, avoiding the need to split them:
the dimension-dependence is encoded in the mesh topology
and the extrusion weights, not in any modification of
the algebraic formula.

\begin{proposition}[Uniqueness of the projected momentum equation]
\label{prop:VI_cons_equiv_incomp}
On a Delaunay--Voronoi mesh, the Leray-projected discrete
momentum equation has the unique form
$\ddt\bv = -\bP_h\Iv(\bom)$.
Any discrete advection operator that differs from the
Lie-derivative form $\Iv(\bom) + \tdd_0 q$ by a gradient
$\tdd_0 g$ for some $g\in C^0(\KKs)$ yields the same
projected equation.

In particular, the distinction between vector-invariant
advection $\Iv(\bom) + \tdd_0 q$ and any alternative form
that differs from it by a gradient is absorbed into the pressure.
\end{proposition}
\begin{proof}
The discrete Cartan formula (\cref{def:Lie}) gives
\[
(\tLie_{\bv}\bv)_j
  = (\Iv(\bom))_j + (\tdd_0\,q)_j,
\]
where $\bom = \tD_1\bv$ is the vorticity and
$q = \Iv\bv\in C^0(\KKs)$ is the discrete kinetic energy density.
This is an algebraic identity, not an approximation.
The unprojected momentum equation~\eqref{eq:M} reads
$\ddt\bv + \Iv(\bom) + \tdd_0(p + q + \Phi) = 0$.
Since $\operatorname{range}(\tdd_0) = V_h^\perp$, the Leray
projection $\bP_h$ onto $V_h = \ker(\bD_2\bM_1)$ annihilates
every gradient.
Hence any discrete advection of the form
$\Iv(\bom) + \tdd_0(q + g)$
for arbitrary $g\in C^0(\KKs)$ yields the same projected equation.
\end{proof}

\begin{remark}[Vector-invariant versus conservative form]
\label{rem:Perot_forms}
The discrete Lie derivative $\tLie_{\bv}\bv$ and a discrete
tensor-flux divergence $\nabla_h\cdot(\bv\otimes\bv)$ differ by the
Leibniz defect $\OO(h^2)$ (\cref{prop:extrusion}, Property~3).
The defect is a gradient, so the projected equations agree
(\cref{prop:VI_cons_equiv_incomp}); all conservation properties
(energy, circulation, helicity) depend only on the projected equation
$\ddt\bv = -\bP_h\Iv(\bom)$.
See Perot~\cite{Perot_JCP} for a related analysis.
\end{remark}

\medskip
\paragraph*{Pressure equation}
We obtain the pressure equation by applying the divergence to \eqref{eq:M}
and using that $\ddt(D_2\,\star_2^{-1}\,\v)=0$.
The pressure is obtained by solving the Poisson equation
\[
	\Delta_h\; p=-D_2\;\star_2^{-1}\;\big(\Iv(\bom)+ \tD_0K\big).
\]
with  the {discrete Laplacian} on dual 0-forms given by
$$\Delta_h :=D_2\,\star_2^{-1}\,\tD_0: C^0(\KKs)\to C^3(\KK).$$

\subsection{Discrete Invariants}\label{subsection_InvariantsIncomp}
The discrete Euler equations inherit all four invariants
of the continuous system: time reversibility and energy conservation
hold as algebraic identities for any linear velocity reconstruction;
Kelvin circulation holds by the cochain complex property;
helicity conservation is approximate, controlled by the reconstruction
and Whitney consistency of the velocity and vorticity at cochain levels
$k=1,2$ (\cref{lem:helicity_consistency}).

\subsubsection{Time Reversibility}
\label{sec:time_reversibility}

\begin{theorem}[Discrete time reversibility]\label{thm:time_reversal}%
The semi-discrete Euler system \eqref{eq:M}- \eqref{eq:D}
is invariant under the discrete time-reversal map
\[
  \mathcal{T}_h:\;
  (t,\,\bv,\,\bom,\,B)
  \;\longmapsto\;
  (-t,\,-\bv,\,-\bom,\,B).
\]
That is, if $\bigl(\bv(t),\bom(t),B(t)\bigr)$ is a solution
on $[0,T]$, then $\bigl(-\bv(T-\tau),\,-\bom(T-\tau),\,B(T-\tau)\bigr)$
is a solution on $[0,T]$.
\end{theorem}

\begin{proof}
Set $(\hat\bv,\hat\bom,\hat B)(\tau) = (-\bv(T{-}\tau),\,-\bom(T{-}\tau),\,B(T{-}\tau))$.
The chain rule gives $\partial_\tau\hat\bv = -\partial_\tau\bv(T{-}\tau)
= +\partial_t\bv|_{t=T-\tau}$. The discrete contraction is bilinear in $(\bv,\bom)$,
so $\Iv(\hat\bom)\big|_{\hat\bv} = I_{-\bv}(-\bom)|_{T-\tau} = +\Iv(\bom)|_{T-\tau}$. 
And $\tD_0\hat B = +\tD_0 B|_{T-\tau}$ since $B$ is
unchanged. Substituting into~\eqref{eq:M} for the time-reversed system:
\[
  \partial_\tau\hat\bv + I_{\hat\bv}(\hat\bom) + \tD_0\hat B
  = \partial_t\bv|_{T-\tau} + \Iv(\bom)|_{T-\tau} + \tD_0 B|_{T-\tau}
  = 0
\]
by~\eqref{eq:M} evaluated at $t = T{-}\tau$. The divergence
constraint~\eqref{eq:D} is linear in $\bv$, so it is satisfied by
$\hat\bv = -\bv(T{-}\tau)$.
Full discretisations inherit time reversibility if the time integrator is symmetric
(e.g.\ implicit midpoint).
\end{proof}

\subsubsection{Energy Conservation}
\label{sec:energy}

\begin{theorem}[Kinetic energy conservation]\label{thm:energy}%
Let $\v\in C^1(K^*)$ be a solution of the discrete incompressible Euler equations \eqref{eq:M}. Define the discrete global kinetic energy
\[
  \Ekin(t):= \frac{1}{2}\;\ip{\v}{\v}_1
 \qquad (= \frac{1}{2}\sum_{j}\frac{A_j}{\ell_j^*}\;\v_j^2).
\]
Then the kinetic energy is conserved: $\dfrac{d\Ekin}{dt}=0$.
\end{theorem}
\begin{proof}
The proof is given in \Cref{app:energy}. 
\end{proof}

\subsubsection{Helicity Conservation}
\label{sec:helicity}

Helicity $H = \int v \wedge \omega$ measures the knottedness of
vortex lines in three dimensions. Unlike energy and circulation,
helicity conservation in the discrete setting is approximate:
it is controlled by the reconstruction and Whitney consistency of the
velocity and vorticity reconstructions
(\cref{lem:helicity_consistency,prop:recon_accuracy})
rather than by an algebraic identity, and the rate of
control matches the convergence rate of the scheme itself.

The averaging reconstruction recovers pointwise velocity vectors
$\tilde\bu(v_i^*)$ at dual vertices. 
The same construction applied to the vorticity 2-cochain
$\bom = \tD_1\bv$ (dividing each face integral $\omega_k$
by the face area $|f_k^*|$ and inverting the same Gram matrix)
recovers a pointwise vorticity $\tilde\bom(v_i^*)$;
both reconstructions are $\OO(h^{r_{\rm rec}})$ (\cref{prop:recon_accuracy}).
The \emph{reconstructions} $\Qh^1\bv$ and $\Qh^2(\tD_1\bv)$
denote the piecewise-linear extensions (barycentric interpolation on the simplicial refinement $\KK_h^{\rm simp}$) of the pointwise
velocity and vorticity values, respectively. As piecewise-linear
fields on a shape-regular simplicial refinement, both reconstructions
satisfy the norm-equivalence bound
\begin{equation}\label{eq:Qh_bdd}
  \nrm{\Qh^1\bv}_{L^2(\Omega)} \le C_Q^{(1)}\,\nrm{\bv}_{L_h^2},
  \qquad
  \nrm{\Qh^2\bom}_{L^2(\Omega)} \le C_Q^{(2)}\,\nrm{\bom}_{\bM_2},
\end{equation}
with constants $C_Q^{(k)}$ depending only on mesh regularity. The
bound for $\Qh^1$ follows from the averaging reconstruction
$|\bu(v_i^*)|^2 \le C\sum_{n}(v_n/\ell_n^*)^2$ 
(bounded condition number of $G_i$ under \Cref{ass:mesh_reg}), 
$L^2$-on-simplex equivalence $\nrm{\Qh^1\bv}_{L^2(K)}^2 \asymp |K|\cdot\max_{v\in\partial K}|\bu(v)|^2$, 
and finite-overlap summation over simplices
of $\KK_h^{\rm simp}$ adjacent to each dual edge $e_n^*$; the bound for
$\Qh^2$ is analogous with face values $\bom(v_i^*)$ and the dual
2-cochain norm $\nrm{\cdot}_{\bM_2}$.

\begin{definition}[Discrete helicity]\label{def:helicity}
The \emph{discrete helicity} is
\begin{equation}\label{eq:helicity_def}
  H_h(t) := \int_\Omega(\Qh^1\bv)\wedge(\Qh^2\tD_1\bv),
\end{equation}
the discrete analog of
$\int_\Omega v\wedge\dd v = \int_\Omega\bu\cdot\bom\,dV$.
\end{definition}

\begin{theorem}[Helicity conservation]\label{thm:helicity}%
Let $\bv$ be a solution of the discrete incompressible Euler system
\eqref{eq:M}, \eqref{eq:D} with $\bom = \tD_1\bv$.
Then the discrete helicity~\eqref{eq:helicity_def} satisfies
\begin{equation}\label{eq:helicity_rate}
  \Bigl|\frac{dH_h}{dt}\Bigr|
  \le C(T)\,h^{\min(r_{\rm rec},\,r_\star)},
\end{equation}
i.e.\ $\OO(h)$ in case (A) and $\OO(h^2)$ in case (B) of
\Cref{conv:cases}, matching the convergence rate of the scheme.
\end{theorem}
\begin{proof}
See \Cref{app:helicity} for the proof.
\end{proof}

\subsubsection{Kelvin Circulation Theorem}
\label{sec:kelvin}
The discrete analogue of Kelvin's theorem (the circulation
$\Gamma = \oint_{c(t)} v$ around a materially advected loop is conserved) requires
a notion of material loop on the dual mesh. A dual 1-chain
$\gamma = \sum_j \alpha_j e_j^*$ with $\alpha_j \in \{-1,0,1\}$ is a
\emph{1-cycle} if $\partial\gamma = 0$ (a closed path through the dual mesh);
its circulation is the duality pairing
$\Gamma(t) = \sum_j \alpha_j(t)\, v_j(t)$.
The discrete advection of $\gamma$ is defined via the chain Lie derivative
$\dot\gamma = \tLie_v^{\rm chain}\gamma$, which commutes with the boundary
operator, so a 1-cycle remains a 1-cycle under advection, the discrete
analogue of the flow map sending closed curves to closed curves.

\begin{definition}[Discrete circulation]\label{def:circulation}
$i)$ (Discrete Circulation) Let $\gamma = \sum_j\alpha_j e_j^*$ be a dual 1-cycle
($\partial\gamma=0$, $\alpha_j\in\{-1,0,+1\}$).
The discrete circulation is
\[
  \Gamma(\gamma,t) = \v(\gamma) = \sum_j\alpha_j\,\v_j(t).
\]
$ii)$ (Material loop advection) A dual 1-chain $\gamma(t)$ is materially advected if its time
evolution is given by the discrete Lie derivative
\[
  \ddt\gamma(t) = \tLie_{\v}^{\mathrm{chain}}\;\gamma(t),
\]
where $\tLie_{\v}^{\mathrm{chain}}$ is the adjoint of the cochain
Lie derivative, defined by the duality pairing
\[
  \ip{\alpha}{\tLie_{\v}^{\mathrm{chain}}\gamma}
  = \ip{\tLie_{\v}\alpha}{\gamma}
  \qquad \forall\;\alpha\in C^1(\KKs).
\]
\end{definition}

\begin{theorem}[Discrete Kelvin circulation theorem]\label{thm:kelvin}%
For a materially advected dual 1-cycle $\gamma(t)$
\[
    \ddt\Gamma(\gamma(t),t)
    = \ddt\bigl[\v(t)\bigl(\gamma(t)\bigr)\bigr] = 0.
\]
\end{theorem}
\begin{proof}
The proof combines three ingredients: (i)~the product rule for duality pairings, (ii)~the discrete Cartan formula $\tLie_{\bv}\bv = \tdd_0 q + \Iv(\bom)$, and (iii)~the cycle property $(\tdd_0\phi)(\gamma)=0$ for any $\phi$. See \Cref{app:kelvin} for the complete four-step argument.
\end{proof}

\begin{corollary}%
For any fixed dual 2-chain $S$ with $\partial S = \gamma$
\[
  \ddt\int_\gamma \v
  = \ddt\int_S\bom = 0,
\]
i.e.,\ the {vortex flux through a material surface} is conserved.
\end{corollary}
\begin{proof}
The result follows from \Cref{thm:kelvin} with discrete Stokes.
\end{proof}


\section{Finite-Dimensional Well-Posedness}\label{section_WellPosedness}

On a fixed mesh with $N$ degrees of freedom,
the discrete Euler and Navier--Stokes systems are
finite-dimensional ODEs.
Energy conservation (Euler) or energy dissipation
(Navier--Stokes) provides the a priori bounds needed
for global existence.
This section establishes that both systems are globally
well-posed for every fixed mesh.

\subsection{Discrete Incompressible Euler Equations}

We recall the semi-discrete Euler system~\eqref{eq:M}--\eqref{eq:D},
\begin{equation}\begin{split} \label{eq:V}
  \ddt\bv + \Iv(\bom) + \tD_0 B &= 0,
 \\%
  \bD_2\bM_1\bv &= 0, \\%
  B &= p + \ekin + \bPhi, \qquad \ekin_i = \frac{(\bM_1\bv\twdg\bv)_i}{2|K_i|},
  \quad \bPhi_i = \Phi(x_i). %
\end{split}\end{equation}

\begin{definition}[Divergence-free subspace and projector]
i) Define the {discrete divergence-free subspace} by
\begin{equation}\label{eq:Vh}
  V_h := \ker(\bD_2\bM_1) = \{\bv\in\RR^N : \bD_2\bM_1\bv = 0\}.
\end{equation}
This is a linear subspace of $\RR^N$ of dimension $N - M + 1$.\\
ii) The discrete Leray projector $\bP_h:\RR^N\to V_h$ is defined by:
\begin{equation}\label{eq:Leray}
  \bP_h\bw = \bw - \tD_0\phi,\qquad
  \text{where}\;\phi\;\text{solves}\;
  L_h\phi = \bD_2\bM_1\bw
\end{equation}
\end{definition}
where $L_h := \bD_2\bM_1\tD_0$ is the discrete Laplacian.

\begin{lemma}[Reduction to an ODE on $V_h$]
\label{thm:projected_ODE}
Applying $\bP_h$ to the momentum equation \eqref{eq:V} and using
$\bP_h\tD_0 B = 0$ yields the following ODE on $V_h$%
\begin{equation}\label{eq:proj_ODE}
    \ddt\bv = -\bP_h\Iv(\tD_1\bv)
    =: \bff_h(\bv),
    \qquad \bv(0) = \bv_0\in V_h.
\end{equation}
This is an autonomous ODE on the finite-dimensional subspace $V_h\subset\RR^N$.
If $\bv(0)\in V_h$, then $\bv(t)\in V_h$ for all time. 
\end{lemma}
\begin{proof}
Applying $\bP_h$ to the momentum equation \eqref{eq:V} yields  $\bP_h\ddt\bv + \bP_h\Iv(\bom) + \bP_h\tD_0 B = 0$.  Since $\tD_0 B\in\mathrm{range}(\tD_0) = V_h^\perp$, we have $\bP_h\tD_0 B = 0$.  For $\bv\in V_h$, $\bP_h\ddt\bv = \ddt\bv$, giving \eqref{eq:proj_ODE}.  Invariance of $V_h$ follows: $\bff_h(\bv) = -\bP_h\Iv(\tD_1\bv) \in V_h$ for any $\bv$.
\end{proof}


\begin{lemma}[Boundedness of the discrete operators]
\label{lem:bounded_ops}
On a fixed mesh $\KK_h$, all discrete operators are bounded linear maps
between finite-dimensional spaces.  In particular, there exist
constants depending on $h$ such that:
\begin{align}
  \nrm{\tD_1\bv}_2 &\le C_1(h)\nrm{\bv}_2,
  \\
  \nrm{\bP_h\bw}_2 &\le \nrm{\bw}_2,
  \\
  \nrm{\bM_1}_2 &= \max_j(\bM_1)_{jj} =: M_+(h),
  \\
  (\bM_1)_{jj} &\ge \min_j(\bM_1)_{jj}=:m_-(h) > 0 \quad\forall\;j.  
\end{align}
Here $C_1(h) = \nrm{\tD_1^h}_2$ is the spectral norm of the
incidence matrix, and $\bP_h$ is an $\bM_1$-orthogonal projector
\end{lemma}
\begin{proof}
All operators are matrices on finite-dimensional spaces, so the bounds follow from the spectral norm.  $\bP_h$ is an $\bM_1$-orthogonal projector, hence $\nrm{\bP_h\bw}_{\bM_1} \le \nrm{\bw}_{\bM_1}$ for all $\bw$; $\bM_1$ is diagonal with positive entries, so $\nrm{\bM_1}_2 = \max_j(\bM_1)_{jj}$ and $\min_j(\bM_1)_{jj} > 0$.
\end{proof}
\begin{lemma}[Lipschitz continuity of the extrusion]
\label{lem:Lip_extrusion}
For $\bv,\bw\in V_h$ with $\nrm{\bv}_2,\nrm{\bw}_2\le R$, the
extrusion operator satisfies
\[
\nrm{\Iv(\tD_1\bv) - \Iw(\tD_1\bw)}_2
  \le L(h,R)\nrm{\bv-\bw}_2,
\]
where $L(h,R) = 2R\,C_P(h)\,C_Q(h)$ is a Lipschitz constant depending on the mesh
parameter $h$ and the radius $R$.
\end{lemma}
\begin{proof}
The proof is given in \Cref{app:wellposedness}.
\end{proof}
\begin{lemma}[Energy bound]
\label{thm:energy_bound}
Solutions $\bv(t)$ of the discrete Euler system \eqref{eq:V} conserve the kinetic energy
\begin{equation}\label{eq:energy_cons}
  \Ekin(t) = \frac{1}{2}\ip{\bv(t)}{\bv(t)}_1 = \Ekin(0) = \Ekin_0
  \quad\forall\;t\in[0,T_*^h).
\end{equation}
This provides the a priori bound
$  \nrm{\bv(t)}_2^2 \le \frac{2\Ekin_0}{m_-(h)}. $%
\end{lemma}
\begin{proof}
Recall from \eqref{eq:V} that $\ddt\bv = -\Iv(\bom) - \tD_0 B$.
Differentiating $\Ekin$ in time gives
$\ddt\Ekin = \ip{\bv}{\ddt\bv}_1
= -\ip{\bv}{\Iv(\bom)}_1 - \ip{\bv}{\tD_0 B}_1$.
The first term vanishes by the energy identity~\eqref{eq:energy_cancel};
the second by summation-by-parts and $\bD_2\bM_1\bv = 0$.
Hence $\ddt\Ekin = 0$.
\end{proof}


\begin{theorem}[Well-posedness of the discrete Euler system]
\label{thm:global}
For any initial data $\bv_0\in V_h$, the projected Euler ODE
\eqref{eq:proj_ODE} has a \emph{unique global solution}
$\bv\in C^1([0,\infty);\,V_h)$, depending continuously on the
initial data. Short-time existence with
$T_*^h \ge 1/L(h,2\nrm{\bv_0}_2)$ (\Cref{thm:picard}) follows
from Picard--Lindel\"of; global existence from energy conservation
$\nrm{\bv(t)}_{L_h^2} = \nrm{\bv_0}_{L_h^2}$ (\Cref{thm:energy}),
which prevents blow-up.
\end{theorem}

\phantomsection\label{thm:picard}%

\begin{proof}
See \Cref{app:wellposedness}.
\end{proof}


%
%
%
%
%

\subsection{Discrete Incompressible Navier--Stokes}
\label{subsect:NS_wellposedness}
Adding a viscous force to the Euler
system gives the semi-discrete Navier--Stokes equations
\begin{equation}\begin{split}\label{eq:NS_M_wp}
  \ddt\bv + \Iv(\bom) + \tD_0 B &= \mathrm{f}_{\rm visc}, \\
  \bD_2\bM_1\bv &= 0, \\
  B &= p + \ekin + \bPhi,
\end{split}\end{equation}
where $\mathrm{f}_{\rm visc}$ denotes one of the viscous forces defined below. 
Applying the Leray projector $\bP_h$ eliminates the pressure gradient
(the viscous term is already divergence-free), giving
\begin{equation}\label{eq:NS_proj_ODE}
    \ddt\bv = -\bP_h\Iv(\tD_1\bv) - \nu\Deltah\bv
    =: \bff_h^\nu(\bv),
    \qquad \bv(0) = \bv_0\in V_h.
\end{equation}

\paragraph{Dissipation} The Newtonian stress tensor $\boldsymbol{\sigma} = 2\nu\mathbf{S}$
has the coordinate-free form
$\mathbf{S} = \tfrac{1}{2}\mathcal{L}_\u g$, and the divergence
appears in the momentum equation via
\[
(\nabla\cdot\boldsymbol{\sigma})^\flat = \nu(\dd\delta v + \delta\dd v),
\]
which for divergence-free flow reduces to $\nu\delta\dd v$; the Ricci
correction on curved manifolds is absorbed by the Hodge star in our
discretisation. A general anisotropic viscosity
$\boldsymbol{\sigma} = \mathbb{C}:\mathbf{S}(\u)$
with fourth-order positive-semidefinite tensor $\mathbb{C}$
yields a second-order operator on 1-forms; the geophysically important
special case is the horizontal/vertical split
$\nu_h\Delta_h\u^h + \nu_v\partial_{zz}\u$ with $\nu_v\ll\nu_h$
(thin-domain flows).
The isotropic viscous term requires a discrete Hodge--Laplacian on dual 1-cochains, built from the codifferential and the exterior derivative.

\begin{definition}[Discrete codifferential and Hodge--Laplacian]\label{def:hodge_laplacian}
$i)$ The {discrete codifferential} on dual 2-cochains is the operator
\begin{equation}\label{eq:codiff}
  \codiff := \bM_1^{-1}\tD_1^T\bM_2,
\end{equation}
mapping vorticity 2-cochains to velocity 1-cochains.
The \emph{discrete curl-curl Laplacian} on dual 1-cochains is
\begin{equation}\label{eq:Deltah}
  \Deltah := \codiff\,\tdd_1 = \bM_1^{-1}\tD_1^T\bM_2\tD_1.
\end{equation}
where $\bM_k$ are diagonal Hodge-star matrices for $k$-cochains, and $\tD_k$ are dual coboundary operators.\\
$ii)$ The viscous term in the discrete incompressible Navier--Stokes system  reduces to
\begin{equation}\label{eq:viscous_incompressible}
  \mathrm{f}_{\rm visc}^{\rm incomp} 
  = -\nu\,\codiff\,\bom,
\end{equation}
\end{definition}

All results below are stated for the isotropic Hodge--Laplacian~\eqref{eq:viscous_incompressible}.
For the prismatic meshes of
geophysical models the natural choice splits the dual 1-edges into
horizontal and vertical subsets with independent coefficients
$\nu_h,\nu_v$; this recovers the Laplacian-viscosity operator structure
used in ICON-O~\cite{korn2017} and NEMO. 
An anisotropic viscosity $\mathbb{C}$ defines a modified inner product
on velocity 1-forms: in place of the Hodge$\star$ $\bM_1$ one uses a
viscosity-weighted Hodge$\star$ $\bM_1^\nu$ with entries
$(\bM_1^\nu)_{jj} = \nu_j\,(\bM_1)_{jj}$, and the viscous force retains
the algebraic structure of~\eqref{eq:viscous_incompressible} with
$\bM_2^\nu$ in place of $\bM_2$.

Well-posedness of the Navier--Stokes system~\eqref{eq:NS_M_wp} follows the Euler argument with energy conservation replaced by energy dissipation: the viscous term strictly decreases
$\Ekin$.
The energy-dissipation and consistency arguments translate verbatim to the anisotropic case.
Nonlinear eddy viscosities of Ladyzhenskaya type (e.g.\ Smagorinsky,
$\nu_T = (C_s\ell)^2|\mathbf{S}|$) are likewise admissible; for these
the discrete operator is monotone, and a standard Minty--Browder
argument upgrades the weak-convergence result below to a unique limit.

\begin{theorem}[Well-posedness of the discrete Navier--Stokes system]
\label{thm:NS_global}
For any $\nu \ge 0$, any mesh~$\KK_h$, and any
$\bv_0\in V_h$, the projected ODE~\eqref{eq:NS_proj_ODE}
has a unique global solution
$\bv\in C^1([0,\infty);\,V_h)$.
The solution satisfies the energy equality
\begin{equation}\begin{split}\label{eq:NS_energy_diss}
  &\frac{d\Ekin}{dt} = -\nu\,\nrm{\bom}^2_{\bM_2} \le 0,\\
 \text{ and consequently }\quad & \Ekin(t)\le\Ekin_0 \text{ and }\nu\int_0^T\nrm{\bom}^2_{\bM_2}\,dt \le \Ekin_0.
\end{split}\end{equation}
\end{theorem}
\begin{proof}
Short-time existence and uniqueness follow
as in the Euler case
(\cref{thm:picard}).
Differentiating $\Ekin$ gives $\ddt\Ekin =-\nu\nrm{\bom}_{\bM_2}^2.$
Since $\Ekin(t)\le\Ekin_0$, the a priori bound
holds uniformly in~$t$
as in \Cref{thm:energy_bound},
preventing blow-up and giving $T_*^h = +\infty$
(\cref{thm:global}).
\end{proof}

\phantomsection\label{lem:Lip_NS}%
\phantomsection\label{lem:NS_energy_bound}%

\section{Continuous Limit in Various Regularity Regimes}\label{sect:cont_limit}

Having established well-posedness and conservation for every fixed mesh,
we now ask if discrete solutions converge to continuous solutions as
$h\to 0$. The answer depends on the regularity of the continuous
solution: smooth references give explicit rates (\Cref{sec:consistency}),
while weaker regularity yields subsequential convergence in progressively
weaker senses (\Cref{sec:weak_convergence}--\ref{subsect:CMV}).

\begin{assumption}[Regularity of reference solution]
\label{ass:smooth_ref}
For $\nu\ge 0$ we assume the existence of a smooth solution
$\bu\in C^1([0,T];\,H^{s}(\Omega))$ with $s\ge 3$ to the incompressible
Euler or Navier--Stokes equations on a closed oriented Riemannian
$d$-manifold $\Omega$ ($d=2,3$).
\end{assumption}

These regularity requirements are implied by Taylor expansions in the convergence proofs:
$\bu\in W^{1,\infty}$ suffices for the general-mesh rate,
sharpened to $W^{2,\infty}$
under reconstruction symmetry; helicity conservation uses $W^{2,\infty}$.
By Sobolev embedding $H^s\hookrightarrow W^{k,\infty}$ for $s>k+d/2$,
so $s\ge 3$ suffices for the general-mesh rate and $s\ge 4$
(also required for helicity) for the improved rate.



\subsection{Smooth Regime: Continuous Limit of the Discrete Euler System}
\label{sec:consistency}
We establish convergence of discrete Euler solutions to smooth
solutions via a truncation-error analysis (\Cref{sec:truncation})
combined with an energy-based stability estimate
(\Cref{sec:convergence_thm}).

\subsubsection{Error Analysis in the Inviscid Case}
\label{sec:truncation}
The discrete solution $\bv^h$ is compared to the de~Rham interpolant
$\bar\bv := \mathcal{R}_h\bu^\flat$ of the smooth solution $\bu$
(\Cref{ass:smooth_ref}).
The interpolant
$\bar\bv$ is generally \emph{not} an element of $V_h$: it carries an
$\OO(h^{r_\star})$ discrete-divergence residual
$e^\perp := (I-\bP_h)\bar\bv$ (\Cref{lem:proj_error}), whose interaction
with the nonlinear cross-terms is what controls the rate in the
stability analysis (\Cref{app:stability_proof}, Step~2). 
The deviation $e := \bv^h - \bar\bv$ satisfies a discrete
evolution equation whose right-hand side is the truncation error
$\tau_h$; the $h$-dependence of $\tau_h$ (consistency) combined with
an energy-based stability estimate yields the convergence rate via
Gr\"onwall's inequality.

\begin{definition}
\label{def:truncation}
The {truncation error} is defined as
\[
\tau_h(t) := \ddt{\bar\bv} + \bP_h\I_{\bar\bv}(\tD_1\bar\bv)
  = \ddt{\bar\bv} - \bff_h(\bar\bv).
\]
The magnitude of $\tau_h$ measures the consistency of the discretisation.
\end{definition}

\begin{theorem}[Euler Consistency]
\label{thm:consistency}%
Under \Cref{ass:mesh_reg}, \Cref{ass:smooth_ref}, and
\Cref{conv:cases},
the Leray-projected truncation error satisfies
\[
\sup_{t\in[0,T]}\nrm{\bP_h\tau_h(t)}_{L_h^2}
    \;\le\; C_\tau\,h^{r_{\rm rec}},
\]
i.e.\ $\OO(h)$ in case (A) and $\OO(h^2)$ in case (B).
The constant $C_\tau$ depends on $T$,
$\nrm{\bu}_{C^1([0,T];W^{r_{\rm rec}+1,\infty})}$,
and the mesh regularity, but not on $h$.
\end{theorem}

\begin{proof}%
The proof is given in \Cref{app:consistency_proof}.
\end{proof}

Let $\bv^h(t)$ be the discrete solution from \Cref{thm:global}
with initial data $\bv_0^h := \bP_h\mathcal{R}_h\bu_0^\flat$,
and define the error $e(t) := \bv^h(t) - \mathcal{R}_h\bu^\flat(t)$.

\begin{theorem}[Euler stability]
\label{thm:stability}%
Under the assumptions of \Cref{thm:consistency}, on orthogonal DV meshes
(\Cref{ass:mesh_reg}), the error satisfies
\[
\sup_{t\in[0,T]}\nrm{{e}(t)}_{L_h^2}
    \le C(T)\,h^{r^\ast}\,|\log h|^{\beta_d},
\qquad r^\ast := \min\bigl(r_{\rm rec},\; r_\star\bigr),
\]
where $\beta_2=\beta_3=1$, and
$C(T)$ depends on $T$, the smooth solution bounds
(up to $\nrm{\bu}_{W^{r_\star+1,\infty}}$), and mesh regularity, but not on $h$.
Under \Cref{conv:cases}: $\OO(h\,|\log h|^{\beta_d})$ in case (A),
$\OO(h^2\,|\log h|^{\beta_d})$ in case (B).
\end{theorem}

\begin{proof}[Proof sketch]
Full details are in \Cref{app:stability_proof,app:trilinear_proof}.
\end{proof}

\begin{lemma}[Trilinear estimate for the Lamb cross-term]\label{lem:trilinear}
Let $\bu\in W^{1,\infty}(\Omega;\Lambda^1)$ be a smooth 1-form,
$\bar\bom = \tD_1\mathcal{R}_h(\bu^\flat)$ its interpolated
vorticity, and ${e}\in C^1(\KKs)$ a dual 1-cochain.
For $d=2$ and $d=3$, the following $h$-independent estimate holds:
\begin{equation}\label{eq:trilinear_bound}
  |\ip{{e}}{I_{{e}}(\bar\bom)}_1|
  \le C_{\rm tri}\nrm{\bu}_{W^{1,\infty}}\nrm{{e}}_{L_h^2}^2,
\end{equation}
where $C_{\rm tri}$ depends only on the mesh regularity constant.
\end{lemma}

\begin{proof}%
The proof is given in \Cref{app:trilinear_proof}.
\end{proof}

\begin{remark}%
\label{rem:A1_crucial}
The energy identity~\eqref{eq:energy_cancel} is central to the proof:
it removes the quadratic error term $\bQ({e},{e})$,
leaving only terms linear in ${e}$.  Without it, the
energy estimate contains a cubic term
$\sim\nrm{{e}}^3_{L_h^2}$
that blows up in finite time. 
\end{remark}

\subsubsection{Euler Convergence Theorem}
\label{sec:convergence_thm}

Combining the $\OO(h^{r_{\rm rec}})$ Euler consistency bound of
\Cref{thm:consistency} with the Euler stability estimate of
\Cref{thm:stability} via Gr\"onwall yields the convergence rate stated in the theorem below.
The argument relies on three ingredients developed in
\Cref{app:stability_proof}: a  trilinear estimate, a
sharp pointwise-reconstruction bound on the projection residual of
\Cref{lem:proj_error}\,(ii),
and a polarised antisymmetry identity that, combined with the
cochain-complex relation $\tD_1\tD_0 = 0$, eliminates the bilinear
Leray remainder that would otherwise reduce the rate in three
dimensions.

\begin{theorem}[Convergence]
\label{thm:convergence}%
Let $\bu\in C^1([0,T];\,H^s(\Omega))$ with $s\ge 3$ be a smooth
solution of the incompressible Euler equations on
a closed oriented Riemannian $d$-manifold $\Omega$ ($d = 2$ or $3$).
Let $\{\KK_h\}$ be a family of Delaunay--Voronoi meshes satisfying
\Cref{ass:mesh_reg}. Let initial data be
$\bv_0^h = \bP_h\mathcal{R}_h\bu_0^\flat$.
Let $\bv^h(t)\in V_h$ be the unique global solution of the discrete
Euler system (\cref{thm:global}).
Then in the discrete norm
\begin{equation}\label{eq:convergence_disc}
  \sup_{t\in[0,T]}\nrm{\bv^h(t) - \mathcal{R}_h\bu^\flat(t)}_{L_h^2}
  \le C(T)\,h^{\min(r_{\rm rec},\,r_\star)}\,|\log h|^{\beta_d},
\end{equation}
where $\beta_2 = \beta_3 = 1$. 
Under \Cref{conv:cases}:
$\OO(h\,|\log h|^{\beta_d})$ in case (A) and
$\OO(h^2\,|\log h|^{\beta_d})$ in case (B), uniformly in
$d = 2, 3$. In the $L^2$ norm by Whitney reconstruction
\[
\sup_{t\in[0,T]}\nrm{\mathcal{W}_h\bv^h(t)
    - \bu(t)}_{L^2(\Omega)}
    \le C(T)\,h\,|\log h|^{\beta_d},
\]
where $C(T)$ depends on $T$, $\nrm{\bu}_{C^1([0,T];H^s)}$, and
the mesh regularity, but not on $h$.
\end{theorem}

\begin{proof}
The estimate~\eqref{eq:convergence_disc}
follows from the stability estimate (\cref{thm:stability}) with rate $r^\ast = \min(r_{\rm rec},\,r_\star)$:
the consistency gives $\OO(h^{r_{\rm rec}})$ forcing (\cref{thm:consistency}),
the initial error is $\OO(h^{r_\star})$ (\cref{lem:proj_error}),
and the nonlinear cross-terms contribute $\OO(h^{r_\star}|\log h|^{\beta_d})$ forcing
(via the pointwise-reconstruction bound of
\Cref{lem:proj_error}\,(ii) and the polarised antisymmetry identity
of \Cref{app:stability_proof}, see in particular Step~2b,
eq.~\eqref{eq:II_bound}).
The linear Gr\"onwall closure of \Cref{app:stability_proof}, Step~3,
then yields the rate.
For the continuous $L^2$-norm
\begin{align*}
  \nrm{\mathcal{W}_h\bv^h - \bu}_{L^2}
  &\le \nrm{\mathcal{W}_h\bv^h
    - \mathcal{W}_h\mathcal{R}_h\bu^\flat}_{L^2}
  + \nrm{\mathcal{W}_h\mathcal{R}_h\bu^\flat
    - \bu}_{L^2} \\
  &\le C_W\nrm{\bv^h - \mathcal{R}_h\bu^\flat}_{L_h^2}
  + C_{\mathrm{int}}\,h 
  \le C_W\cdot C(T)\,h^{r^\ast} + C_{\mathrm{int}}\,h 
  \le C'(T)\,h,
\end{align*}
where $C_W$ is the norm of the Whitney map (\cref{def:Whitney}),
and the second line uses~\eqref{eq:convergence_disc} and the
Whitney approximation property~\eqref{eq:Whitney_approx}.
The continuous $L^2$-rate is $\OO(h)$ for all $d=2,3$, limited
either by the first-order Whitney interpolation or the discrete rate.
\end{proof}

The next theorem states that the convergence of the Euler solutions implies the convergence of the conserved quantities.

\begin{theorem}[Convergence of discrete conserved quantities]
\label{thm:conv_cons}
As $h\to 0$, the discrete conserved quantities converge to their
continuous counterparts:
\begin{align}
  \Ekin_h &= \frac{1}{2}\ip{\bv^h}{\bv^h}_1
  \to \Ekin = \frac{1}{2}\int_\Omega|\bu|^2\,dV,
  \\
  H_h &= \bv^h\twdg\tD_1\bv^h
  \to H = \int_\Omega\bu\cdot\bom\,dV,
  \\
  \Gamma_h^{(\gamma)} &= \sum_{j\in\gamma_h}v_j^h
  \to \Gamma = \oint_\gamma\bu\cdot d\ell.  
\end{align}
Moreover, the discrete energy and circulation are
constant in time ($\ddt\Ekin_h = 0$, $\ddt\Gamma_h = 0$;
\Cref{thm:energy,thm:kelvin}), while the discrete helicity is
\emph{approximately constant}: $|\ddt H_h| \le C\,h^{\min(r_{\rm rec},\,r_\star)}$
(\cref{thm:helicity}).  For the exactly conserved quantities:
\[
|\Ekin_h(t) - \Ekin(t)| = |\Ekin_h(0) - \Ekin(0)| = \OO(h^{r_\star}),
  \qquad
  |\Gamma_h^{(\gamma)}(t) - \Gamma(t)| = |\Gamma_h^{(\gamma)}(0) - \Gamma(0)| = \OO(h).
\]
For the helicity, the approximate conservation and convergence give:
\[
|H_h(t) - H_h(0)| \le C\,h^{\min(r_{\rm rec},\,r_\star)}\,t,
  \qquad |H_h(0) - H(0)| = \OO(h^{r_{\rm rec}}),
\]
so that $|H_h(t) - H(t)| = \OO(h^{\min(r_{\rm rec},\,r_\star)})$
on $[0,T]$.
\end{theorem}
\begin{proof}The proof is in \Cref{app:thm:conv_cons_proof}.

\end{proof}
\subsection{Smooth Regime: Continuous Limit of the Discrete Navier--Stokes System}
\label{subsect:NS_convergence}

This subsection extends the smooth-regime convergence theory from
discrete Euler to discrete Navier--Stokes~\eqref{eq:NS_proj_ODE}.
The proof structure is identical to the Euler case with one essential
addition: the viscous dissipation $-\nu\|\tD_1 e\|_{\bM_2}^2$ appears
on the left-hand side of the stability estimate and absorbs the
viscous part of the truncation error. The viscous truncation is
$\OO(h^{r_\star})$ in the weak form (\Cref{thm:NS_consistency}),
at best $\OO(h^2)$, while the inviscid truncation dominates at
$\OO(h^{r_{\rm rec}})$; for $d\le 3$ the convergence rate is
unchanged from the Euler case, uniformly in $\nu\ge 0$.
Unconditional convergence to weak Leray--Hopf solutions
 is treated in \Cref{sec:weak_convergence}.

\subsubsection{Error Analysis in the Viscous Case}

Let $\bu(t)$ be the smooth Navier--Stokes solution from \Cref{ass:smooth_ref}, and define $\bar\bv(t) = \mathcal{R}_h\bu^\flat(t)$ as before.
Substituting $\bar\bv$ into the discrete Navier--Stokes system \eqref{eq:NS_proj_ODE}, the truncation error is
\[
\tau_h^\nu(t): = \ddt{\bar\bv} - \bff_h^\nu(\bar\bv)
  = \ddt{\bar\bv} + \bP_h I_{\bar\bv}(\tD_1\bar\bv) + \nu\Deltah\bar\bv.
\]

\begin{theorem}[Consistency of the Navier--Stokes discretisation]
\label{thm:NS_consistency}%
The truncation error $\tau_h^\nu = \tau_h^{\rm Euler} + \tau_h^{\rm visc}$
satisfies, for any $e\in C^1(\KKs)$,
\begin{equation}\label{eq:NS_trunc_bound}
  \bigl|\ip{e}{\tau_h^\nu}_1\bigr|
  \le C_\tau\,h^{r_{\rm rec}}\,\nrm{e}_{L_h^2}
    + \nu\,C_\delta'\,h^{r_\star}\,\nrm{\bu}_{H^{s+2}}\,
      \bigl(\nrm{e}_{L_h^2} + \nrm{\tD_1 e}_{\bM_2}\bigr),
\end{equation}
where $r_{\rm rec}$ is the reconstruction accuracy exponent (\cref{eq:r_rec_def}),
$r_\star$ is the Hodge accuracy exponent (\cref{eq:rstar_def}),
$C_\tau$ is the Euler consistency constant (\cref{thm:consistency}),
and $C_\delta'$ depends on mesh regularity only.
In case (A) of \Cref{conv:cases} the effective rate is $\OO(h)$;
in case (B) the rate is $\OO(h^2)$, uniformly in $d$.
\end{theorem}

\begin{proof}
The proof is given in \Cref{app_proof_lem:visc_trunc_weak}.
\end{proof}

The error ${e}(t) := \bv^h(t) - \bar\bv(t)$ satisfies the same structure as the Euler error equation~\eqref{eq:error_ODE}, with an additional viscous dissipation term
\begin{equation}\label{eq:NS_error_expanded}
  \ddt{e} = -\bP_h[\bQ(\bar\bv,{e}) + \bQ({e},\bar\bv) + \bQ({e},{e})] - \nu\Deltah{e} - \tau_h^\nu.
\end{equation}

The interpolant $\bar\bv = \mathcal{R}_h\bu^\flat$ of a divergence-free field satisfies the topological constraint $\tD_2\tD_1\bar\bv = 0$, but the metric divergence $\bD_2\bM_1\bar\bv$ is only $\OO(h^{r_\star})$ due to the Hodge* approximation, 
so $\bar\bv\notin V_h$ 
carries the projection error that is estimated in the next lemma.
\begin{lemma}[Projection error]
\label{lem:proj_error}
For a smooth divergence-free velocity $\bu\in W^{r_\star,\infty}(\Omega)$ it holds that
\begin{align*}
&i)\quad \nrm{(\mathbf{I}-\bP_h)\mathcal{R}_h\bu^\flat}_{L_h^2}
  \le C_\pi\,h^{r_\star}\nrm{\bu}_{W^{r_\star,\infty}},\\
&ii)\quad    \nrm{(\mathbf{I}-\bP_h)\mathcal{R}_h\bu^\flat}_{\rm rec}
 \le C_{\rm rec}\,h^{r_\star}\,|\log h|^{\beta_d}\,\nrm{\bu}_{W^{r_\star+1,\infty}},
\end{align*}
where $\nrm{\cdot}_{\rm rec}$ is the pointwise-reconstruction norm
of \Cref{def:disc_norms}, and $\beta_d=1$ for $d=2,3$.
\end{lemma}
\begin{proof}
The proofs are given in \Cref{app:proj_error_proof}.
\end{proof}

\begin{theorem}[Nonlinear stability for Navier--Stokes]
\label{thm:NS_stability}%
Under assumptions of \Cref{thm:NS_consistency}, let
$r \le \min(r_{\rm rec},\,r_\star)$ and suppose the initial error satisfies
$\nrm{{e}(0)}_{L_h^2}\le C_0\,h^r$.
Then
\[
\sup_{t\in[0,T]}\nrm{{e}(t)}_{L_h^2}
    \le C^\nu\,h^r,
\]
where $C^\nu$ depends on $T$, $\nu$, the smooth solution bounds, and mesh regularity, but not on $h$.
\end{theorem}
\begin{proof}%
The complete proof is given in \Cref{app:NS_stability_proof}.
\end{proof}

\begin{remark}%
\label{rem:NS_A1}
The energy identity~\eqref{eq:energy_cancel} plays the same role as in the Euler convergence: it eliminates the quadratic self-interaction $\bQ({e},{e})$ from the energy estimate. In the Navier--Stokes case, one might hope that the viscous dissipation could control this term even without the identity.  This cannot be guaranteed: the discrete Reynolds number $\nrm{\bar\bv}_{L_h^\infty}\cdot h/\nu$ can be large, and the quadratic term would dominate the dissipation.  Thus the energy identity remains essential for unconditional stability, i.e.\ stability that does not require a CFL-like condition linking $h$ and $\nu$.
\end{remark}

\subsubsection{Navier--Stokes Convergence Theorem}
The viscous consistency bound (\cref{thm:NS_consistency}) and the
stability estimate with dissipation (\cref{thm:NS_stability}) combine
via Gr\"onwall's lemma to give the same rate as the Euler case,
confirming that the viscous term does not degrade convergence for $d \leq 3$.

\begin{theorem}[Convergence of discrete Navier--Stokes to continuous Navier--Stokes]
\label{thm:NS_convergence}%
Let $\nu > 0$ and let $\bu\in C^1([0,T];\,H^s(\Omega))$ with $s\ge 3$ be a smooth
solution of the incompressible Navier--Stokes equations on
a closed oriented Riemannian $d$-manifold $\Omega$ ($d = 2$ or $3$).
Let $\{\KK_h\}$ be a family of meshes satisfying
\Cref{ass:mesh_reg}, and initial data are given by
$\bv_0^h = \bP_h\mathcal{R}_h\bu_0^\flat$.
Let $\bv^h(t)\in V_h$ be the unique global solution of the discrete
Navier--Stokes system (\cref{thm:NS_global}).
Then in the discrete norm
\begin{equation}\label{eq:NS_convergence_disc}
    \sup_{t\in[0,T]}\nrm{\bv^h(t)
    - \mathcal{R}_h\bu^\flat(t)}_{L_h^2}
    \le C^\nu(T)\,h^{\min(r_{\rm rec},\,r_\star)}\,|\log h|^{\beta_d},
\end{equation}
where $\beta_2 = \beta_3 = 1$.
Under \Cref{conv:cases}: $\OO(h\,|\log h|^{\beta_d})$ in case (A) and
$\OO(h^2\,|\log h|^{\beta_d})$ in case (B), uniformly in $d = 2, 3$.
In the continuous $L^2$ norm
\begin{equation}\label{eq:NS_convergence}
    \sup_{t\in[0,T]}\nrm{\mathcal{W}_h\bv^h(t)
    - \bu(t)}_{L^2(\Omega)}
    \le C^\nu(T)\,h,
\end{equation}
where the rate is limited to $\OO(h)$ by the first-order Whitney
interpolation error~\eqref{eq:Whitney_approx}, regardless of $d$.
The constant $C^\nu(T)$ depends on $T$, $\nu$, the regularity of the continuous solution $\nrm{\bu}_{C^1([0,T];H^{s+2})}$, and
the mesh regularity constants, but not on $h$.

The convergence is uniform in $\nu$: for $\nu\in[0,\nu_0]$ the constant $C^\nu(T)$ can be chosen
uniformly in $\nu$, provided the reference solution remains smooth on $[0,T]$
uniformly in $\nu$.
\end{theorem}

\begin{proof}
The proof follows the structure of the Euler result (\cref{thm:convergence}).
The error~\eqref{eq:NS_convergence_disc} follows from
\Cref{thm:NS_stability} with $r = \min(r_{\rm rec},\,r_\star)$:
the initial error is $\OO(h^{r_\star})$ (\cref{lem:proj_error}),
and truncation error
$\OO(h^{r_{\rm rec}}) + \nu\,\OO(h^{r_\star})$
(\cref{thm:NS_consistency}); the Gr\"onwall argument
gives the rate $\min(r_{\rm rec},\,r_\star)$.
The $L^2$-bound~\eqref{eq:NS_convergence} follows 
with Whitney interpolation error~\eqref{eq:Whitney_approx}.
\end{proof}

\subsection{Weak Regime: Leray--Hopf Solutions}
\label{sec:weak_convergence}

The smooth-regime results (\Cref{thm:convergence,thm:NS_stability})
depend on a smooth reference solution, whose existence in 3D
for large data remains a Millennium Problem. This subsection abandons
the smoothness assumption and establishes \emph{unconditional}
subsequential convergence in $L^2$ to a Leray--Hopf weak
solution~\cite{Leray1934,Hopf1951};
the price is loss of an explicit rate.
The argument follows the compactness-plus-energy-inequality paradigm:
the discrete energy inequality (\Cref{lem:NS_energy_bound}) provides
uniform a priori bounds, Aubin--Lions extracts a convergent subsequence
in $L^2$ to a weak solution (\Cref{thm:weak_convergence}).  
A weak--strong uniqueness criterion
(\Cref{thm:weak_strong}) recovers the smooth-regime rate whenever a
Leray--Hopf strong solution exists on $[0,T]$, bridging the two regimes.

\begin{lemma}[Uniform a priori estimates]\label{lem:uniform_apriori}
Let $\bv^h(t)$ be a solution of the discrete Navier--Stokes equations \eqref{eq:NS_proj_ODE} on $\KK_h$,
and $\bu^h(t,x) = \mathcal{W}_h\bv^h(t)$ its Whitney reconstruction.  Then the following estimates are satisfied
\begin{enumerate}[nosep]
\item \textit{Uniform $L^\infty(0,T;\,L^2)$ bound: }
$ \sup_{t\in[0,T]}\nrm{\bu^h(t)}_{L^2(\Omega)}^2 \le C_W^2\Ekin_0.$
\item \textit{Uniform $L^2(0,T;\,H^1)$ bound}
$  \nu\int_0^T\nrm{\nabla\bu^h(t)}_{L^2(\Omega)}^2\,dt \le C_W^2 \Ekin_0 + \OO(h^{r_\star}),$
where the $\OO(h^{r_\star})$ term arises from the Hodge* approximation.
\item \textit{Discrete energy inequality}
$\frac{1}{2}\nrm{\bu^h(T)}_{L^2}^2 + \nu\int_0^T\nrm{\nabla\bu^h}_{L^2}^2\,dt
  \le \frac{1}{2}\nrm{\bu^h(0)}_{L^2}^2 + C_\star h^{r_\star}\,\Ekin_0,$
where the $C_\star h^{r_\star}$ term accounts for the Hodge* error.
\end{enumerate}
\end{lemma}

\begin{proof}
Bound $(1)$  follows from energy dissipation $\Ekin(t)\le \Ekin_0$ (\cref{lem:NS_energy_bound}) and norm equivalence $\nrm{\bu^h}_{L^2} \le C_W\nrm{\bv^h}_{L_h^2}$ (\cref{def:Whitney}).
For $(2)$, the discrete enstrophy $\nrm{\tD_1\bv^h}_{\bM_2}^2$ equals the continuous enstrophy of the Whitney-reconstructed field up to Hodge* errors: $\nrm{\nabla\bu^h}_{L^2}^2 = \nrm{\tD_1\bv^h}_{\bM_2}^2 + \OO(h^{r_\star})\nrm{\bu^h}_{H^1}^2$.
Bound $(3)$ combines $(1)$ and $(2)$.
\end{proof}

\begin{theorem}[Existence of Leray--Hopf weak solutions]
\label{thm:weak_convergence}%
Let $\nu > 0$, let $\bu_0\in L^2(\Omega)$ with $\nabla\cdot\bu_0 = 0$, and let $\{\KK_h\}_{h>0}$ be a family of 
admissible Delaunay--Voronoi meshes satisfying \Cref{ass:mesh_reg}, with $h\to 0$.
Let $\bv^h(t)$ solve the discrete Navier--Stokes system \eqref{eq:NS_M_wp}
with initial data $\bv_0^h:= \bP_h\mathcal{R}_h\bu_0^\flat$, and let $\bu^h:= \mathcal{W}_h\bv^h$ be the Whitney reconstruction.
Then there exists a subsequence $h_k\to 0$ and a vector field $\bu\in L^\infty(0,T;\,L^2(\Omega))\cap L^2(0,T;\,H^1(\Omega))$ such that
\begin{enumerate}[nosep]
\item $\bu^{h_k}\rightharpoonup \bu$ weakly in $L^2(0,T;\,H^1(\Omega))$,
\item $\bu^{h_k}\to \bu$ strongly in $L^\infty(0,T;\,L^2(\Omega))$,
\item $\bu$ is a {Leray--Hopf weak solution} of the incompressible Navier--Stokes equations: for all divergence-free test functions $\boldsymbol\varphi\in C_c^\infty(\Omega\times[0,T))$:
\begin{equation}\label{eq:weak_NS}
  -\int_0^T\!\!\int_\Omega \bu\cdot\partial_t\boldsymbol\varphi
  + \int_0^T\!\!\int_\Omega (\bom\times\bu)\cdot\boldsymbol\varphi
  + \nu\int_0^T\!\!\int_\Omega \nabla\bu:\nabla\boldsymbol\varphi
  = \int_\Omega\bu_0\cdot\boldsymbol\varphi(\cdot,0).
\end{equation}
\item $\bu$ satisfies the \text{energy inequality}:
\begin{equation}\label{eq:Leray_energy_ineq}
  \frac{1}{2}\nrm{\bu(t)}_{L^2}^2 + \nu\int_0^t\nrm{\nabla\bu}_{L^2}^2\,ds
  \le \frac{1}{2}\nrm{\bu_0}_{L^2}^2
  \qquad\text{for a.e.\ } t\in[0,T].
\end{equation}
\end{enumerate}
\end{theorem}

\begin{proof}[Proof sketch]
The proof is given in \Cref{app:Leray_Hopf_proof}.
\end{proof}
The discrete Navier--Stokes system satisfies the energy {equality} (\cref{lem:NS_energy_bound}).  
In the limit $h\to 0$, the equality may degrade to an {inequality}.  
This degradation is not a defect of the discretisation; it reflects the nature of Leray--Hopf weak solutions, for which the energy inequality is the strongest statement that
can be guaranteed.  If the limiting solution happens to be a {strong} solution on an interval of regularity, then the energy equality is recovered.

\begin{remark}[Pressure in the Leray--Hopf regime]\label{rem:pressure_LH}
The weak formulation~\eqref{eq:weak_NS} is tested against divergence-free $\boldsymbol\varphi$, so the discrete Bernoulli gradient $\tD_0 B^h$ is annihilated and pressure does not enter the convergence proof of \Cref{thm:weak_convergence}: only velocity is identified in the limit. A discrete pressure $p^h$ is nevertheless available a posteriori from the Poisson problem. 
\end{remark}

The weak--strong uniqueness in the next theorem establishes that if a strong solution exists on some interval, the discrete solutions converge to it with rates on that interval. Beyond the strong existence time, the discrete solutions still converge subsequentially to a Leray--Hopf weak solution, though without rates or uniqueness guarantees.

\begin{theorem}[Weak-strong uniqueness]
\label{thm:weak_strong}%
Suppose the continuous Navier--Stokes equations admit a strong solution $\bu_{\rm strong}\in C([0,T];H^1(\Omega))\cap L^2(0,T;H^2(\Omega))$ on $[0,T]$.  Then it holds that:
\begin{enumerate}[nosep]
\item Any Leray--Hopf weak solution $\bu$ obtained in \Cref{thm:weak_convergence} coincides with it: $\bu = \bu_{\rm strong}$ on $[0,T]$.
\item The \emph{full sequence} $\bu^h = \mathcal{W}_h\bv^h$ (not just a subsequence) converges to $\bu_{\rm strong}$ in $L^2(0,T;L^2(\Omega))$.
\item The convergence rate
$\OO(h^{\min(r_{\rm rec},\,r_\star)}\,|\log h|^{\beta_d})$ from
\Cref{thm:NS_convergence} is recovered.
\end{enumerate}
\end{theorem}

\begin{proof}
Part (1) is classical Prodi--Serrin weak--strong uniqueness: any
Leray--Hopf weak solution satisfying the energy inequality must
coincide with a strong solution that has the same initial data,
provided the strong solution belongs to the Prodi--Serrin class,
which $\bu_{\rm strong}\in C([0,T];H^1)\cap L^2(0,T;H^2)$ satisfies.
Part (2) follows from (1): every subsequential limit is the same, so the full sequence converges.
Part (3): once full-sequence convergence is established, the quantitative error analysis of \Cref{thm:NS_stability} applies, giving the rate.
\end{proof}

The Leray--Hopf existence and convergence result
(\cref{thm:weak_convergence}) carries over to the anisotropic
Hodge--Laplacian of \Cref{subsect:NS_wellposedness} without change.
For monotone nonlinear eddy viscosities of Ladyzhenskaya type, the
discrete operator's monotonicity additionally permits a Minty--Browder
identification of the nonlinear flux, yielding a \emph{unique} weak
solution of the corresponding Ladyzhenskaya--Smagorinsky system; as
this lies outside the conservative theory developed here, we do not
pursue it.

\subsection{Low Regularity Regime: Onsager Threshold}
\label{subsect:CMV}

The central difficulty for the inviscid Euler equations is
\emph{non-uniqueness}.
Wild weak solutions constructed by convex integration
\cite{scheffer1993,shnirelman1997,DeLellisSzekelyhidi2009,DeLellisSzekelyhidi2013,BDLSV2019}
can violate energy conservation and are manifestly non-unique
below the H\"older threshold $\alpha = 1/3$.
Whether uniqueness holds above $1/3$ depends on the regularity
class.
At Lipschitz regularity ($\alpha \ge 1$), Brenier--De~Lellis--Sz\'ekelyhidi
\cite{BDLS2011} prove measure-valued weak--strong uniqueness (MVWSU) for
conservative measure-valued solutions (CMV); this is the strongest
unconditional result available.
In the energy-conserving range $1/3 < \alpha < 1$, the Constantin--E--Titi
theorem \cite{CET1994} guarantees that the kinetic energy is conserved, but
the MVWSU argument breaks down: weak--strong uniqueness among
$C^{0,\alpha}$ Euler solutions in this range is a major open problem.
Below $1/3$, both dissipative and energy-conservative weak solutions exist
\cite{Isett2018,BDLSV2019}, and non-uniqueness is maximal.

A convergence theory for discrete schemes must contend with this landscape
without knowing in advance which regularity regime the true solution
occupies.
The DEC scheme achieves the following: discrete Euler solutions converge
subsequentially to CMV solutions unconditionally (\cref{thm:CMV});
the defect measure $\sigma$ vanishes precisely at the Onsager threshold
$\alpha > 1/3$ (\cref{thm:sigma_vanish}), reducing the MVWSU hypothesis
to ordinary weak--strong uniqueness and yielding full-sequence convergence
for Lipschitz solutions (\cref{thm:Holder_unconditional}); and no
subsequence can converge to a dissipative Euler solution at any regularity
(\cref{prop:no_dissipative}).
The last exclusion follows because the DEC scheme inherits
energy conservation from the discrete level, so the limit cannot
dissipate energy anomalously.
\subsubsection{Conservative measure-valued Euler solutions}
\begin{definition}[Conservative measure-valued (CMV) Euler solution]\label{def:CMV}
A pair $(\bw, \sigma)$ consisting of
$\bw \in L^\infty(0,T;\,L^2(\Omega))$ with $\nabla\cdot\bw = 0$
and a non-negative symmetric matrix-valued Radon measure
$\sigma_{ij}(t) \in \mathcal{M}(\Omega)$ for a.e.\ $t$
is a \emph{CMV Euler solution}
with initial data $\bu_0$ and energy $E_0 = \nrm{\bu_0}_{L^2}^2$
if:
\begin{enumerate}[nosep]
\item \emph{CMV Euler equation:} for all divergence-free
  $\boldsymbol\varphi \in C_c^\infty(\Omega\times[0,T))$,
  \begin{equation}\label{eq:CMV_Euler}
    -\int_0^T\!\!\int_\Omega\bw\cdot\partial_t\boldsymbol\varphi
    - \int_0^T\!\!\int_\Omega(w_iw_j + \sigma_{ij})\partial_j\varphi_i
    = \int_\Omega\bu_0\cdot\boldsymbol\varphi(\cdot,0).
  \end{equation}
\item \emph{Energy identity:}
  $\nrm{\bw(t)}_{L^2}^2 + \int_\Omega\operatorname{tr}(\sigma)(t) = E_0$
  for all $t \in [0,T]$.
\item \emph{Initial data:} $\bw(0) = \bu_0$ and $\sigma(0) = 0$.
\end{enumerate}
If $\sigma = 0$, then $\bw$ is an energy-conservative weak Euler solution.
\end{definition}

\begin{lemma}[Approximate CMV consistency]\label{lem:weak_consistency}
Let $\bu^h = \mathcal{W}_h\bv^h$ be the Whitney reconstruction of
a discrete Euler solution (\cref{thm:global}).
Then for every divergence-free
$\boldsymbol\varphi \in C_c^\infty(\Omega\times[0,T))$,
\begin{equation}\label{eq:approx_CMV}
  -\int_0^T\!\!\int_\Omega \bu^h\cdot\partial_t\boldsymbol\varphi
  - \int_0^T\!\!\int_\Omega u^h_i\,u^h_j\,\partial_j\varphi_i
  = \int_\Omega \bu^h(0)\cdot\boldsymbol\varphi(\cdot,0)
  + R^h(\boldsymbol\varphi),
\end{equation}
where $|R^h(\boldsymbol\varphi)| \le C(\boldsymbol\varphi,\Ekin_0)\,h^{\min(r_{\rm rec},\,r_\star)}$.
\end{lemma}

\begin{proof}
Let $\boldsymbol\psi(t) := \mathcal{R}_h\boldsymbol\varphi^\flat(t)$ denote the
de~Rham interpolant of the test function.
The discrete momentum equation
$\ddt\bv + I_\bv(\bom) + \tD_0 B = 0$
tested against $\boldsymbol\psi$ gives
\[
  \ip{\ddt\bv}{\boldsymbol\psi}_1
  + \ip{I_\bv(\bom)}{\boldsymbol\psi}_1
  + \ip{\tD_0 B}{\boldsymbol\psi}_1 = 0.
\]
Summation by parts gives
$\ip{\tD_0 B}{\boldsymbol\psi}_1 = B^T \bD_2\bM_1\boldsymbol\psi$.
This pairing vanishes on the discretely divergence-free subspace
$V_h = \ker(\bD_2\bM_1)$, and the interpolant
$\boldsymbol\psi = \mathcal{R}_h\boldsymbol\varphi^\flat$ of the
divergence-free field $\boldsymbol\varphi$ is \emph{not} in $V_h$ in
general: de~Rham interpolation does not commute with the discrete
divergence, so $\bD_2\bM_1\boldsymbol\psi$ is not identically zero but
$\OO(h^{r_\star})$ per cell by the Hodge$\star$ approximation
(\cref{lem:interp_error}\,(iii)). The pressure pairing therefore does not
drop out; it is bounded by
$|\ip{\tD_0 B}{\boldsymbol\psi}_1|
\le C\,h^{r_\star}\,\nrm{B}_{L_h^\infty}\,\nrm{\boldsymbol\varphi}_{W^{1,\infty}}$
and contributes to the residual $R^h$ below.
The discrete inner product
$\ip{\ddt\bv}{\boldsymbol\psi}_1
= \sum_j(\bM_1)_{jj}\ddt v_j\,\psi_j$
approximates $\int_\Omega\ddt\bu^h\cdot\boldsymbol\varphi$
up to the Hodge$\star$ error:
$|\ip{\ddt\bv}{\boldsymbol\psi}_1
- \int_\Omega(\mathcal{W}_h\ddt\bv)\cdot\boldsymbol\varphi|
\le C\,h^{r_\star}\,\nrm{\ddt\bv}_{L_h^2}\,\nrm{\boldsymbol\varphi}_{W^{1,\infty}}$.
The discrete pairing
$\ip{I_\bv(\bom)}{\boldsymbol\psi}_1$
is the discrete counterpart of the continuous Lamb pairing
$\int_\Omega(\bom\times\bu)\cdot\boldsymbol\varphi$, which for
divergence-free $\boldsymbol\varphi$ equals
$-\int_\Omega u_i\,u_j\,\partial_j\varphi_i$ after integration by parts,
the Bernoulli gradient $\nabla(\tfrac12|\bu|^2)$ dropping against
$\nabla\!\cdot\boldsymbol\varphi = 0$. The discrete pairing reproduces
this value up to $\OO(h^{r_{\rm rec}})$:
the error arises from replacing the continuous contraction
$\iota_{\bu}\bom$ by the discrete extrusion $I_\bv(\bom)$
and the continuous inner product by the $\bM_1$-pairing,
both of which introduce errors controlled by the
reconstruction accuracy $r_{\rm rec}$ and the cochain scaling
$h^{d-1}$ (the same mechanism as in the consistency
proof, \Cref{app:consistency_proof}).
Combining the three pairings, integrating over $[0,T]$,
and integrating the time-derivative term by parts
converts the
cochain identity into
~\eqref{eq:approx_CMV}.
Each dual pairing
introduces an error that is absorbed into the residual $R^h$.
The residual collects the Hodge$\star$ errors ($\OO(h^{r_\star})$)
from the pressure and the inner-product transfer,
and the extrusion error ($\OO(h^{r_{\rm rec}})$)
from the convective term.
Since $r_\star = r_{\rm rec}$ under \Cref{conv:cases},
$|R^h| \le C(\boldsymbol\varphi,\Ekin_0)\,h^{\min(r_{\rm rec},\,r_\star)} = C(\boldsymbol\varphi,\Ekin_0)\,h^{r_\star}$.
\end{proof}

\begin{theorem}[Convergence to a CMV Euler solution]\label{thm:CMV}%
Let $\bu_0 \in L^2(\Omega)$ with $\nabla\cdot\bu_0 = 0$ and
$E_0 = \nrm{\bu_0}_{L^2}^2$.
Let $\bu^h = \mathcal{W}_h\bv^h$ with
$\bv^h(0) = \bP_h\mathcal{R}_h\bu_0^\flat$.
Then there exists a subsequence $h_k \to 0$ and a CMV Euler solution
$(\bw, \sigma)$ such that:
\begin{enumerate}[label=\textup{(\roman*)},nosep]
\item $\bu^{h_k}(t) \rightharpoonup \bw(t)$ weakly in $L^2(\Omega)$
  for every $t$;
\item $\nrm{\bu^{h_k}(t)}_{L^2}^2 \to E_0$ for every $t$;
\item $\bu^h(0) \to \bu_0$ in $L^2$ and $\sigma(0) = 0$.
\end{enumerate}
\end{theorem}

\begin{proof}[Proof idea]
Discrete energy conservation (\cref{thm:energy_bound}) bounds
$\bu^h=\mathcal{W}_h\bv^h$ uniformly in $L^\infty(0,T;L^2)$, and the
vector-invariant weak identity of \Cref{lem:weak_consistency} gives
equicontinuity of $t\mapsto\ip{\bu^h(t)}{\boldsymbol\varphi}_{L^2}$
against divergence-free tests; Arzel\`a--Ascoli with a diagonal
argument extracts a time-pointwise weak limit $\bw$, divergence-free
by the $\OO(h^{r_\star})$ approximate incompressibility. The Young
measure generated by $\bu^{h_k}$ yields the concentration defect
$\sigma_{ij}$, and energy conservation passes to the trace,
giving $\int_\Omega\mathrm{tr}(\sigma)(t)=E_0-\nrm{\bw(t)}_{L^2}^2$
with $\sigma(0)=0$; passing to the limit in \eqref{eq:approx_CMV}
identifies $(\bw,\sigma)$ as a CMV Euler solution. The full argument
is given in \Cref{app:CMV_proof}.
\end{proof}

\subsubsection{Strong convergence at H\"older regularity}

\begin{definition}[Measure-valued weak\,--\,strong uniqueness]\label{def:MVWSU}
\emph{MVWSU} holds for a weak Euler solution
$\bu \in L^\infty(0,T;L^2)$ if: whenever $(\bw,\sigma)$ is a CMV
Euler solution with $\bw(0) = \bu(0)$, then $\bw = \bu$ and
$\sigma = 0$ for all $t$.
\end{definition}

\begin{theorem}[Conditional convergence at $\alpha \ge 1/3$]\label{thm:Holder_conditional}
Let $\bu \in L^\infty(0,T;\,C^{0,\alpha}(\Omega))$ with $\alpha \ge 1/3$
be a weak Euler solution that conserves energy.
If MVWSU holds for $\bu$,
then $\bu^h(t) \to \bu(t)$ strongly in $L^2$ for all $t$
(full sequence).
\end{theorem}

\begin{proof}
Let $(\bu^h_k)_k$ be any subsequence.  By \Cref{thm:CMV},
a further subsequence $(\bu^h_{k'})_{k'}$ and a CMV Euler solution
$(\bw,\sigma)$ exist with
$\bu^{h_{k'}}(t)\rightharpoonup\bw(t)$ weakly in $L^2$ for every $t$,
and $\bw(0) = \bu_0$, $\sigma(0) = 0$.
Since the discrete and reference initial data coincide
the MVWSU hypothesis applies and gives
$\bw = \bu$, $\sigma = 0$ for all $t$.

It remains to raise weak to strong convergence.
By~\eqref{eq:CMV_energy_limit},
$\nrm{\bu^{h_{k'}}(t)}_{L^2}^2 \to E_0$ for every $t$.
The energy-conservation on $\bu$ gives
$\nrm{\bu(t)}_{L^2}^2 = \nrm{\bu_0}_{L^2}^2 = E_0$.
Hence $\nrm{\bu^{h_{k'}}(t)}_{L^2} \to \nrm{\bu(t)}_{L^2}$;
combined with $\bu^{h_{k'}}(t)\rightharpoonup\bu(t)$,
the Radon--Riesz theorem gives strong convergence
$\bu^{h_{k'}}(t)\to\bu(t)$ in $L^2$ for every $t$.
Since every subsequence contains a further subsequence
converging to the same limit $\bu$, the full sequence
$\bu^h(t)\to\bu(t)$ strongly in $L^2$.
\end{proof}

\begin{theorem}[Convergence for Lipschitz weak Euler solutions]\label{thm:Holder_unconditional}%
Suppose there exists a weak Euler solution
$\bu \in L^\infty(0,T;\,W^{1,\infty}(\Omega))$ on $[0,T]$.
Then $\bu^h(t) \to \bu(t)$ strongly in $L^2$ for all $t$ (full sequence).
\end{theorem}

\begin{proof}%
By \Cref{thm:Holder_conditional} it suffices to verify
two hypotheses for $\bu$: energy conservation, and MVWSU.

On a closed manifold,
$W^{1,\infty}(\Omega) = C^{0,1}(\Omega)\hookrightarrow C^{0,\alpha}(\Omega)$
for any $\alpha\in(1/3,1]$ (Morrey).
The Constantin--E--Titi theorem~\cite{CET1994}
gives energy conservation
$\nrm{\bu(t)}_{L^2} = \nrm{\bu_0}_{L^2}$ for all $t$
when $\bu\in L^3(0,T;C^{0,\alpha})$ with $\alpha>1/3$;
this is a fortiori satisfied for
$\bu\in L^\infty(0,T;C^{0,1})$.

Brenier--De~Lellis--Sz\'ekelyhidi~\cite[Thm.~1]{BDLS2011}
establish weak--strong uniqueness for CMV solutions
against a Lipschitz reference: if $(\bw,\sigma)$
is a CMV Euler solution with $\bw(0) = \bu(0)$
and $\bu\in L^\infty(0,T;\mathrm{Lip}(\Omega))$,
then $\bw = \bu$ and $\sigma = 0$ for all $t$.
We sketch the relative-energy argument
to identify the role of $\sigma$ in the Gr\"onwall bound.
Define
\[
  \mathcal{E}(t) := \tfrac{1}{2}\nrm{\bw(t) - \bu(t)}_{L^2}^2
  + \tfrac{1}{2}\int_\Omega\operatorname{tr}(\sigma)(t).
\]
The CMV energy identity (\cref{def:CMV})
$\nrm{\bw(t)}_{L^2}^2 + \int\operatorname{tr}(\sigma)(t) = E_0$
combined with the energy conservation
$\nrm{\bu(t)}_{L^2}^2 = E_0$ just established
gives $\mathcal{E}(0) = 0$
(using $\bw(0) = \bu_0 = \bu(0)$, $\sigma(0) = 0$).
Testing the CMV equation~\eqref{eq:CMV_Euler} against $\bu$
and the strong Euler equation against $\bw$,
combining the two,
and using the matrix Cauchy--Schwarz inequality
$|\sigma_{ij}\partial_j u_i|
\le \nrm{\nabla\bu}_{L^\infty}\operatorname{tr}(\sigma)$
to bound the contribution of the defect measure,
one obtains the inequality
\[
  \ddt\mathcal{E}(t)
  \le C\,\nrm{\nabla\bu(t)}_{L^\infty}\,\mathcal{E}(t).
\]
Gr\"onwall and $\mathcal{E}(0) = 0$ give $\mathcal{E}(t) \equiv 0$,
hence $\bw = \bu$ and $\sigma = 0$.
\end{proof}

\begin{proposition}[Exclusion of dissipative Euler solutions]\label{prop:no_dissipative}%
No subsequence of discrete Euler solutions can converge
strongly in $L^2$ to a weak Euler solution $\bu$ with
$\nrm{\bu(t)}_{L^2} < \nrm{\bu_0}_{L^2}$ for some $t > 0$.
\end{proposition}

\begin{proof}
Discrete energy conservation and Hodge* consistency give
$\nrm{\bu^{h_k}(t)}_{L^2}^2 \to \nrm{\bu_0}_{L^2}^2$,
contradicting
$\nrm{\bu^{h_k}(t)}_{L^2} \to \nrm{\bu(t)}_{L^2} < \nrm{\bu_0}_{L^2}$.
\end{proof}

\begin{remark}[Commutativity of the inviscid and mesh-refinement limits]
\label{rem:commuting_limits}
Three paths connect $\bu^{h,\nu}$ to $\bu^0$:
$h\to 0$ then $\nu\to 0$;
$\nu\to 0$ then $h\to 0$;
and the diagonal $\nu = \nu(h)\to 0$.
For smooth solutions all three converge at $\OO(h^{\min(r_{\rm rec},\,r_\star)})$.
For Lipschitz solutions ($\alpha = 1$), Paths~1 and~2 reach the same limit by uniqueness.
Without regularity, each path produces a CMV solution,
reducing the double limit to the single PDE problem $\nu\to 0$.
\end{remark}

\subsubsection{Defect measure at the Onsager threshold}
\label{subsect:Duchon_Robert}

The unconditional convergence in \Cref{thm:Holder_unconditional}
requires Lipschitz regularity because the
MVWSU proof necessitates
 $\alpha \ge 1$.
We show that $\sigma$ vanishes once the discrete
solutions satisfy a uniform $C^{0,\alpha}$ bound with $\alpha > 1/3$:
the compact embedding $C^{0,\alpha}\hookrightarrow L^2$
upgrades the weak convergence from \Cref{thm:CMV}
to strong $L^2$ convergence, so the $L^2$~norms converge
$\nrm{\bw(t)}_{L^2}^2 = \lim_k\nrm{\bu^{h_k}(t)}_{L^2}^2 = E_0$;
the CMV energy identity then forces $\sigma = 0$.

\begin{theorem}[Vanishing of the defect measure at
the Onsager threshold]
\label{thm:sigma_vanish}%
Let $(\bu^h)_{h>0}$ be the Whitney-reconstructed discrete Euler
solutions from \Cref{thm:CMV}, satisfying the uniform H\"older bound
\begin{equation}\label{eq:uniform_Holder}
  \sup_{h>0}\,[\bu^h]_{L^\infty(0,T;\,C^{0,\alpha}(\Omega))} \le M
\end{equation}
for some $\alpha > 1/3$ and $M > 0$.
Then every subsequential CMV limit $(\bw,\sigma)$ satisfies
$\sigma(t) = 0$ for all $t\in[0,T]$,
and $\bu^{h_k}(t)\to\bw(t)$ strongly in $L^2(\Omega)$.
\end{theorem}

\begin{proof}
Let $\bu^{h_k}\to (\bw,\sigma)$ be a CMV subsequence
from \Cref{thm:CMV}.
The subsequence $h_k$ is fixed by \Cref{thm:CMV}.
For each $t$, the family $\{\bu^{h_k}(t)\}_k$ is uniformly
bounded in $C^{0,\alpha}(\Omega)$ by~\eqref{eq:uniform_Holder}.
Since $\Omega$ is compact, Arzel\`a--Ascoli gives a further
subsequence converging in $C^0(\Omega)$; the limit must
be $\bw(t)$ by uniqueness of the weak $L^2$ limit
(\cref{thm:CMV}(i)), so the full subsequence converges:
$\bu^{h_k}(t)\to\bw(t)$ in $C^0(\Omega)$ and hence in $L^2(\Omega)$.
In particular,
$\nrm{\bw(t)}_{L^2}^2 = \lim_k\nrm{\bu^{h_k}(t)}_{L^2}^2 = E_0$,
using discrete energy conservation
(\cref{thm:CMV}).
The uniform bound
$[\bu^{h_k}]_{C^{0,\alpha}}\le M$ is preserved in the
$C^0$~limit by lower semicontinuity:
$[\bw]_{C^{0,\alpha}}\le\liminf_k[\bu^{h_k}]_{C^{0,\alpha}}\le M$.
Hence $\bw\in L^\infty(0,T;\,C^{0,\alpha})$.
The CMV energy identity (\cref{thm:CMV}) gives
$\int\operatorname{tr}(\sigma)(t) = E_0 - \nrm{\bw(t)}_{L^2}^2 = 0$
by Step~1.
Since $\sigma\ge 0$, this implies $\sigma(t) = 0$ for all~$t$.
The convergence
$\nrm{\bu^{h_k}(t)}_{L^2}^2 \to E_0 = \nrm{\bw(t)}_{L^2}^2$
combined with
$\bu^{h_k}(t)\rightharpoonup\bw(t)$ (weak convergence, \Cref{thm:CMV}(i))
gives $\bu^{h_k}(t)\to\bw(t)$ strongly in $L^2$ by Radon--Riesz.
\end{proof}

\begin{remark}[The Onsager landscape]\label{rem:Onsager_revised}
Two H\"older thresholds organise the passage:
$\alpha \ge 1$ (Lipschitz) gives unconditional full-sequence convergence
with $\sigma = 0$ (\cref{thm:Holder_unconditional});
$1/3 < \alpha < 1$ gives $\sigma = 0$ (\cref{thm:sigma_vanish})
and strong subsequential convergence, with full-sequence convergence
conditional on the open problem of weak--strong uniqueness;
$\alpha \le 1/3$ allows anomalous dissipation, but dissipative
solutions are excluded (\cref{prop:no_dissipative}).
\end{remark}

\section{Extension to Bounded Domains with Dirichlet Boundary Conditions}
\label{sect:bounded_domains}

The preceding theory is developed on closed manifolds without boundary.
This section shows that all results extend to bounded domains
$\Omega \subset \mathbb{R}^3$ with homogeneous Dirichlet (no-slip)
conditions $\bu|_{\partial\Omega} = 0$.
The extension requires only two modifications: the function space $V_h$ is
replaced by $V_h^0$ (the subspace of cochains vanishing on boundary dual
edges), and the Leray projector $\bP_h$ by its boundary-adapted
counterpart $\bP_h^0$.
All conservation, well-posedness, and convergence proofs carry over
with these replacements; the only nontrivial new argument is the
boundary consistency estimate (\cref{app:bdy_consistency_proof}), which
shows that the truncation error at boundary cells is of the same order as
in the interior.

\begin{assumption}[Mesh Assumptions]
\label{ass:bdy_mesh}
In addition to \Cref{ass:mesh_reg}, we require:
\begin{enumerate}[nosep]
  \item \textit{Boundary alignment:}
    $\partial\Omega$ is a union of primal 2-cells (faces).
    Every primal face $f_j\subset\partial\Omega$ is called a {boundary face};
    its dual edge $e_j^*$ has one endpoint at the circumcentre of the
    unique interior prism $T_a$ with $f_j\prec T_a$, and the other endpoint
    at the boundary face $f_j$ itself (a ``half-edge'').
  \item \textit{Truncated dual cells:}
    Each truncated dual 3-cell of a boundary vertex remains star-shaped,
    so that the Hodge* ratios are well-defined and positive.
  \item \textit{Uniform regularity up to the boundary:}
    the mesh regularity constants in \Cref{ass:mesh_reg}
    hold uniformly for all cells, including those adjacent to $\partial\Omega$.
\end{enumerate}
\end{assumption}

We partition dual 1-edges into
$\mathcal{E}^* = \mathcal{E}^*_{\mathrm{int}} \;\dot\cup\; \mathcal{E}^*_{\partial}$,
where $\mathcal{E}^*_\partial = \{e_j^* : f_j\subset\partial\Omega\}$.
The no-slip condition becomes $v_j = 0$ for $e_j^*\in\mathcal{E}^*_\partial$.

\begin{definition}[Boundary-adapted spaces]\label{def:bdy_spaces}
$i)$ $C^1_0(\KKs) := \{\bv\in C^1(\KKs) : v_j = 0 \;\forall\; e_j^*\in\mathcal{E}^*_\partial\}$.
\quad
$ii)$ $V_h^0 := \ker(\bD_2\bM_1) \cap C^1_0(\KKs)$.
\quad
$iii)$ The {Leray projector} $\bP_h^0 : C^1_0(\KKs) \to V_h^0$
is defined by $\bP_h^0\bw = \bw - \tD_0\phi$, where $\phi$ solves
$L_h^0\phi = \bD_2\bM_1\bw$ with
$L_h^0 := \bD_2\bM_1\tD_0$ and Dirichlet conditions on $\phi$.
\end{definition}

The discrete Navier--Stokes system on a bounded domain is
\begin{equation}\label{eq:NS_proj_ODE_bdy}
  \ddt\bv = -\bP_h^0\Iv(\tD_1\bv) - \nu\Deltah\bv
  =: \bff_h^{0,\nu}(\bv),
  \qquad \bv(0) = \bv_0\in V_h^0.
\end{equation}

\begin{theorem}[Bounded-domain results]
\label{thm:global_bdy}
Under \Cref{ass:bdy_mesh}, all results from the closed-manifold
theory extend to bounded domains with $V_h\to V_h^0$,
$\bP_h\to\bP_h^0$:
\begin{enumerate}[nosep]
  \item \textit{Well-posedness.}
    For $\nu\ge 0$ and $\bv_0\in V_h^0$,
    \eqref{eq:NS_proj_ODE_bdy} has a unique global solution
    $\bv\in C^1([0,\infty);\,V_h^0)$, with energy
    $\Ekin(t)\le\Ekin_0$.  
  \item \textit{Consistency.}\label{thm:consistency_bdy}
    Under \Cref{ass:smooth_ref}, the truncation error satisfies
$ \sup_{t\in[0,T]}\nrm{\tau_h^{0,\nu}(t)}_{L_h^2}\le C\,h$;
    for $\nu=0$ the rate improves to $\OO(h^{r_{\rm rec}})$.
  \item \textit{Convergence.}\label{thm:convergence_bdy}
    Under \Cref{ass:smooth_ref}, the error satisfies
$ \sup_{t\in[0,T]}\nrm{\bv^h(t) - \mathcal{R}_h\bu^\flat(t)}_{L_h^2}
      \le C(T)\,h^{\min(r_{\rm rec},\,r_\star)}$,
    uniformly in $\nu\ge 0$.
  \item \textit{Leray--Hopf weak solutions.}\label{thm:weak_bdy}
    For $\bu_0\in L^2(\Omega)$ divergence-free with
    $\bu_0\cdot\bn = 0$ on $\partial\Omega$, the Whitney
    reconstructions $\bu^h = \mathcal{W}_h\bv^h$ converge
    (along a subsequence) strongly in $L^2(0,T;L^2)$ to a
    Leray--Hopf weak solution satisfying the energy inequality.
\end{enumerate}
\end{theorem}

\appendix
\section{Appendix: Proofs}\label{sect:appendix}

\subsection{Fundamental Properties }
\label{app:basic}

In this subsection we prove four fundamental properties
of DEC operators (\cref{prop:extrusion}):
cochain complex, energy identity, approximate Leibniz, and Cartan formula.

\begin{proof}[{\bf Proof of \Cref{prop:extrusion}}]
\textit{Property~1}.
The matrix $D_k$ (resp.\ $\tD_k$) is the signed incidence matrix of the
primal (resp.\ dual) cell complex: $(D_k)_{\sigma^{k+1},\sigma^k} =
[\sigma^{k+1}:\sigma^k]\in\{\pm 1,0\}$ records whether the $k$-cell $\sigma^k$ lies in
the boundary of the $(k+1)$-cell $\sigma^{k+1}$ with the induced
orientation.
Since the boundary of a boundary is empty, $\partial^2 = 0$, the product
$D_{k+1}D_k$ counts each interior $k$-cell twice with opposite signs and
therefore vanishes: $D_{k+1}D_k = 0$.
The dual identity $\tD_{k+1}\tD_k = 0$ follows by an analogous 
argument on the dual cell complex.

\medskip
\textit{Property~2}. (Energy identity, $\langle\bv, \Iv(\tD_1\bv)\rangle_1 = 0$).
We use the matrix form of \Cref{def:contraction}. Setting
$\bom = \tD_1\bv$ and pairing with $\bv$ on both sides yields
\begin{equation}\label{eq:energy_cancel}
  2\,\bv^T\bM_1\,\Iv(\tD_1\bv)
    = \langle \tU(\bv)\,\tD_1\bv, \bv   \rangle_1 - \langle \tD_1^T\,\tU(\bv)^T\bv, \bv \rangle_1 = 0,
\end{equation}
where $\tU(\bv)\in\mathbb{R}^{N_E\times N_F}$ is the velocity-weighted
matrix~\eqref{eq:Utilde_def} (linear in $\bv$),
$\tD_1\in\mathbb{R}^{N_F\times N_E}$ is the discrete curl, and
$\tD_1^T\in\mathbb{R}^{N_E\times N_F}$ its transpose.
The argument uses only the bilinearity of $\tU$ in $\bv$; no
specific reconstruction is required.

\medskip
\textit{Property~3}.
We define the de~Rham interpolants
$\tilde\alpha := \mathcal{R}_h\alpha$, $\tilde\beta := \mathcal{R}_h\beta$
of smooth 1-forms $\alpha,\beta$, and the Leibniz defect on a dual 3-cell $\sigma^*$  
\begin{equation}\label{eq:defect}
\mathcal{D}(\sigma^*):=
\Bigl(\tdd(\tilde\alpha\twdg\tilde\beta)
 - \tdd\tilde\alpha\twdg\tilde\beta
 - (-1)^p\tilde\alpha\twdg\tdd\tilde\beta\Bigr)(\sigma^*),
 \qquad p=q=1.
\end{equation}
We prove that
\begin{equation}\label{eq:defect_bound}
|\mathcal{D}(\sigma^*)|
 \le C\,h^2\,\nrm{\alpha}_{W^{1,\infty}}\nrm{\beta}_{W^{1,\infty}},
\end{equation}
with $C$ depending on the mesh-regularity.

The defect has two sources.
First, the discrete wedge product approximates
$\int_{\sigma^*}\alpha\wedge\beta$ using products of
face-centred values rather than integrals of pointwise products;
this quadrature-product error is present even if
all continuous values were known.
Second, the continuous vector proxies $\mathbf{a}(\mathbf{x}_j)$
are approximated by reconstructed values $\bar\bu_j(\tilde\alpha)$,
thus producing the reconstruction error.

\medskip
\noindent\textit{Step~1: Defect expansion.}
The product $\tilde\alpha\twdg\tilde\beta$ is a dual
2-cochain, so $\tdd(\tilde\alpha\twdg\tilde\beta)$
and the defect are dual 3-cochains.
We evaluate on a dual 3-cell $K_i^*$ (a primal prism in 3D).
The first term of the Leibniz defect \eqref{eq:defect} on $K_i^*$ expanded by discrete Stokes applied to the
2-cochain $\tilde\alpha\twdg\tilde\beta$
\begin{equation}\label{eq:term1}
\tdd(\tilde\alpha\twdg\tilde\beta)(K_i^*)
 = (\tilde\alpha\twdg\tilde\beta)(\partial K_i^*)
 = \sum_{f_k^*\prec K_i^*}
   [K_i^*:f_k^*]
   (\tilde\alpha\twdg\tilde\beta)(f_k^*),
\end{equation}
where the sum runs over dual 2-faces $f_k^*$ of $\partial K_i^*$
and $[K_i^*:f_k^*]=\pm 1$ is the incidence sign.
Applying the wedge definition~\eqref{eq:wedge_11} to each face gives
\begin{equation}\label{eq:term1b}
(\tilde\alpha\twdg\tilde\beta)(f_k^*)
= \sum_{e_j^*\prec f_k^*} w_{jk}(\tilde\alpha)\,\tilde\beta(e_j^*),
\end{equation}
where the extrusion-based wedge weight is
$
w_{jk}(\tilde\alpha)
 = \tfrac{1}{2}D_{1,jk}\,\bar\bu_j(\tilde\alpha)\cdot\hat{e}_k
$.
with $\bar\bu_j(\tilde\alpha)$ the velocity reconstructed from the
1-cochain $\tilde\alpha$ (\cref{def:averaging_recon}) and $D_{1,jk}$
the incidence coefficient.

We expand the second and third terms in~\eqref{eq:defect} analogously:
$\tdd\tilde\alpha$ and $\tdd\tilde\beta$ are 2-cochains, so
$\tdd\tilde\alpha\twdg\tilde\beta$ and $\tilde\alpha\twdg\tdd\tilde\beta$
are 3-cochains, and their evaluation on $K_i^*$ sums over
dual 2-faces of $K_i^*$, followed by the wedge formula on each face.
Collecting signs, the defect on $K_i^*$ is a double sum
over dual 2-faces $f_k^*\prec K_i^*$ and dual edges $e_j^*\prec f_k^*$:
\begin{equation}\label{eq:defect_expanded}
\mathcal{D}(K_i^*)
 = \sum_{f_k^*\prec K_i^*}[K_i^*\!:\!f_k^*]
   \sum_{e_j^*\prec f_k^*}
   \Bigl[
   w_{jk}(\tilde\alpha)\,\tilde\beta(e_j^*)
 - w_{jk}(\tdd\tilde\alpha)\,\tilde\beta(e_j^*)
 - w_{jk}(\tilde\alpha)\,\tdd\tilde\beta(e_j^*)
   \Bigr].
\end{equation}
Here the weight $w_{jk}(\tdd\tilde\alpha)$ is formed from the
2-cochain $\tdd\tilde\alpha$ via its vector proxy $(d\alpha)^\sharp$
(the vorticity pseudovector), which is the Hodge dual
of the 2-form $d\alpha$.

\medskip
\noindent\textit{Step~2: Decomposition.}
We define the true weight at dual edge $e_j^*$ by evaluating the
continuous vector proxies at the face centre $\mathbf{x}_j$:
\[
w_{jk}^{\rm ex}(\alpha) = \tfrac{1}{2}D_{1,jk}\,
 \mathbf{a}(\mathbf{x}_j)\cdot\hat{e}_k,
 \qquad
 w_{jk}^{\rm ex}(d\alpha) = \tfrac{1}{2}D_{1,jk}\,
 (d\alpha)^\sharp(\mathbf{x}_j)\cdot\hat{e}_k.
\]
The error of each weight is defined by
\[
\delta w_{jk} := w_{jk}(\tilde\alpha) - w_{jk}^{\rm ex}(\alpha)
\text{ and }  \delta w_{jk}^{d} := w_{jk}(\tdd\tilde\alpha) - w_{jk}^{\rm ex}(d\alpha).
\]
Substituting $w_{jk} = w_{jk}^{\rm ex} + \delta w_{jk}$ into the
inner sum of~\eqref{eq:defect_expanded} and expanding in
$\delta w_{jk}$, $\delta w_{jk}^d$, we obtain
\begin{equation}\label{eq:L_split}
\mathcal{D}(K_i^*)
 = \mathcal{D}^{\rm prod}(K_i^*)
 + \mathcal{D}^{\rm recon}(K_i^*)
 + \OO(h^3),
\end{equation}
where
\begin{align*}
\mathcal{D}^{\rm prod}(K_i^*)
 &:= \sum_{f_k^*\prec K_i^*}[K_i^*\!:\!f_k^*]
   \sum_{e_j^*\prec f_k^*}
   \Bigl(w_{jk}^{\rm ex}(\alpha)\,\tilde\beta(e_j^*)
 - w_{jk}^{\rm ex}(d\alpha)\,\tilde\beta(e_j^*)
 - w_{jk}^{\rm ex}(\alpha)\,\tdd\tilde\beta(e_j^*)
   \Bigr),\\
\mathcal{D}^{\rm recon}(K_i^*)
 &:= \sum_{f_k^*\prec K_i^*}[K_i^*\!:\!f_k^*]
   \sum_{e_j^*\prec f_k^*}
   \Bigl(\delta w_{jk}\,\tilde\beta(e_j^*)
 - \delta w_{jk}^d\,\tilde\beta(e_j^*)
 - \delta w_{jk}\,\tdd\tilde\beta(e_j^*)
   \Bigr).
\end{align*}
The product-rule error $\mathcal{D}^{\rm prod}$ uses face-centred
values and is independent of the reconstruction.
The reconstruction error $\mathcal{D}^{\rm recon}$ is proportional to
$\delta w_{jk}$ and vanishes for exact reconstructions.

\medskip
\noindent\textit{Step~3: Bounding the quadrature-product error.}
This error arises because
the product $w_{jk}^{\rm ex}(\alpha)\,\tilde\beta(e_j^*)$
approximates $\int_{e_j^*} \mathbf{a}\cdot\mathbf{b}$
at the midpoint $\mathbf{x}_j$ of $e_j^*$.
By a Taylor expansion on the edge $e_j^*$ of length $\OO(h)$
\[
\Bigl|\mathbf{a}(\mathbf{x}_j)\cdot\mathbf{b}(\mathbf{x}_j)\,\ell_j^*
  - \int_{e_j^*}\mathbf{a}\cdot\mathbf{b}\,\dd s\Bigr|
\le C\,\nrm{\nabla(\mathbf{a}\cdot\mathbf{b})}_{L^\infty}\,(\ell_j^*)^2
\le C\,\nrm{\nabla\alpha}_{W^{0,\infty}}\nrm{\nabla\beta}_{W^{0,\infty}}\,h^2.
\]
The inner sum in $\mathcal{D}^{\rm prod}(K_i^*)$ runs over at most
$C_\sigma$ dual edges $e_j^*$ for each $f_k^*$,
and the outer sum over at most $C_\sigma$ faces $f_k^*$.
Consequently
\[
|\mathcal{D}^{\rm prod}(K_i^*)|
 \le C_\sigma\,h^2\,\nrm{\alpha}_{W^{1,\infty}}\nrm{\beta}_{W^{1,\infty}}.
\]

\medskip
\noindent\textit{Step~4: Bounding the reconstruction error.}
The weight error $\delta w_{jk}$ is bounded by the reconstruction
accuracy.
The face velocity is $\bar\bu_j = \tfrac{1}{2}(\bu(v_a^*)+\bu(v_b^*))$
with $v_a^*, v_b^*$ the two dual vertices adjacent to face $f_j$.
By \Cref{prop:recon_accuracy}\,(i), each vertex reconstruction satisfies
$|\bu(v_i^*) - \mathbf{a}(v_i^*)| \le C_R h\,\nrm{\alpha}_{W^{1,\infty}}$,
so the face velocity shares the same bound.
Combined with $|D_{1,jk}|\le 1$, this gives
\begin{equation}\label{eq:w_error}
|\delta w_{jk}| = |w_{jk}(\tilde\alpha) - w_{jk}^{\rm ex}(\alpha)|
 \le C_1\,h\,\nrm{\alpha}_{W^{1,\infty}}.
\end{equation}
For the de~Rham interpolant we find
$|\tilde\beta(e_j^*)| \le \nrm{\beta}_{L^\infty}\,\ell_j^* \le C_2\,h\,\nrm{\beta}_{L^\infty}$,
since it is the integral of a bounded 1-form over an edge of length $\OO(h)$.
Each product satisfies $|\delta w_{jk}||\tilde\beta(e_j^*)| \le C_1 C_2 h^2$.
Summing over the boundary edges gives
\[
|\mathcal{D}^{\rm recon}(K_i^*)|
 \le C_r\,h^2\,\nrm{\alpha}_{W^{1,\infty}}\nrm{\beta}_{W^{1,\infty}}.
\]
Under reconstruction symmetry
\Cref{prop:recon_accuracy}\,(ii), we get
 $|\delta w_{jk}| \le C_1'h^2$,
so $\mathcal{D}^{\rm recon}(K_i^*) = \OO(h^3)$.

\medskip
Combining Step~3 and Step~4 implies
\[
  |\mathcal{D}(K_i^*)|
  \le |\mathcal{D}^{\rm prod}(K_i^*)| + |\mathcal{D}^{\rm recon}(K_i^*)|
  \le C\,h^2\,\nrm{\alpha}_{W^{1,\infty}}\nrm{\beta}_{W^{1,\infty}},
\]
which establishes~\eqref{eq:defect_bound}; the reconstruction-limited
contribution improves to $\OO(h^3)$ under reconstruction
symmetry~\eqref{eq:recon_symmetry}.


\textit{Property~4}.
The discrete Lie derivative is defined by Cartan's formula
$\tLie_{\bv} := \tdd\circ\Iv + \Iv\circ\tdd$ (\cref{def:Lie}),
and Property~4 holds by definition.
We verify that this definition is consistent with the geometric Lie
derivative at leading order.
For a smooth 1-form $\alpha$ and velocity field $\mathbf{u}$, the
continuous Cartan formula gives
$\mathcal{L}_{\mathbf{u}}\alpha = d(\iota_{\mathbf{u}}\alpha) +
\iota_{\mathbf{u}}(d\alpha)$.
For de~Rham interpolants $\tilde\alpha = \mathcal{R}_h\alpha$ and
$\bv = \mathcal{R}_h\bu^\flat$, the discrete Lie derivative satisfies
\[
(\tLie_{\bv}\tilde\alpha)(e_j^*)
 = \bigl(\tdd\Iv(\tilde\alpha) + \Iv(\tdd\tilde\alpha)\bigr)(e_j^*).
\]
We bound the two contributions separately:
(a)~$\Iv(\tilde\alpha) - \mathcal{R}_h(\iota_{\mathbf{u}}\alpha)$
is an $\OO(h^{r_{\rm rec}})$ edge error in the reconstruction of the
contraction, via \Cref{prop:recon_accuracy} combined with the bilinear quadrature in the wedge underlying
$\Iv$, \eqref{eq:wedge_11}. Applying $\tdd$ preserves the rate on dual 2-faces, 
$\OO(h^{r_{\rm rec}})$. (b)~$\tdd\tilde\alpha = \mathcal{R}_h(d\alpha)$
(\Cref{def:ext-deriv}). $\Iv(\tdd\tilde\alpha) - \mathcal{R}_h(\iota_{\mathbf{u}}d\alpha)$
is again $\OO(h^{r_{\rm rec}})$ by reconstruction analysis, now applied at $k=2$.
Combining and using the de~Rham property
$\mathcal{R}_h(d(\iota_{\mathbf{u}}\alpha) + \iota_{\mathbf{u}}d\alpha)
= \mathcal{R}_h(\mathcal{L}_{\mathbf{u}}\alpha)$:
\[
(\tLie_{\bv}\tilde\alpha)(e_j^*)
 = \mathcal{R}_h(\mathcal{L}_{\mathbf{u}}\alpha)(e_j^*)
   + \OO(h^{r_{\rm rec}+1}),
\]
where $h$ in the edge estimate is due to the
edge length $|e_j^*| = O(h)$, the de~Rham interpolant integrates a
1-form along the edge, and an $O(h^{r_{\rm rec}})$ pointwise consistency
error integrated over a length-$O(h)$ edge gives an
$O(h^{r_{\rm rec}+1})$ edge bound. This proves that the discrete
Lie derivative approximates the continuous Lie derivative to first order in $r_{\rm rec}$,
and second order under reconstruction symmetry.
\end{proof}

\subsection{Approximation Properties}
\label{app:approximation}

We prove the interpolation and Hodge-star consistency
of \Cref{lem:interp_error}. The identities (gradient
and curl consistency) follow from the de~Rham
commutativity of the discrete differential with the cochain
reconstruction; the divergence and Laplace consistency arises from the diagonal Hodge-star
approximation and inherit its rate.

\begin{proof}[{\bf Proof of \Cref{lem:interp_error}}]
\textit{Items~(i) and (ii): Gradient and curl consistency.}
Both identities follow from the de~Rham commutativity property
(\cref{def:ext-deriv}).

For (i): for any primal 0-cochain (dual edge) $e_j^*$ we have,
\[
  (\tD_0\mathcal{R}_h p)_j
  = (\mathcal{R}_h p)(\partial e_j^*)
  = \int_{\partial e_j^*} p
  = \int_{e_j^*} dp
  = (\mathcal{R}_h(dp))_j,
\]
so $\tD_0\mathcal{R}_h p = \mathcal{R}_h(dp)$ and the norm vanishes.

For (ii): for any dual 2-face $f_k^*$,
\[
  (\tD_1\mathcal{R}_h\bu^\flat)_k
  = (\mathcal{R}_h\bu^\flat)(\partial f_k^*)
  = \int_{\partial f_k^*}\bu^\flat
  = \int_{f_k^*}d(\bu^\flat)
  = (\mathcal{R}_h(d\bu^\flat))_k,
\]
so $\tD_1\mathcal{R}_h\bu^\flat = \mathcal{R}_h(d\bu^\flat)$.

\medskip
\textit{Item~(iii): Divergence consistency.}
For a smooth divergence-free field $\bu\in W^{r_\star,\infty}(\Omega)$,
the Hodge* converts the circulation cochain to a flux cochain:
$(\bM_1\mathcal{R}_h\bu^\flat)_j = (|f_j|/|e_j^*|)\int_{e_j^*}\bu\cdot d\ell$.
By \Cref{lem:hodge_error} ($k=1$), the face error is
\[
  \bigl|(\bM_1\mathcal{R}_h\bu^\flat)_j - \Phi_j\bigr|
  \le C_\star\,h^{r_\star}\nrm{\bu}_{W^{r_\star,\infty}}\,|f_j|.
\]
Applying $\bD_2$ to the correct fluxes gives us
$(\bD_2\boldsymbol\Phi)_i = \int_{\partial K_i}\bu\cdot d\mathbf{A}
= \int_{K_i}\nabla\cdot\bu\,dV = 0$ by the divergence theorem.
Therefore
\[
  |(\bD_2\bM_1\mathcal{R}_h\bu^\flat)_i|
  \le \sum_{f_j\prec K_i}|(\bM_1\mathcal{R}_h\bu^\flat)_j - \Phi_j|
  \le C_\star h^{r_\star}\nrm{\bu}_{W^{r_\star,\infty}}\sum_{f_j\prec K_i}|f_j|.
\]
Each prism $K_i$ has a uniformly bounded number of faces
(\cref{ass:mesh_reg}) and each face has area $|f_j| = \OO(h^{d-1})$,
so $\sum_{f_j\prec K_i}|f_j| = \OO(h^{d-1})$, giving the bound
\begin{equation}\label{eq:div_pointwise}
  |(\bD_2\bM_1\mathcal{R}_h\bu^\flat)_i|
  \le C\,h^{r_\star + d - 1}\nrm{\bu}_{W^{r_\star,\infty}}.
\end{equation}
In the $\ell^2$-norm, summing over $N_{\rm cells} = \OO(h^{-d})$ cells results in
\[
  \nrm{\bD_2\bM_1\mathcal{R}_h\bu^\flat}_{\ell^2}^2
  \le C^2 h^{2(r_\star+d-1)} N_{\rm cells}\nrm{\bu}_{W^{r_\star,\infty}}^2
  = C^2 h^{2r_\star + d - 2}\nrm{\bu}_{W^{r_\star,\infty}}^2.
\]
Hence $\nrm{\bD_2\bM_1\mathcal{R}_h\bu^\flat}_{\ell^2}
= \OO(h^{r_\star + (d-2)/2})$.
Since $(d-2)/2 \ge 0$ for $d\ge 2$, this is $\OO(h^{r_\star})$ as
claimed in~\eqref{eq:div_consistency}, where the actual bound is 
stronger for $d \ge 3$.

\medskip
\textit{Interpolation error~\eqref{eq:interp_error_v}.}
Define $e^\perp := (I-\bP_h)\mathcal{R}_h\bu^\flat$. Since $\bP_h$ is the
$\bM_1$-orthogonal projection onto $V_h = \ker(\bD_2\bM_1)$, we have
$e^\perp = \tD_0\phi\in V_h^\perp = \operatorname{range}(\tD_0)$ for the
potential $\phi$ of \cref{eq:Leray}, and $\bM_1$-orthogonality gives
\begin{equation}\label{eq:eperp_self}
  \nrm{e^\perp}_{L_h^2}^2
  = \ip{e^\perp}{(I-\bP_h)\mathcal{R}_h\bu^\flat}_1
  = \ip{e^\perp}{\mathcal{R}_h\bu^\flat}_1,
\end{equation}
the middle step uses $\ip{e^\perp}{\bP_h\mathcal{R}_h\bu^\flat}_1 = 0$.
Let $\boldsymbol\Phi$ be the  dual-flux $2$-cochain,
$\Phi_j := \int_{(\sigma_j^1)^*}\!\star\bu^\flat$, and define the
Hodge residual at a face by
\begin{equation}\label{eq:flux_residual}
  \br := \bM_1\mathcal{R}_h\bu^\flat - \boldsymbol\Phi.
\end{equation}
The fluxes are
$(\bD_2\boldsymbol\Phi)_i = \int_{\partial K_i}\bu\cdot\dd\mathbf A
= \int_{K_i}\nabla\!\cdot\!\bu\,\dd V = 0$, because
$e^\perp = \tD_0\phi$ and $\bD_2 = \pm\tD_0^T$, 
\begin{equation}\label{eq:Phi_orthogonal}
  (e^\perp)^T\boldsymbol\Phi
  = (\tD_0\phi)^T\boldsymbol\Phi
  = \pm\,\phi^T\bD_2\boldsymbol\Phi
  = 0.
\end{equation}
Substituting $\bM_1\mathcal{R}_h\bu^\flat = \boldsymbol\Phi + \br$ into
\eqref{eq:eperp_self} and using \eqref{eq:Phi_orthogonal},
\begin{equation}\label{eq:eperp_to_residual}
  \nrm{e^\perp}_{L_h^2}^2
  = \ip{e^\perp}{\bM_1^{-1}\br}_1
  \le \nrm{e^\perp}_{L_h^2}\,\nrm{\bM_1^{-1}\br}_{L_h^2},
\end{equation}
by Cauchy--Schwarz, where
$\nrm{\bM_1^{-1}\br}_{L_h^2}^2 = \nrm{\br}_{L_{h,1}^2}^2$.
The factor $\nrm{\br}_{L_{h,1}^2}$ is the global Hodge-error
estimate of \Cref{lem:hodge_error} (second bound, $k=1$):
\begin{equation}\label{eq:r_hodge}
  \nrm{\br}_{L_{h,1}^2}
  = \nrm{\bM_1\mathcal{R}_h\bu^\flat - \boldsymbol\Phi}_{L_{h,1}^2}
  \le C_\star\,h^{r_\star}\,\nrm{\bu}_{W^{r_\star,\infty}}.
\end{equation}
By dividing \eqref{eq:eperp_to_residual} by $\nrm{e^\perp}_{L_h^2}$
and inserting \eqref{eq:r_hodge} we obtain
\[
  \nrm{(I-\bP_h)\mathcal{R}_h\bu^\flat}_{L_h^2}
  \le C_\star\,h^{r_\star}\,\nrm{\bu}_{W^{r_\star,\infty}},
\]
i.e.\ \eqref{eq:interp_error_v} with $C_{\rm int} = C_\star$. 
The bound for \Cref{lem:proj_error}\,(i) is proven analogously.
\end{proof}
\begin{proof}[{\bf Proof of \Cref{lem:hodge_error}}]
Let $\alpha$ be a smooth 1-form and $f_j$ a primal face
with unit normal $\hat{n}_j$ and dual edge $e_j^*$.
We treat $k=1$ in detail (Steps~1--4, below) and handle $k=0,2$ in Step~5.
We compare the Hodge* action
$(\bM_1)_{jj}\,(\mathcal{R}_h\alpha)_j$ with the primal flux
$\Phi_j = \int_{f_j}\alpha(\hat{n}_j)\,dA$.

\medskip
\textit{Step~1: Error functional and cancellation on constants.}
The de~Rham map gives
$(\mathcal{R}_h\alpha)_j = \int_{e_j^*}\alpha(\hat{t}_j)\,ds$,
where $\hat{t}_j$ is the unit tangent to~$e_j^*$.
By local orthogonality \Cref{ass:mesh_reg} item~(3)
both integrands involve the same normal component $\alpha_n := \alpha(\hat{n}_j)$.
We define the error functional
\[
  E_j(\alpha) := \frac{|f_j|}{|e_j^*|}\int_{e_j^*}\alpha_n\,ds
  - \int_{f_j}\alpha_n\,dA.
\]
For a constant 1-form, $\alpha_n = c$,
$E_j = c(|f_j|/|e_j^*|)\cdot|e_j^*| - c|f_j|  = 0$.
Hence $E_j$ vanishes on constants, i.e.\ on $P_0(\Omega_j)$.

\medskip
\textit{Step~2: Bramble--Hilbert bound (general meshes).}
Let $\Omega_j := K_a\cup K_b$ be the union of the two primal
cells sharing face~$f_j$.
We verify the three hypotheses of the Bramble--Hilbert lemma:
(a)~ by the triangle inequality, $|E_j(\alpha)| \le 2|f_j|\|\alpha_n\|_{L^\infty(\Omega_j)}$,
so $E_j$ is bounded with constant $2|f_j|$;
(b)~$E_j$ vanishes on $P_0(\Omega_j)$ (Step~1);
(c)~$\Omega_j$ has diameter $\le C_{\rm qu}h$ and satisfies the
cone condition with radius $\ge\sigma_0 h$ (\cref{ass:mesh_reg}).
Bramble--Hilbert then gives
\begin{equation}\label{Bramble}
|E_j(\alpha)| \le C\,h\,|f_j|\,|\alpha_n|_{W^{1,\infty}(\Omega_j)}
  \le C_\star\,h\,|f_j|\,\nrm{\alpha}_{W^{1,\infty}},
\end{equation}
where the factor $|f_j|$ follows from the $L^\infty$ boundedness constant in (a).

\medskip
\textit{Step~3: Bramble--Hilbert bound (centroid proximity).}
Let $\alpha_n(\mathbf{x}) = \mathbf{c}\cdot\mathbf{x}+d$ be affine,
$\mathbf{c}: = \nabla\alpha_n$.
The midpoint rule is exact for affine integrands on line segments, so
$\frac{|f_j|}{|e_j^*|}\int_{e_j^*}\alpha_n\,ds = |f_j|\,\alpha_n(\bar{x}_{e_j^*})$,
where $\bar{x}_{e_j^*}$ is the midpoint of~$e_j^*$.
This implies
$\int_{f_j}\alpha_n\,dA = \alpha_n(\bar{x}_{f_j})\,|f_j|$,
and the error functional equates to,
\begin{equation}\label{eq:Ej_linear}
  E_j(\alpha)
  = |f_j|\,\bigl[\alpha_n(\bar{x}_{e_j^*}) - \alpha_n(\bar{x}_{f_j})\bigr]
  = |f_j|\,\mathbf{c}\cdot\bigl(\bar{x}_{e_j^*} - \bar{x}_{f_j}\bigr).
\end{equation}
Under \Cref{prop:centroid_proximity},
$|\bar{x}_{e_j^*} - \bar{x}_{f_j}| \le C_{\rm cg}\,h^2$
so
$|E_j(\alpha)| \le C_{\rm cg}h^2|\mathbf{c}||f_j|$.
For smooth $\alpha$, decompose into
$\alpha_n = c + \ell + \rho$, where $c = \alpha_n(\bar{x}_{f_j})$
is the centroid value (constant),
$\ell(\mathbf{x}) = \nabla\alpha_n(\bar{x}_{f_j})\cdot(\mathbf{x}-\bar{x}_{f_j})$
is the linear part, and
$\|\rho\|_{L^\infty(\Omega_j)} \le Ch^2|\alpha_n|_{W^{2,\infty}(\Omega_j)}$
is the quadratic Taylor remainder.
Then $E_j(c) = 0$ (Step~1),
$|E_j(\ell)| \le C_{\rm cg}h^2\|\nabla\alpha_n\|_{L^\infty}|f_j|$
 \eqref{eq:Ej_linear}, and
$|E_j(\rho)| \le 2|f_j|\|\rho\|_{L^\infty(\Omega_j)}
\le Ch^2|f_j||\alpha_n|_{W^{2,\infty}(\Omega_j)}$
 \eqref{Bramble}.
Combining gives
\[
|E_j(\alpha)| \le C_\star\,h^2\,\nrm{\alpha}_{W^{2,\infty}}\,|f_j|.
\]

\textit{Step~4: Global estimate.}
The error $\bM_1\mathcal{R}_h\alpha - \boldsymbol\Phi$ is a primal
2-cochain; we measure it in the $(\bM_1^{-1})$-norm
$\nrm{\mathbf{w}}_{L_{h,1}^2}^2 := \sum_j(\bM_1^{-1})_{jj}w_j^2
= \sum_j(|e_j^*|/|f_j|)w_j^2$ (denoted $L_{h,1}^2$ to distinguish it from the dual 1-cochain norm $L_h^2$).
With $r_\star = 1$ (Step~2) or $r_\star = 2$ (Step~3):
$|E_j| \le C_\star h^{r_\star}|f_j|\|\alpha\|_{W^{r_\star,\infty}}$, so
\begin{align*}
  \nrm{\bM_1\mathcal{R}_h\alpha - \boldsymbol\Phi}_{L_{h,1}^2}^2
  = \sum_j \frac{|e_j^*|}{|f_j|}|E_j|^2
  \le C_\star^2\,h^{2r_\star}\,\nrm{\alpha}_{W^{r_\star,\infty}}^2
      \sum_j |e_j^*|\,|f_j|.
\end{align*}
On a quasi-uniform mesh, $|e_j^*| = \OO(h)$, $|f_j| = \OO(h^{d-1})$,
and $N_{\rm edges} = \OO(h^{-d})$, so
$\sum_j|e_j^*||f_j| = \OO(h^d)\cdot\OO(h^{-d}) = \OO(1)$.
Hence $\nrm{\bM_1\mathcal{R}_h\alpha - \boldsymbol\Phi}_{L_{h,1}^2}
\le C_\star h^{r_\star}\nrm{\alpha}_{W^{r_\star,\infty}}$.

\medskip
\textit{Step~5: Cases $k = 0$ and $k = 2$.}
The proofs are analogous. For $k=0$, the scalar Hodge*, the error
functional $E_i(f) = |K_i^*|f(x_i^*) - \int_{K_i}f\,dV$ on a convex,
shape-regular primal cell~$K_i$ satisfies $E_i(1)=0$ and
$|E_i(f)| \le C|K_i|\|f\|_{L^\infty}$; Bramble--Hilbert implies
$|E_i(f)| \le Ch\,|K_i|\,|f|_{W^{1,\infty}}$,under centroid proximity this improves to
$Ch^2|K_i|\|f\|_{W^{2,\infty}}$.
For $k=2$, the vorticity Hodge*, the Delaunay--Voronoi orthogonality
aligns the integrands of $E_k(\omega)=(|e_k|/|f_k^*|)\int_{f_k^*}\omega
- \int_{e_k}\star\omega$ along the same normal component, so
Step~1--3 apply: $|E_k(\omega)|\le C_\star h|e_k|\|\omega\|_{W^{1,\infty}}$,
improving to $C_\star h^2|e_k|\|\omega\|_{W^{2,\infty}}$ under centroid
proximity. The global $L_{h,k}^2$ estimate follows by
summing as in Step~4.
\end{proof}

Because the mass-weighting estimate of Step~4 of the preceding proof
recurs in the convergence proofs, we state it as an independent lemma.

\begin{lemma}[Mass-weighted edge assembly]\label{lem:edge_assembly}
Let $\{a_j\}$ be a dual $1$-cochain on a 
Delaunay--Voronoi mesh satisfying \Cref{ass:mesh_reg}, with uniform
edge bound $|a_j|\le B\,h^{p}$ for some $p\in\mathbb{R}$ and all
edges $j$. Then, with the dual $1$-cochain norm
$\nrm{a}_{L_h^2}^2=\sum_j(\bM_1)_{jj}a_j^2$ and
$(\bM_1)_{jj}=|f_j|/|e_j^*|=\OO(h^{d-2})$ over $N_{\rm edges}=\OO(h^{-d})$
edges,
\[
  \nrm{a}_{L_h^2}\;\le\;C\,B\,h^{\,p-1},
\]
the  factor $h^{-1}$ arising as the square root of
$(\bM_1)_{jj}\cdot N_{\rm edges}=\OO(h^{-2})$. 
\end{lemma}
\begin{proof}
On a quasi-uniform mesh we have
$(\bM_1)_{jj}=|f_j|/|e_j^*|=\OO(h^{d-2})$ over $N_{\rm edges}=\OO(h^{-d})$
edges, so
\[
  \nrm{a}_{L_h^2}^2=\sum_j(\bM_1)_{jj}a_j^2
  \le B^2h^{2p}\,\bigl(\max_j(\bM_1)_{jj}\bigr)\,N_{\rm edges}
  =\OO\!\left(B^2h^{2p-2}\right),
\]
giving $\nrm{a}_{L_h^2}\le CBh^{p-1}$. The primal $2$-cochain statement
is identical with $(\bM_1^{-1})_{jj}=|e_j^*|/|f_j|$ and
$\sum_j|e_j^*||f_j|=\OO(1)$.
\end{proof}

\subsection{Conservation Properties}
\label{app:conservation}

\subsubsection{Energy Conservation}
\label{app:energy}

\begin{proof}[{\bf Proof of \Cref{thm:energy}}]
Differentiating the kinetic energy $\Ekin$ gives
\[
  \frac{d\Ekin}{dt} = \ip{\bv}{\ddt\bv}_1
  = -\ip{\bv}{\Iv(\bom)}_1
    -\ip{\bv}{\tD_0 B}_1=0.
\]
The first term vanishes by the energy identity
(\cref{prop:extrusion}, Property~2, \eqref{eq:antisym_11}).
The second vanishes due to  the SBP identity
and the incompressibility \cref{eq:D}.
\end{proof}

\subsubsection{Kelvin Circulation Theorem}
\label{app:kelvin}
\begin{proof}[Proof of \Cref{thm:kelvin}]
The circulation at time $t$ is given by the duality pairing
$\Gamma(t) =\bigl( \v(t),\gamma(t)\bigr) = \sum_j\alpha_j(t)\,v_j(t)$.
This pairing is bilinear in $(\v,\gamma)$. The Leibniz rule gives
\begin{equation}\label{eq:K1app}
  \ddt\bigl[\v\bigl(\gamma(t)\bigr)\bigr]
  = \bigl(\ddt{\v}\bigr)(\gamma) \;+\; \v\bigl(\ddt\gamma\bigr).
\end{equation}
By definition of material advection (\cref{def:circulation}, part~$ii$),
the second term becomes $\v\bigl(\tLie_{\v}^{\mathrm{chain}}\gamma\bigr)$.
The chain Lie derivative $\tLie_{\v}^{\mathrm{chain}}$ is defined as the
adjoint of the cochain Lie derivative with respect to the duality pairing:
$\alpha\bigl(\tLie_{\v}^{\mathrm{chain}}\gamma\bigr)
= (\tLie_{\v}\alpha)(\gamma)$
for $\alpha\in C^1(\KKs)$. This implies
\begin{equation}\label{eq:K2app}
  \v\bigl(\tLie_{\v}^{\mathrm{chain}}\gamma\bigr)
  = (\tLie_{\v}\v)(\gamma).
\end{equation}
The discrete Cartan formula (\cref{def:Lie}) applied to $\tLie_{\v}\v$ yields
\begin{equation}\label{eq:K3app}
  \tLie_{\v}\v
  = \tdd_0(\Iv\v) + \Iv(\tdd_1\v)
  = \tdd_0(\Iv\v) + \Iv(\bom),
\end{equation}
where $\Iv\v\in C^0(\KKs)$ is the 0-cochain from applying
the extrusion to $\bv$. 
We substitute the momentum equation~\eqref{eq:M}
and the Cartan decomposition~\eqref{eq:K3app}
into \eqref{eq:K1app}--\eqref{eq:K2app} and obtain
\begin{equation}\label{eq:K4app}
  \ddt\Gamma
  = \bigl(-\Iv(\bom) - \tdd_0 B\bigr)(\gamma)
  \;+\; \bigl(\tdd_0(\Iv\v) + \Iv(\bom)\bigr)(\gamma)
  = \bigl(\tdd_0(\Iv\v - B)\bigr)(\gamma).
\end{equation}
It suffices to show that $(\tdd_0\psi)(\gamma) = 0$ for a 0-cochain $\psi$.

We claim $\partial\circ\tLie_\bv^{\rm chain} = \tLie_\bv^{\rm chain}\circ\partial$.
By the definition of $\tLie_\bv^{\rm chain}$ (\cref{def:circulation})
and the discrete Stokes theorem ($\ip{\tdd_0\phi}{\gamma}_1 = \ip{\phi}{\partial\gamma}_0$)
it holds for any 0-cochain $\phi$,
\[
  \ip{\phi}{\partial(\tLie_\bv^{\rm chain}\gamma)}_0
  = \ip{\tdd_0\phi}{\tLie_\bv^{\rm chain}\gamma}_1
  = \ip{\tLie_\bv(\tdd_0\phi)}{\gamma}_1
  = \ip{\tdd_0(\tLie_\bv\phi)}{\gamma}_1
  = \ip{\tLie_\bv\phi}{\partial\gamma}_0
  = \ip{\phi}{\tLie_\bv^{\rm chain}(\partial\gamma)}_0,
\]
where the third equality uses $\tLie_\bv\circ\tdd = \tdd\circ\tLie_\bv$,
which follows from the Cartan formula and $\tdd^2=0$.
Since the identity holds for all 0-cochains $\phi$ the claim follows:
$\partial\tLie_\bv^{\rm chain} = \tLie_\bv^{\rm chain}\partial$. 
We show that $\partial\gamma(0) = 0$ implies $\partial\gamma(t) = 0$ for all $t$.
With $\ddt\gamma = \tLie_\bv^{\rm chain}\gamma$, we get
\[
  \ddt\bigl(\partial\gamma\bigr)
  = \partial\bigl(\ddt\gamma\bigr)
  = \partial\bigl(\tLie_\bv^{\rm chain}\gamma\bigr).
\]
Since $\partial\gamma(0) = 0$ and $\ddt(\partial\gamma) = \tLie_\bv^{\rm chain}(\partial\gamma)$
is a linear ODE with initial value zero, it follows $\partial\gamma(t) = 0$ for all $t$.
By discrete Stokes, for a 0-cochain $\psi$
\[
  (\tdd_0\psi)(\gamma) = \psi(\partial\gamma) = \psi(0) = 0.
\]
Setting $\psi = \Iv\v - B$ yields $\ddt\Gamma = 0$.
\end{proof}


\subsubsection{Helicity Conservation}
\label{app:helicity}

\begin{lemma}[Helicity consistency]\label{lem:helicity_consistency}
Let $\alpha$ be a smooth 1-form and $\beta$ a smooth 2-form.  Then
the reconstructions of \Cref{def:helicity} satisfy
\begin{equation}\label{eq:helicity_consistency}
  \nrm{\Qh^1\mathcal{R}_h\alpha - \alpha}_{L^2}
  \le C\,h^{r_{\rm rec}}\,\nrm{\alpha}_{W^{r_{\rm rec},\infty}},
  \qquad
  \nrm{\Qh^2\mathcal{R}_h\beta - \beta}_{L^2}
  \le C\,h^{r_{\rm rec}}\,\nrm{\beta}_{W^{r_{\rm rec},\infty}}.
\end{equation}
\end{lemma}

\begin{proof}
We prove the bound for $\Qh^1$, for $\Qh^2$ the argument is identical.
The reconstruction $\Qh^1\mathcal{R}_h\alpha$ is
piecewise linear on the simplicial refinement $\KK_h^{\rm simp}$ of
the DV mesh, which subdivides each Voronoi cell into simplices using cell
and face centres as auxiliary vertices. We estimate
$\nrm{\Qh^1\mathcal{R}_h\alpha - \alpha}_{L^2}$. 

We insert the standard
$P^1$ interpolant $I_K\alpha^\sharp$, the unique linear function on
simplex $K$ matching $\alpha^\sharp$ at the simplex vertices, which
lies in the same discrete space as $\Qh^1\mathcal{R}_h\alpha$. The
error splits into a $P^1$-vs-$P^1$ part, controlled by the
reconstruction error at the nodes via the maximum principle for
barycentric interpolation, and a $P^1$-vs-smooth part, controlled by
standard finite-element approximation theory.

On each simplex $K$ of the simplicial refinement $\KK_h^{\rm simp}$, the error decomposes as
\begin{equation}\label{error_eq}
  \Qh^1\mathcal{R}_h\alpha - \alpha
  = (\Qh^1\mathcal{R}_h\alpha - I_K\alpha^\sharp)
  + (I_K\alpha^\sharp - \alpha).
\end{equation}
For the first term in \Cref{error_eq}, the reconstruction error,
\Cref{prop:recon_accuracy} gives
$|\tilde\bu(v_j^*)-\alpha^\sharp(v_j^*)|
\le C_R\,h^{r_{\rm rec}}\,\nrm{\alpha}_{W^{r_{\rm rec},\infty}}$
at each vertex.
Since the $L^\infty$ norm of a barycentric interpolant on a simplex
is bounded by the maximum of its nodal values,
\[
\nrm{\Qh^1\mathcal{R}_h\alpha - I_K\alpha^\sharp}_{L^\infty(K)}
\le C_R\,h^{r_{\rm rec}}\,\nrm{\alpha}_{W^{r_{\rm rec},\infty}},
\]
we arrive at $\nrm{\Qh^1\mathcal{R}_h\alpha - I_K\alpha^\sharp}_{L^2(K)}
\le C_R\,h^{r_{\rm rec}}\,|K|^{1/2}\,\nrm{\alpha}_{W^{r_{\rm rec},\infty}}$.

\medskip
For the second term in \Cref{error_eq} we obtain
from standard approximation theory
$\nrm{I_K\alpha^\sharp - \alpha^\sharp}_{L^2(K)}
\le C\,h^2\,|K|^{1/2}\,\nrm{\alpha}_{W^{2,\infty}(K)}$.
This requires $W^{2,\infty}$-regularity; however, since $r_{\rm rec} \le 2$,
the reconstruction error $\OO(h^{r_{\rm rec}})$ dominates and equals or exceeds
the interpolation error $\OO(h^2)$, so the overall bound is
$\OO(h^{r_{\rm rec}}\nrm{\alpha}_{W^{r_{\rm rec},\infty}})$.
Squaring and summing over all simplices of the refinement
gives~\eqref{eq:helicity_consistency}.
\end{proof}

\begin{proof}[Proof of \Cref{thm:helicity}]
Differentiating $H_h(t) = \int_\Omega(\Qh^1\bv)\wedge(\Qh^2\bom)$ in
time and using $\bom = \tD_1\bv$, the Whitney--de~Rham commutativity
$\dd\circ\Qh^1 = \Qh^2\circ\tD_1$, and integrating by parts
\begin{align}
  \ddt H_h
  &= \int_\Omega(\Qh^1\dot\bv)\wedge(\Qh^2\bom)
    + \int_\Omega(\Qh^1\bv)\wedge(\Qh^2\tD_1\dot\bv)%
  = \int_\Omega(\Qh^1\dot\bv)\wedge(\Qh^2\bom)
    + \int_\Omega(\Qh^1\bv)\wedge\dd(\Qh^1\dot\bv) \notag\\
  &= \int_\Omega(\Qh^1\dot\bv)\wedge(\Qh^2\bom)
    + \int_\Omega\dd(\Qh^1\bv)\wedge(\Qh^1\dot\bv)%
  = \int_\Omega(\Qh^1\dot\bv)\wedge(\Qh^2\bom)
    + \int_\Omega(\Qh^2\bom)\wedge(\Qh^1\dot\bv)\notag\\
  &= 2\int_\Omega(\Qh^1\dot\bv)\wedge(\Qh^2\bom),
  \label{eq:Hh_rate_id}
\end{align}
where the third equality uses IBP on $\Omega$,
and the fourth uses wedge
commutation $\alpha^p\wedge\beta^q = (-1)^{pq}\beta^q\wedge\alpha^p$
with $p = q = 2$.

Since the Bernoulli term $\tD_0 B$ lies in the gradient subspace
$\tD_0(C^0(\KKs)) \perp_{\bM_1} V_h$, the projection
$\bP_h\tD_0 B = 0$ vanishes. Decomposing
$\bP_h\bQ(\bv,\bv) = \bQ(\bv,\bv) - \tD_0\psi$ for the discrete
pressure $\psi$,
\[
  \ddt H_h
  = -2\int_\Omega\Qh^1\bP_h\bQ(\bv,\bv)\wedge\Qh^2\bom
  = -2\int_\Omega\Qh^1\bQ(\bv,\bv)\wedge\Qh^2\bom
  + 2\int_\Omega\dd(\Qh^0\psi)\wedge\Qh^2\bom,
\]
using $\Qh^1\tD_0\psi = \dd\Qh^0\psi$. The pressure correction
vanishes by IBP,
$\int\dd(\Qh^0\psi)\wedge\Qh^2\bom = -\int\Qh^0\psi\wedge\dd\Qh^2\bom
= -\int\Qh^0\psi\wedge\Qh^3(\tD_2\bom) = 0$,
since $\tD_2\bom = \tD_2\tD_1\bv = 0$ by
\Cref{prop:extrusion} Property~1. Hence
\begin{equation}\label{eq:Hh_Lamb_only}
  \ddt H_h = -2\int_\Omega\Qh^1\bQ(\bv,\bv)\wedge\Qh^2\bom.
\end{equation}
We have the continuous identity
\begin{equation}\label{eq:Lamb_helicity_zero}
  \int_\Omega(\bom\times\bu)\wedge\bom^\flat = 0.
\end{equation}
Subtracting~\eqref{eq:Lamb_helicity_zero} from~\eqref{eq:Hh_Lamb_only}
and adding/subtracting the intermediate term
$\int(\bom\times\bu)\wedge\Qh^2\bom$:
\begin{align*}
  \tfrac{1}{2}\ddt H_h
  &= -\int_\Omega\bigl[\Qh^1\bQ(\bv,\bv) - \bom\times\bu\bigr]
    \wedge\Qh^2\bom
    -\int_\Omega(\bom\times\bu)\wedge\bigl[\Qh^2\bom - \bom^\flat\bigr].
\end{align*}
Cauchy--Schwarz and the consistency estimates
(\Cref{lem:helicity_consistency} and \Cref{prop:recon_accuracy})
give
\begin{align*}
  |\tfrac{1}{2}\ddt H_h|
  &\le \nrm{\Qh^1\bQ(\bv,\bv) - \bom\times\bu}_{L^2}\,
       \nrm{\Qh^2\bom}_{L^2}
   + \nrm{\bom\times\bu}_{L^2}\,
     \nrm{\Qh^2\bom - \bom^\flat}_{L^2} \\
  &\le C\,h^{r_{\rm rec}}\,\nrm{\bu}_{W^{r_{\rm rec},\infty}}^2
       \nrm{\bom}_{L^\infty}
     + C\,h^{r_{\rm rec}}\,\nrm{\bu}_{L^\infty}\nrm{\bom}_{L^\infty}
       \nrm{\bom}_{W^{r_{\rm rec},\infty}}.
\end{align*}
The first factor of the first term combines the reconstruction error
of $\bv$ ($O(h^{r_{\rm rec}})$ via \Cref{prop:recon_accuracy})
and the Whitney consistency of $\Qh^1$ applied to the Lamb
form, both of order $h^{r_{\rm rec}}$.

Under \Cref{conv:cases}: in case (A),
$r_{\rm rec} = 1$ giving $|\ddt H_h| = \OO(h)$;
in case (B), $r_{\rm rec} = 2$ giving $|\ddt H_h| = \OO(h^2)$.
The bound holds uniformly in $d = 2, 3$ and uses only
\eqref{eq:Hh_rate_id}--\eqref{eq:Lamb_helicity_zero},
together with reconstruction and Whitney consistency at level
$k = 1, 2$.
\end{proof}

\subsection{Finite-Dimensional Well-Posedness}
\label{app:wellposedness}

We show the Lipschitz continuity of the
Lamb vector (\cref{lem:Lip_extrusion}), from which short-time
(\cref{thm:picard}) and global (\cref{thm:global}) well-posedness
follow by standard ODE theory and the energy invariant.

\begin{proof}[Proof of \Cref{lem:Lip_extrusion}]
The map $\bv\mapsto\bQ(\bv,\bv):=\Iv(\tD_1\bv)$ is quadratic in $\bv$.
We claim there exists $C_Q(h) > 0$ such that
$\nrm{\bQ(\bu_1,\bu_2)}_2 \le C_Q(h)\,\nrm{\bu_1}_2\,\nrm{\bu_2}_2$
for $\bu_1,\bu_2\in\RR^N$.
For $\bQ$ it holds
$\nrm{\bM_1\bQ(\bu_1,\bu_2)}_2
\le \tfrac{1}{2}\bigl(\nrm{\tU(\bu_1)\,\tD_1\bu_2}_2
+ \nrm{\tD_1\,\tU(\bu_1)^T\bu_2}_2\bigr)$.
Each term is bounded by the spectral norms of $\tU(\bu_1)$
and $\tD_1$ .
Hence $C_Q(h)$ depends on $\nrm{\bM_1^{-1}}_2$, $\nrm{\tD_1}_2$.
For $\bv,\bw\in\RR^N$ with $\nrm{\bv}_2,\nrm{\bw}_2 \le R$,
the polarisation identity
gives
$  \bQ(\bv,\bv) - \bQ(\bw,\bw)
  = \bQ(\bv,\bv{-}\bw) + \bQ(\bv{-}\bw,\bw)$. 
Applying the bilinear bound yields
\[
\nrm{\bQ(\bv,\bv)-\bQ(\bw,\bw)}_2
  \le C_Q(h)\bigl(\nrm{\bv}_2+\nrm{\bw}_2\bigr)\nrm{\bv-\bw}_2
  \le 2R\,C_Q(h)\,\nrm{\bv-\bw}_2.
\]
This establishes the Lipschitz bound for the  extrusion.
Since $\nrm{\bP_h}_2\le 1$ (orthogonal projection), the Lipschitz continuity
for the Leray-projected map $\bff_h(\bv) = -\bP_h\bQ(\bv,\bv)$
follows with the same constant~$L_Q(h,R)$.
\end{proof}

\subsubsection{Global Well-Posedness
  (\texorpdfstring{\Cref{thm:picard,thm:global}}{})}

\begin{proof}[Proof of \Cref{thm:global} (Global well-posedness)]
Local existence and uniqueness on some interval $[0,T_*^h)$ follow from
applying the Picard--Lindel\"of theorem to the ODE $\ddt\bv = \bff_h(\bv)$
on the Banach space $(V_h,\nrm{\cdot}_2)$, and using \Cref{lem:Lip_extrusion}.
The solution can fail to be global only if $\nrm{\bv(t)}_2\to\infty$
as $t\nearrow T_*^h$ (blow-up in finite time).
Since $\bff_h(\bv) = -\bP_h\bQ(\bv,\bv)$ and $\bP_h$ projects
onto $V_h$, the energy identity (Property~2 of \Cref{prop:extrusion})
gives $\frac{d}{dt}\Ekin(t) = 0$, so
$\nrm{\bv(t)}_{\bM_1}^2 = 2\Ekin(0)$ for all $t\in[0,T_*^h)$.
Since $\nrm{\cdot}_{\bM_1}$ and $\nrm{\cdot}_2$ are equivalent norms
on the finite-dimensional space $V_h$, $\nrm{\bv(t)}_2$ remains bounded on $[0,T_*^h)$,
contradicting the blow-up assumption.
Hence $T_*^h = +\infty$.
\end{proof}

\subsection{Convergence Proofs}
\label{app:convergence}

In this subsection we provide the convergence analysis:
discrete Poincar\'e inequality (\cref{app:poincare}),
consistency estimate (\cref{app:consistency_proof}),
the nonlinear stability argument (\cref{app:stability_proof}),
and the trilinear estimate (\cref{app:trilinear_proof}).
The Navier--Stokes extensions and the Leray--Hopf weak convergence
follow in \cref{app:NS_stability_proof}--\cref{app:Leray_Hopf_proof}.
Three estimates occur repeatedly: an edge consistency bound of order
$h^{r_{\rm rec}+1}$, uniform in $d$
, the Helmholtz splitting
$e = \varepsilon + e^\perp$ used with the energy
identity, and a comparison of the Hodge star with the exact
dual-cell integral by means of the Whitney interpolation operator
$\mathcal{W}_h$.

\subsubsection{Discrete Poincar\'e Inequality}
\label{app:poincare}
\phantomsection\label{thm:poincare}%

\begin{proof}
The inequality is equivalent to a lower bound, uniform in $h$, on the lowest
non-zero eigenvalue $\lambda_1^h$ of $\Deltah$ over the non-harmonic
divergence-free cochains. Applying the Whitney interpolation operator
$\mathcal{W}_h$ to a cochain produces a piecewise-linear differential
form, and the Rayleigh quotient of the discrete Laplacian agrees with
the Rayleigh quotient of the continuous Laplacian evaluated on that form
up to a factor $1+\OO(h^{r_\star})$. The lower-order finite-element
Poincar\'e constant of~\cite{arnold2006,arnold2010}, which is uniform in
$h$, therefore transfers to $\lambda_1^h$; the convergence
$\lambda_1^h\to\lambda_1$ follows by evaluating Rayleigh quotients on
de~Rham interpolants of continuous eigenforms.

\smallskip
On a closed manifold $\Omega$, the space of divergence-free cochains
$V_h = \ker(\bD_2\bM_1)$ decomposes as
$V_h = \mathcal{H}_h^1 \oplus \widetilde V_h$, where
$\mathcal{H}_h^1 = \ker(\tD_1)\cap V_h$ is the space of discrete harmonic
1-cochains (curl-free and divergence-free) and
$\widetilde V_h$ is its $\bM_1$-orthogonal complement \emph{within $V_h$}.
Note that $\widetilde V_h$ is distinct from the $V_h^\perp$ used elsewhere
in the paper, which denotes the $\bM_1$-orthogonal complement of $V_h$
in the full cochain space, i.e.\ the gradient piece $\tD_0\phi$ from
the Helmholtz decomposition. On a simply connected $\Omega$,
$\mathcal{H}_h^1 = \{0\}$ and $\widetilde V_h = V_h$.
On a multiply connected closed $\Omega$, $\mathcal{H}_h^1$ is
finite-dimensional with $\dim\mathcal{H}_h^1 = \dim H^1(\Omega;\mathbb{R})$ uniformly
in $h$;
uses of \eqref{eq:disc_Poincare} are limited
to the non-harmonic component, which suffices for the convergence
analysis since the harmonic component of the velocity error is
controlled by the initial-data projection.
For ${e}\in \widetilde V_h$,
\begin{equation}\label{eq:disc_Poincare}
  \nrm{\tD_1{e}}_{\bM_2}^2 =
  \ip{e}{\Deltah{e}}_1 \ge
  \lambda_1^h\nrm{{e}}_{L_h^2}^2,
\end{equation}
where $\lambda_1^h$ is the smallest eigenvalue of $\Deltah$ on $\widetilde V_h$.  By the min-max principle for the self-adjoint operator $\Deltah$,
\[
  \lambda_1^h
  = \min_{{e}\in \widetilde V_h,\,\nrm{{e}}_{L_h^2}=1}
    \nrm{\tD_1{e}}_{\bM_2}^2
    = \min_{{e}\in \widetilde V_h\setminus\{0\}}
    \frac{\nrm{\tD_1{e}}_{\bM_2}^2}{\nrm{{e}}_{L_h^2}^2}
  > 0.
\]
We must show that $\lambda_1^h$ is bounded below independently of $h$ and that $\lambda_1^h\to\lambda_1$ as $h\to 0$.
The discrete Poincar\'e inequality for the lowest-order FEEC complex is
proved in \cite{arnold2006,arnold2010} with a constant $\lambda_1^{\rm FEEC}>0$
uniform in $h$.
On Delaunay--Voronoi meshes satisfying \Cref{ass:mesh_reg},
the DEC and FEEC Rayleigh quotients are uniformly comparable.
Let $\mathcal{W}_h : C^1(\KKs)\to \Lambda^1_h$ denote the Whitney lift from
dual 1-cochains to the lowest-order FEEC 1-form finite-element space.
For any $e\in C^1(\KKs)$, \Cref{lem:hodge_error} applied with $k=1$ gives
\[
  \bigl|\,\ip{e}{e}_1 - \nrm{\mathcal{W}_h e}_{L^2}^2\,\bigr|
  \le C_\star\,h^{r_\star}\,\nrm{\mathcal{W}_h e}_{L^2}^2,
  \qquad
  \bigl|\,\nrm{\tD_1 e}_{\bM_2}^2 - \nrm{\dd\,\mathcal{W}_h e}_{L^2}^2\,\bigr|
  \le C_\star\,h^{r_\star}\,\nrm{\dd\,\mathcal{W}_h e}_{L^2}^2,
\]
where in the second inequality we used the Whitney--de~Rham commutativity
$\dd\,\mathcal{W}_h e = \mathcal{W}_h(\tD_1 e)$ and \Cref{lem:hodge_error} at
level $k=2$. Hence both the $L^2$ and $H^1$ ratios satisfy
\[
  \frac{\ip{e}{e}_1}{\nrm{\mathcal{W}_h e}_{L^2}^2}
  = 1+\OO(h^{r_\star}),\qquad
  \frac{\nrm{\tD_1 e}_{\bM_2}^2}{\nrm{\dd\,\mathcal{W}_h e}_{L^2}^2}
  = 1+\OO(h^{r_\star}),
\]
uniformly in $e$ and $h$. Consequently the Rayleigh quotient
satisfies
$\nrm{\tD_1 e}_{\bM_2}^2 / \ip{e}{e}_1 \;\ge\; (1-C\,h^{r_\star})\cdot \nrm{\dd\,\mathcal{W}_h e}_{L^2}^2 / \nrm{\mathcal{W}_h e}_{L^2}^2$,
and the RHS is uniformly bounded below by $\lambda_1^{\rm FEEC}$ by the FEEC
Poincaré inequality, giving
\[
  \lambda_1^h \ge (1-C\,h^{r_\star})\,\lambda_1^{\rm FEEC} > 0
  \quad\text{uniformly in }h\le h_0.
\]
Spectral convergence $|\lambda_1^h - \lambda_1| \le C\,h$
follows by testing Rayleigh quotients with de~Rham
interpolants of continuous eigenforms \cite{dodziuk1976}.
\end{proof}



\subsubsection{Consistency (\texorpdfstring{\Cref{thm:consistency}}{})}
\label{app:consistency_proof}

\begin{proof}
We denote by $\bu\in C^1([0,T];W^{r_{\rm rec},\infty}(\Omega))$ the smooth
Euler solution of \Cref{ass:smooth_ref}. We estimate in three
steps. Step~1 bounds $\nrm{\bP_h\tau_h}_{L_h^2}$ through its pairing
against $\bw\in V_h$; the relations $\bP_h\tD_0 = 0$
 and $\bP_h\bw = \bw$ reduce the pairing to 
the Lamb consistency cochain
$\alpha := I_{\bar\bv}(\mathcal{R}_h\omega) - \mathcal{R}_h(\iota_\bu\omega)$,
the discrete interior product minus its true value, so that
$\nrm{\bP_h\tau_h}_{L_h^2}\le\nrm{\alpha}_{L_h^2}$ bounds the
consistency error. The contraction
(\Cref{def:contraction}) splits $\alpha = \tilde\alpha - \xi$, with
$\tilde\alpha$ the trapezoidal extrusion error and $\xi$ the
antisymmetrisation cochain, whose leading part is the discrete gradient
of kinetic energy. The two pieces require different estimators:
$\tilde\alpha$ is bounded in $L_h^2$ directly at rate $h^{r_{\rm rec}}$
(Step~2), whereas $\xi$ is of gradient size and only its projection
$\bP_h\xi$ is small, recovered by comparison with a discrete
gradient (Step~3). Both reductions rest on the same edge estimate:
a reconstruction or quadrature error $\OO(h^{r_{\rm rec}})$ against a
face-cochain factor $\OO(h^{d-1})$ provides a residual of
order $h^{r_{\rm rec}+d-1}$ on dual edges; the inverse Hodge-star weight
$(\bM_1^{-1})_{jj} = \OO(h^{2-d})$  converts this into a cochain estimate of order
$h^{r_{\rm rec}+1}$, uniformly in $d$;
\Cref{lem:edge_assembly}, (with $p = r_{\rm rec}+1$), implies the
$L_h^2$ rate $h^{r_{\rm rec}}$, using \Cref{conv:cases}. The dimensional
factors $h^{d-1}$ and $h^{2-d}$ cancel to a $d$-independent
power of $h$.

\medskip
\noindent\emph{Step~1: Reduction to cochain.}
We have $\ddt{\bar\bv} = \mathcal{R}_h(\partial_t\bu^\flat)$. After substituting the
continuous Euler equation
$\partial_t\bu^\flat = -\iota_\bu\omega - \dd B$ and using de~Rham
commutativity $\mathcal{R}_h(\dd B) = \tD_0\mathcal{R}_h B$ with
$\tD_1\bar\bv = \mathcal{R}_h\omega$, we obtain
\begin{equation}\label{eq:tau_h_expanded}
  \tau_h
  = \ddt{\bar\bv} + \bP_h I_{\bar\bv}(\tD_1\bar\bv)
  = -\mathcal{R}_h(\iota_\bu\omega)
    - \tD_0\mathcal{R}_h B
    + \bP_h I_{\bar\bv}(\mathcal{R}_h\omega).
\end{equation}
We test against $\bw\in V_h$ and expand as follows
\begin{align*}
\ip{\bw}{\bP_h\tau_h}_1
  &= \ip{\bw}{\tau_h}_1
   && \text{(self-adjointness, $\bw\in V_h$)}\\
  &= -\ip{\bw}{\mathcal{R}_h(\iota_\bu\omega)}_1
     - \ip{\bw}{\tD_0\mathcal{R}_h B}_1
     + \ip{\bw}{\bP_h I_{\bar\bv}(\mathcal{R}_h\omega)}_1
   && \text{(linearity)}\\
  &= -\ip{\bw}{\mathcal{R}_h(\iota_\bu\omega)}_1
     + \ip{\bw}{I_{\bar\bv}(\mathcal{R}_h\omega)}_1= \ip{\bw}{\alpha}_1
   && \text{(SBP, $\bw\in V_h$, idempotence)}
\end{align*}
Riesz representation in $V_h$ then implies
\begin{equation}\label{eq:dual_reduction}
  \nrm{\bP_h\tau_h}_{L_h^2}
  \;=\; \sup_{\bw\in V_h,\,\nrm{\bw}_{L_h^2}\le 1}\ip{\bw}{\alpha}_1
  \;\le\; \nrm{\alpha}_{L_h^2},
\end{equation}
this reduces the assertion to
\begin{equation}\label{eq:extrusion_consistency}
  \nrm{\alpha}_{L_h^2}
  \;\le\; C_I\,h^{r_{\rm rec}}\,
  \nrm{\bu}_{W^{r_{\rm rec}+1,\infty}}\bigl(\nrm{\bu}_{W^{r_{\rm rec}+1,\infty}} + \nrm{\omega}_{L^{\infty}}\bigr).
\end{equation}
The contraction (\Cref{def:contraction}) gives
\begin{equation}\label{eq:Iv_weighted}
  \bigl[I_{\bar\bv}(\mathcal{R}_h\omega)\bigr]_j
  \;=\; \tfrac{1}{2}\bigl(\bM_1^{-1}\tU\,\mathcal{R}_h\omega\bigr)_j - \xi_j,
  \qquad
  \xi_j := \tfrac{1}{2}\bigl(\bM_1^{-1}\tD_1\tU^T\bar\bv\bigr)_j,
\end{equation}
and with the trapezoidal extrusion error
\begin{equation}\label{eq:alpha_def}
  \tilde\alpha_j := \tfrac{1}{2}\bigl(\bM_1^{-1}\tU\,\mathcal{R}_h\omega\bigr)_j - \int_{e_j^*}\iota_\bu\omega,
\end{equation}
the Lamb cochain decomposes as
\begin{equation}\label{eq:alpha_decomp}
  \alpha \;=\; \tilde\alpha \;-\; \xi.
\end{equation}

\medskip
\noindent\emph{Step~2: Edge bound.}
Fix a dual edge $e_j^*$ of length $\ell_j^*$ connecting dual vertices
$v_a^*, v_b^*$. By the de~Rham definition,
$\int_{e_j^*}\iota_\bu\omega = (\mathcal{R}_h(\iota_\bu\omega))_j$.
The contraction term is computed as
\[
  \tfrac{1}{2}(\bM_1^{-1}\tU\,\mathcal{R}_h\omega)_j
  = (\bM_1^{-1})_{jj}\sum_{f_k^*\prec e_j^*}
    D_{1,jk}\,w_{jk}(\bar\bv)\,(\mathcal{R}_h\omega)(f_k^*),
\]
with $D_{1,jk}=\pm1$ the incidence, the diagonal weight
$(\bM_1^{-1})_{jj} = \ell_j^*/|f_j| = \OO(h^{2-d})$, and the extrusion weight
$w_{jk}(\bar\bv) = \tfrac{1}{2}\bigl(\tilde\bu(v_a^*)\cdot\hat{e}_k
+ \tilde\bu(v_b^*)\cdot\hat{e}_k\bigr)$
is the average of the reconstructed velocity projected
onto $\hat{e}_k$. The edge error splits into two
weights, the trapezoidal formula with exact vertex values gives,
\[
  w_{jk}^{\rm trap}
  := \tfrac{1}{2}\bigl(\bu(v_a^*)\cdot\hat{e}_k
  + \bu(v_b^*)\cdot\hat{e}_k\bigr),
\]
and the exact integral weight,
\[
  w_{jk}^{\rm exact}
  := \frac{1}{|e_j^*|}\int_{e_j^*}(\bu(x)\cdot\hat{e}_k)\,d\ell(x).
\]

\emph{Step~(2a) weight error.}
By \Cref{prop:recon_accuracy},
$|\tilde\bu(v_i^*) - \bu(v_i^*)|
\le C_R\,h^{r_{\rm rec}}\,\nrm{\bu}_{W^{r_{\rm rec},\infty}}$,
so
\[
  |w_{jk}(\bar\bv) - w_{jk}^{\rm trap}|
  \le \tfrac{1}{2}(|\tilde\bu(v_a^*)-\bu(v_a^*)|
    + |\tilde\bu(v_b^*)-\bu(v_b^*)|)
  \le C_R\,h^{r_{\rm rec}}\,\nrm{\bu}_{W^{r_{\rm rec},\infty}}.
\]
For the trapezoidal-vs-integral error, the classical bound implies
$|w_{jk}^{\rm trap} - w_{jk}^{\rm exact}|
\le C\,h^{r_{\rm rec}}\,\nrm{\bu}_{W^{r_{\rm rec},\infty}}$.
After combining we arrive at
\begin{equation}\label{eq:weight_error}
  |w_{jk}(\bar\bv) - w_{jk}^{\rm exact}|
  \le C\,h^{r_{\rm rec}}\,\nrm{\bu}_{W^{r_{\rm rec},\infty}}.
\end{equation}

\emph{Step~(2b) geometric approximation error.}
The difference between the exact-weight discrete formula
and the continuous interior-product integral is
\[
  \Delta_j^B := (\bM_1^{-1})_{jj}\sum_{f_k^*\prec e_j^*}
    D_{1,jk}\,w_{jk}^{\rm exact}\,(\mathcal{R}_h\omega)(f_k^*)
    - \int_{e_j^*}\iota_\bu\omega.
\]
Since $(\mathcal{R}_h\omega)(f_k^*) = \int_{f_k^*}\omega\cdot d\bA$ is exact
and
$w_{jk}^{\rm exact} = \frac{1}{|e_j^*|}\int_{e_j^*}\bu\cdot\hat{e}_k\,d\ell$,
the weighted product $(\bM_1^{-1})_{jj}\,w_{jk}^{\rm exact}\cdot(\mathcal{R}_h\omega)(f_k^*)$
approximates the contribution of face $f_k^*$ to the line integral
$\int_{e_j^*}\iota_\bu\omega$ by replacing the spatially varying
integrand with the product of its edge-average and its face-integral.
For constant fields $\bu\equiv\bu_0$, $\omega\equiv\omega_0$, the
factorisation is exact and $\Delta_j^B = 0$.

For non-constant fields, expand $\bu(x) = \bu(x_0)
+ (x - x_0)\cdot\nabla\bu(x_0) + \OO(|x - x_0|^2)$ around the midpoint
$x_0$ of $e_j^*$. The constant term equates to zero. The linear term contributes
$\sum_{f_k^*\prec e_j^*} D_{1,jk}\,(\nabla\bu(x_0)\cdot\hat e_k)
\cdot\int_{f_k^*}(x - x_0)\,\omega\cdot d\bA$,
where $D_{1,jk} = \pm 1$ is the incidence matrix. This
is the geometric quantity that appears in the
reconstruction-accuracy analysis of \Cref{prop:recon_accuracy}: the contribution
vanishes if and only if the reconstruction stencil at $e_j^*$
is symmetric about $x_0$, which is the reconstruction
symmetry~\eqref{eq:recon_symmetry}. When that condition
holds, the remainder is
controlled by the quadratic term, of size
$\OO(h^2)\nrm{\bu}_{W^{2,\infty}}|f_k^*|$. Without the condition, the
linear term survives at size $\OO(h)\nrm{\bu}_{W^{1,\infty}}|f_k^*|$.
We sum over the bounded number of faces incident to $e_j^*$
and get
\begin{equation}\label{eq:DeltaB_general}
  |\Delta_j^B|
  \le C\,(\bM_1^{-1})_{jj}\,h\,\nrm{\bu}_{W^{1,\infty}}\cdot
    \sum_{f_k^*\prec e_j^*}|(\mathcal{R}_h\omega)(f_k^*)|
  \le C\,h^{2-d}\,h\,h^{d-1}\,\nrm{\bu}_{W^{1,\infty}}\,\nrm{\omega}_{L^\infty}
  = C\,h^{2}\,\nrm{\bu}_{W^{1,\infty}}\,\nrm{\omega}_{L^\infty}.
\end{equation}
Under reconstruction symmetry~\eqref{eq:recon_symmetry}
and $\bu\in W^{2,\infty}$, the bound improves by one order
\begin{equation}\label{eq:DeltaB_sym}
  |\Delta_j^B|
  \le C\,h^{3}\,\nrm{\bu}_{W^{2,\infty}}\,\nrm{\omega}_{L^\infty}
  \qquad\text{(under~\eqref{eq:recon_symmetry})}.
\end{equation}
In both cases $|\Delta_j^B| = \OO(h^{r_{\rm rec}+1})$, uniformly in
$d$: $\OO(h^2)$ for $r_{\rm rec}=1$ on general meshes, and $\OO(h^3)$
for $r_{\rm rec}=2$ under symmetry; the face-cochain factor $h^{d-1}$
and the inverse Hodge-star weight $h^{2-d}$ become a dimension-independent power of $h$.

\medskip
The vorticity cochain satisfies
$|(\mathcal{R}_h\omega)(f_k^*)|
\le |f_k^*|\nrm{\omega}_{L^\infty}
= \OO(h^{d-1})\nrm{\omega}_{L^\infty}$.
The weight errors of Part~(2a) enter through
\[
  (\bM_1^{-1})_{jj}\sum_{f_k^*\prec e_j^*}
  |w_{jk}(\bar\bv) - w_{jk}^{\rm exact}|\,|(\mathcal{R}_h\omega)(f_k^*)|
  \le C\,h^{2-d}\cdot h^{r_{\rm rec}}\cdot h^{d-1}
  = C\,h^{r_{\rm rec}+1},
\]
and Part~(2b) contributes $|\Delta_j^B| \le C\,h^{r_{\rm rec}+1}$ by
\eqref{eq:DeltaB_general}--\eqref{eq:DeltaB_sym}. Combining the two
gives the uniform edge bound
\begin{equation}\label{eq:extrusion_pointwise}
  |\tilde\alpha_j|
  \le C\,h^{r_{\rm rec}+1}\,
  \nrm{\bu}_{W^{r_{\rm rec},\infty}}\nrm{\omega}_{L^\infty},
\end{equation}
uniformly in $d = 2,3$, with $C$ depending on mesh regularity, not on $h$ or $j$.

\smallskip
With the uniform edge bound~\eqref{eq:extrusion_pointwise} of order
$h^{r_{\rm rec}+1}$,
\Cref{lem:edge_assembly}, (with $p = r_{\rm rec}+1$) gives
\begin{equation}\label{eq:tildealpha_assembly}
  \nrm{\tilde\alpha}_{L_h^2}
  \;\le\; C_{\tilde\alpha}\,h^{r_{\rm rec}}\,
       \nrm{\bu}_{W^{r_{\rm rec},\infty}}\,\nrm{\omega}_{L^\infty}.
\end{equation}
\medskip
\noindent\emph{Step~3: Bound on $\bP_h\xi$.}
Since $\bP_h\tD_0 = 0$, $\bP_h\xi = \bP_h(\xi - \tD_0 f)$ for any
$0$-cochain $f$, and
$\nrm{\bP_h\xi}_{L_h^2}\le\nrm{\xi - \tD_0 f}_{L_h^2}$. 
We choose the de~Rham interpolant of the continuous
kinetic energy:
\begin{equation}\label{eq:f_choice_smooth}
  f \;:=\; \mathcal{R}_h\bigl(\tfrac{1}{2}|\bu|^2\bigr)\in C^0(\KKs),
  \qquad
  \bigl(\mathcal{R}_h\tfrac{1}{2}|\bu|^2\bigr)_i = \tfrac{1}{2}|\bu(v_i^*)|^2.
\end{equation}
Both objects are gradients, annihilated by $\bP_h$, but $\tD_0 f$
is the discrete gradient of a smooth continuous quantity, whose
vertex-to-vertex variation is governed by smoothness of $\bu$. 
Step~3 reduces to
\begin{equation}\label{eq:xi_approx_gradient}
  \nrm{\xi - \tD_0\mathcal{R}_h(\tfrac12|\bu|^2)}_{L_h^2}
  \;\le\; C_\xi\,h^{r_{\rm rec}}\,
  \nrm{\bu}_{W^{r_{\rm rec}+1,\infty}}^2.
\end{equation}
The bound goes through one intermediate weight: $\xi^{\rm c}$ is $\xi$ with
$\tilde\bu(v_i^*)$ replaced by $\bu(v_i^*)$, and the triangle
inequality splits~\eqref{eq:xi_approx_gradient} into a reconstruction
error ( first term) and a quadrature comparison (second term):
\begin{equation}\label{eq:xi_triangle}
  \nrm{\xi - \tD_0\mathcal{R}_h(\tfrac12|\bu|^2)}_{L_h^2}
  \;\le\;
\nrm{\xi - \xi^{\rm c}}_{L_h^2}
  +
\bigl\|\xi^{\rm c} - \tD_0\mathcal{R}_h\bigl(\tfrac{1}{2}|\bu|^2\bigr)\bigr\|_{L_h^2}.
\end{equation}
The difference $\xi - \xi^{\rm c}$ in \eqref{eq:xi_triangle} is
bilinear in $\bar\bv$ and is linearly dependent on the
reconstruction error
$\delta_i := \tilde\bu(v_i^*) - \bu(v_i^*)$, each bounded by
$|\delta_i| \le C h^{r_{\rm rec}}\nrm{\bu}_{W^{r_{\rm rec},\infty}}$
(\Cref{prop:recon_accuracy}). The matrix structure of $\xi$ gives
$(\xi - \xi^{\rm c})_j$ as $\tfrac12(\bM_1^{-1})_{jj}$ times a sum over
primal faces $f_k$ adjacent to $e_j^*$ of products
$(\delta\text{-error})\times(\text{face-cochain factor }|f_k^*|)$;
with $(\bM_1^{-1})_{jj} = \OO(h^{2-d})$ and $|f_k^*| = \OO(h^{d-1})$
the per-edge entry is of size $h^{r_{\rm rec}+1}$, uniformly in $d$.
\Cref{lem:edge_assembly} (with $p = r_{\rm rec}+1$) yields
\begin{equation}\label{eq:xi_recon}
  \nrm{\xi - \xi^{\rm c}}_{L_h^2}
  \;\le\; C\,h^{r_{\rm rec}}\,\nrm{\bu}_{W^{r_{\rm rec},\infty}}^2.
\end{equation}

For the second term in \cref{eq:xi_triangle} we obtain
from~\eqref{eq:Iv_weighted}
 $\xi^{\rm c}_j = \tfrac{1}{2}(\bM_1^{-1})_{jj}\sum_{f_k\succ e_j^*}
(\tD_1)_{jk}\,[\tU^{T,{\rm c}}\bu]_k$, where $f_k$ ranges over primal
faces incident to the dual edge $e_j^*$ and $[\tU^{T,{\rm c}}\bu]_k$
encodes the trapezoidal average of $\bu\cdot\hat e_l$ over the dual
edges $e_l^*\in \partial f_k^*$. 
For constant velocities the discrete Stokes implies
$\sum_{l\prec\partial f_k^*}\bigl[\bu\cdot\hat e_l\bigr]\,|e_l^*|
= \int_{\partial f_k^*}\bu^\flat = \int_{f_k^*}\dd\bu^\flat$. Taylor expansion of
$\bu(x)\cdot\hat e_l$ at edge midpoints reduces $[\tU^{T,{\rm c}}\bu]_k$
to
$\mathcal{R}_h(\tfrac{1}{2}|\bu|^2)(v_b^*) - \mathcal{R}_h(\tfrac{1}{2}|\bu|^2)(v_a^*)$
plus a Taylor remainder $r_k$ of the same shape as $\Delta_j^B$ in
Part~(2b):
$|r_k| \le C\,h^{r_{\rm rec}}\,|f_k^*|\,\nrm{\bu}_{W^{r_{\rm rec}+1,\infty}}^2$
Hence at dual edges it holds
$|\xi^{\rm c}_j - [\tD_0\mathcal{R}_h(\tfrac12|\bu|^2)]_j|
\le C\,h^{r_{\rm rec}+1}\,\nrm{\bu}_{W^{r_{\rm rec}+1,\infty}}^2$,
the same mass-weighted edge product as in (3a), converted to a cochain
by $\tfrac12(\bM_1^{-1})_{jj} = \OO(h^{2-d})$,
and \Cref{lem:edge_assembly} (with $p = r_{\rm rec}+1$) yields
\begin{equation}\label{eq:3b_bound}
  \bigl\|\xi^{\rm c} - \tD_0\mathcal{R}_h\bigl(\tfrac{1}{2}|\bu|^2\bigr)\bigr\|_{L_h^2}
  \;\le\; C\,h^{r_{\rm rec}}\,\nrm{\bu}_{W^{r_{\rm rec}+1,\infty}}^2.
\end{equation}

We substitute~\eqref{eq:xi_recon} and~\eqref{eq:3b_bound} into the
triangle inequality~\eqref{eq:xi_triangle} and get
~\eqref{eq:xi_approx_gradient}.
Combining with~\eqref{eq:alpha_decomp}, \eqref{eq:dual_reduction},
and the triangle inequality implies
$\nrm{\bP_h\alpha}_{L_h^2} \le \nrm{\tilde\alpha}_{L_h^2} + \nrm{\bP_h\xi}_{L_h^2}$,
\begin{equation}\label{eq:PhT_by_tildealpha}
  \nrm{\bP_h\tau_h}_{L_h^2}
  \;\le\; \nrm{\tilde\alpha}_{L_h^2}
  \;+\; C_\xi\,h^{r_{\rm rec}}\,\nrm{\bu}_{W^{r_{\rm rec}+1,\infty}}^2.
\end{equation}
Using the bound~\eqref{eq:tildealpha_assembly},
this yields~\eqref{eq:extrusion_consistency} and implies the assertion.
\end{proof}

\subsubsection{Euler Stability (\texorpdfstring{\Cref{thm:stability}}{})}
\label{app:stability_proof}

\begin{proof}
We denote by $r := r_{\rm rec}$ and $r_\star$ the consistency and
Hodge-star rates from \Cref{thm:consistency,lem:proj_error}. The
reference is  the de~Rham interpolant of the smooth solution $\bar\bv(t) := \mathcal{R}_h\bu^\flat(t)$, 
in general $\bar\bv\notin V_h$. The error $e := \bv^h - \bar\bv$ admits the
orthogonal decomposition
\begin{equation}\label{eq:helmholtz_split}
  e = \varepsilon + e^\perp,
  \qquad
  \varepsilon := \bP_h e \in V_h,
  \qquad
  e^\perp := (I-\bP_h)e = -(I-\bP_h)\bar\bv,
\end{equation}
where the second equality uses $\bv^h\in V_h$. By \Cref{lem:proj_error},
\begin{equation}\label{eq:eperp_bounds}
  \nrm{e^\perp}_{L_h^2} \le C_\pi\,h^{r_\star}\,\nrm{\bu}_{W^{r_\star,\infty}},
  \qquad
  \nrm{e^\perp}_{\rm rec} \le C_{\rm rec}\,h^{r_\star}\,|\log h|^{\beta_d}\,
  \nrm{\bu}_{W^{r_\star+1,\infty}}.
\end{equation}
Subtracting the equation $\dot{\bv}^h = \bff_h(\bv^h)$ with
$\bff_h = -\bP_h\bQ(\cdot,\cdot)$ from the truncation identity
$\dot{\bar\bv} = -\bP_h\bQ(\bar\bv,\bar\bv) + \tau_h$ and using
bilinearity of $\bQ$, we obtain the error equation
\begin{equation}\label{eq:error_ODE}
  \ddt e \;=\;
  -\bP_h\bigl[\,
    \underbrace{\bQ(\bar\bv,\,e)}_{\text{(I) smooth $\times$ error}}
    \;+\;
    \underbrace{\bQ(e,\,\bar\bv)}_{\text{(II) error $\times$ smooth}}
    \;+\;
    \underbrace{\bQ(e,\,e)}_{\text{(III) cubic self-interaction}}
  \,\bigr]
  \;-\; \tau_h.%
\end{equation}
Term~(I) can be estimated by a bilinear bound because the smooth
velocity $\bar\bv$ in the extrusion slot is pointwise-controlled.
Term~(II) lacks such pointwise control on $e$  and requires the trilinear estimate (\Cref{lem:trilinear}) and the
energy identity to convert the Helmholtz residual into a
smooth-velocity contribution. Term~(III) vanishes at leading order
by the energy identity $\ip{e}{\bQ(e,e)}_1 = 0$
(\Cref{prop:extrusion}\,(2)); the surviving Helmholtz residual is
again resolved by polarisation.

The proof has three steps. Step~1 derives the energy equation;
Step~2 bounds the three cross-terms in sub-steps 2a, 2b, 2c;
Step~3 closes by Gr\"onwall.

\bigskip
\noindent\textit{Step~1. Truncation residual.}
We take the $\bM_1$-inner product of~\eqref{eq:error_ODE} with $e$
\begin{equation}\label{eq:energy_balance}
  \tfrac12\ddt\nrm{e}_{L_h^2}^2
  \;=%
        -\ip{e}{\bP_h \bQ(\bar\bv,\,e)}_1 -  \ip{e}{\bP_h \bQ(e,\,\bar\bv)}_1 - \ip{e}{\bP_h \bQ(e,\,e)}_1
        \;-\; \ip{e}{\tau_h}_1.
\end{equation}
Next we bound the truncation residual: splitting
$\tau_h = \bP_h\tau_h + (I-\bP_h)\tau_h$, the first piece is
controlled by \Cref{thm:consistency}:
$\nrm{\bP_h\tau_h}_{L_h^2}\le C_\tau h^r$. For the second piece,
with $\bP_h\bQ(\bar\bv,\bar\bv)\in V_h$,
\[
  (I-\bP_h)\tau_h = (I-\bP_h)\dot{\bar\bv}
                 = (I-\bP_h)\mathcal{R}_h(\partial_t\bu^\flat),
\]
and \Cref{lem:proj_error} applied to the divergence-free field
$\partial_t\bu$ yields
$\nrm{(I-\bP_h)\tau_h}_{L_h^2} \le C h^{r_\star}\nrm{\partial_t\bu}_{W^{r_\star,\infty}}$.
Using $\ip{e}{(I-\bP_h)\tau_h}_1 = \ip{e^\perp}{(I-\bP_h)\tau_h}_1$
and Cauchy--Schwarz gives,
\begin{equation}\label{eq:tau_residual}
  \bigl|\ip{e}{\tau_h}_1\bigr|
  \;\le\; C_\tau\,h^r\,\nrm{e}_{L_h^2}
       \;+\; C\,h^{2r_\star}\,\nrm{\bu}_{C^1_t W^{r_\star,\infty}_x}.
\end{equation}
The first contribution is of order $h^r$, linear in
$\nrm{e}_{L_h^2}$; the second is of order $h^{2r_\star}$.

\bigskip
\noindent\textit{Step~2a. Term (I).}
The smooth velocity $\bar\bv$ in the extrusion slot satisfies
$|\bar\u_j(\bar\bv)| \le C\nrm{\bu}_{W^{1,\infty}}$ uniformly for
$h\le h_0 := 1/C_R$ (\Cref{prop:recon_accuracy}\,(i)). Combined with
bounded stencil size, and overlap (\Cref{ass:mesh_reg}),
this results in the bound
\begin{equation}\label{eq:bilinear_assembly}
  \nrm{\bQ(\bar\bv,\,\eta)}_{L_h^2}
  \;\le\; C_Q\,\nrm{\bu}_{W^{1,\infty}}\,\nrm{\eta}_{L_h^2}
  \qquad \forall\,\eta\in C^1(\KKs).
\end{equation}
Self-adjointness of $\bP_h$ and Cauchy--Schwarz imply,
\begin{equation}\label{eq:I_bound}
  \bigl|\ip{e}{\bP_h\bQ(\bar\bv,e)}_1\bigr|
  \;=\; \bigl|\ip{\varepsilon}{\bQ(\bar\bv,e)}_1\bigr|
  \;\le\; \nrm{\varepsilon}_{L_h^2}\,\nrm{\bQ(\bar\bv,e)}_{L_h^2}
  \;\le\; C_Q\,\nrm{\bu}_{W^{1,\infty}}\,\nrm{e}_{L_h^2}^2,
\end{equation}
using $\nrm{\varepsilon}_{L_h^2}\le\nrm{e}_{L_h^2}$ in the last step.

\bigskip
\noindent\textit{Step~2b. Term (II).}
The averaging reconstruction (\Cref{def:averaging_recon}) and the
extrusion weight (\Cref{def:contraction}) are continuous in the
pointwise-reconstruction norm with mesh-regularity-only constants,
$|\bar\u_j(\bv)|\le C_R\nrm{\bv}_{\rm rec}$ and
$|w_{jk}(\bv)|\le C_w\nrm{\bv}_{\rm rec}$. Combining
with~\eqref{eq:eperp_bounds},
\[
  |\bar\u_j(e^\perp)|
  \le C_R\,\nrm{e^\perp}_{\rm rec}
  \le C\,h^{r_\star}\,|\log h|^{\beta_d}\,\nrm{\bu}_{W^{r_\star+1,\infty}}.
\]
Repeating the argument of~\eqref{eq:bilinear_assembly} yields
\begin{equation}\label{eq:bilinear_eperp}
  \nrm{\bQ(e^\perp,\,\eta)}_{L_h^2}
  \;\le\; C\,h^{r_\star}\,|\log h|^{\beta_d}\,\nrm{\bu}_{W^{r_\star+1,\infty}}\,\nrm{\eta}_{L_h^2}.
\end{equation}
The logarithmic factor $|\log h|^{\beta_d}$ is present in both $d=2$
and $d=3$ ($\beta_2=\beta_3=1$); it is dimension-independent (\Cref{app:proj_error_proof}, part~(ii.3)) .

\smallskip
On a closed manifold with diagonal Hodge-star we have $V_h^\perp = \mathrm{ran}(\tD_0)$
(cf.~\eqref{H:bounded_projection}). Hence $e^\perp = \tD_0\phi$ for
 0-cochains $\phi$, and
\begin{equation}\label{eq:eperp_kills_curl}
  \tD_1 e^\perp = \tD_1\tD_0\phi = 0,
  \qquad \text{so}\qquad
  \bQ(\,\cdot\,,e^\perp) = \I_{(\cdot)}(\tD_1 e^\perp) \equiv 0.
\end{equation}
Define the trilinear form $\Phi(u,v,w) := \ip{u}{\bQ(v,w)}_1$. The
energy identity $\Phi(\bv,\bv,\bv) = 0$
(\Cref{prop:extrusion}\,(2)) holds for $\bv\in C^1(\KKs)$.
We substitute $\bv = aX+bY+cZ$,
where $X,Y,Z\in C^1(\KKs)$ are arbitrary cochains,
$a,b,c\in\mathbb{R}$.
Extracting the coefficient of $abc$ implies the identity
\begin{equation}\label{eq:polarised}
\begin{aligned}
  0 \;=\;& \Phi(X,Y,Z) + \Phi(X,Z,Y) + \Phi(Y,X,Z)
         + \Phi(Y,Z,X) + \Phi(Z,X,Y) + \Phi(Z,Y,X),
\end{aligned}
\end{equation}
the coefficient of $a^2b$ in the two-variable substitution gives
\begin{equation}\label{eq:polarised_two}
  0 \;=\; \Phi(X,X,Y) + \Phi(X,Y,X) + \Phi(Y,X,X).
\end{equation}
The hypothesis of \Cref{lem:trilinear} requires
$\bu \in W^{1,\infty}(\Omega;\Lambda^1)$ for the smooth velocity
producing the curl $\bar\bom = \tD_1\mathcal{R}_h\bu^\flat
= \tD_1\bar\bv$, and $e\in C^1(\KKs)$ for the discrete cochain. Both
are satisfied: the first by \Cref{ass:smooth_ref} (which provides
$\bu\in C([0,T];H^s)$ with $s\ge 3$, hence $W^{1,\infty}$ by Sobolev
embedding for $d\le 3$); the second by construction.

\smallskip
With these ingredients, we conclude by self-adjointness of $\bP_h$,
\[
  \ip{e}{\bP_h\bQ(e,\bar\bv)}_1
  \;=\; \ip{e}{\bQ(e,\bar\bv)}_1
        \;-\; \ip{e^\perp}{\bQ(e,\bar\bv)}_1.
\]
The first term is bounded by \Cref{lem:trilinear}:
$|\ip{e}{\bQ(e,\bar\bv)}_1|\le C_{\rm tri}\nrm{\bu}_{W^{1,\infty}}\nrm{e}_{L_h^2}^2$.
For the second term, one decomposes $e = \varepsilon + e^\perp$ in
the extrusion slot:
\begin{equation}\label{eq:Helmholtz_split_b}
  \ip{e^\perp}{\bQ(e,\bar\bv)}_1
  \;=\; \ip{e^\perp}{\bQ(\varepsilon,\bar\bv)}_1
        \;+\; \ip{e^\perp}{\bQ(e^\perp,\bar\bv)}_1.
\end{equation}
The second summand has $e^\perp$ in the extrusion slot and is
bounded by~\eqref{eq:bilinear_eperp} and Cauchy--Schwarz:
\begin{equation}\label{eq:Helmholtz_b_direct}
  \bigl|\ip{e^\perp}{\bQ(e^\perp,\bar\bv)}_1\bigr|
  \;\le\; C\,h^{2r_\star}\,\nrm{\bu}_{W^{r_\star+1,\infty}}\,\nrm{\bu}_{W^{1,\infty}}.
\end{equation}
The first summand has $\varepsilon$ in the extrusion slot, 
We apply~\eqref{eq:polarised} with
$(X,Y,Z) = (\varepsilon,\bar\bv,e^\perp)$. The terms
$\Phi(\varepsilon,\bar\bv,e^\perp)$ and
$\Phi(\bar\bv,\varepsilon,e^\perp)$ vanish
because of~\eqref{eq:eperp_kills_curl}, giving
\begin{equation}\label{eq:polarised_solve_b}
  \ip{e^\perp}{\bQ(\varepsilon,\bar\bv)}_1
  \;=\; -\ip{\varepsilon}{\bQ(e^\perp,\bar\bv)}_1
        \;-\; \ip{\bar\bv}{\bQ(e^\perp,\varepsilon)}_1
        \;-\; \ip{e^\perp}{\bQ(\bar\bv,\varepsilon)}_1.
\end{equation}
Each right-hand-side pairing now has either $\bar\bv$ or $e^\perp$ in the
extrusion slot.
Applying~\eqref{eq:bilinear_assembly} or~\eqref{eq:bilinear_eperp}, with
Cauchy--Schwarz gives
\begin{equation}\label{eq:Helmholtz_b_polarised}
  \bigl|\ip{e^\perp}{\bQ(\varepsilon,\bar\bv)}_1\bigr|
  \;\le\; C\,h^{r_\star}\,\nrm{\bu}_{W^{r_\star+1,\infty}}\,\nrm{\varepsilon}_{L_h^2}.
\end{equation}
Using the trilinear estimate
with~\eqref{eq:Helmholtz_b_direct},~\eqref{eq:Helmholtz_b_polarised},
results in
\begin{equation}\label{eq:II_bound}
  \bigl|\ip{e}{\bP_h\bQ(e,\bar\bv)}_1\bigr|
  \;\le\; C_{\rm tri}\nrm{\bu}_{W^{1,\infty}}\nrm{e}_{L_h^2}^2
       \;+\; C h^{r_\star}\nrm{\bu}_{W^{r_\star+1,\infty}}\nrm{e}_{L_h^2}
       \;+\; C h^{2r_\star},
\end{equation}
where smooth-solution norms have been absorbed into the constants in
the second and third terms.

\bigskip
\noindent\textit{Step~2c. Term (III).}
Through the energy identity $\ip{e}{\bQ(e,e)}_1 = 0$ and self-adjointness of
$\bP_h$ the cubic term becomes
\[
  \ip{e}{\bP_h\bQ(e,e)}_1
  \;=\; \ip{\varepsilon}{\bQ(e,e)}_1
  \;=\; \ip{e}{\bQ(e,e)}_1 - \ip{e^\perp}{\bQ(e,e)}_1
  \;=\; -\ip{e^\perp}{\bQ(e,e)}_1.
\]
Bilinearity implies $\bQ(e,e) = \bQ(\varepsilon,\varepsilon) +
\bQ(\varepsilon,e^\perp) + \bQ(e^\perp,\varepsilon) + \bQ(e^\perp,e^\perp)$.
Pairings of $e^\perp$ with $\bQ(\varepsilon,e^\perp)$ and with
$\bQ(e^\perp,e^\perp)$ vanish by~\eqref{eq:eperp_kills_curl}. For the
remaining two pairings we find with ~\eqref{eq:bilinear_eperp} and Cauchy--Schwarz 
\begin{align}
\bigl|\ip{e^\perp}{\bQ(e^\perp,\varepsilon)}_1\bigr|
&\leq 
 Ch^{2r_\star}\nrm{\bu}_{W^{r_\star+1,\infty}}\nrm{\varepsilon}_{L_h^2},\label{pairing_1}\\
\bigl|\ip{e^\perp}{\bQ(\varepsilon,\varepsilon)}_1\bigr|
&\leq 
Ch^{r_\star}\nrm{\bu}_{W^{r_\star+1,\infty}}\nrm{\varepsilon}_{L_h^2}^2,\label{pairing_2}
\end{align}
where we have used for the second estimate~\eqref{eq:polarised_two} with
$(X,Y) = (\varepsilon,e^\perp)$. After one eliminates
$\Phi(\varepsilon,\varepsilon,e^\perp) = 0$
via~\eqref{eq:eperp_kills_curl},
$\ip{e^\perp}{\bQ(\varepsilon,\varepsilon)}_1
= -\ip{\varepsilon}{\bQ(e^\perp,\varepsilon)}_1$.
Then~\eqref{eq:bilinear_eperp} yields \Cref{pairing_2}.

\noindent
By combining, and  using $\nrm{\varepsilon}_{L_h^2}\le\nrm{e}_{L_h^2}$ we have
\begin{equation}\label{eq:III_bound}
  \bigl|\ip{e}{\bP_h\bQ(e,e)}_1\bigr|
  \;\le\; C h^{r_\star}\nrm{\bu}_{W^{r_\star+1,\infty}}\nrm{e}_{L_h^2}^2
       \;+\; C h^{2r_\star}\nrm{\bu}_{W^{r_\star+1,\infty}}\nrm{e}_{L_h^2}.
\end{equation}

\bigskip
\noindent\textit{Step~3. Gr\"onwall closure.}
From ~\eqref{eq:tau_residual},~\eqref{eq:I_bound},~\eqref{eq:II_bound},~\eqref{eq:III_bound},
follows for the energy balance~\eqref{eq:energy_balance}
\begin{equation}\label{eq:gronwall_inequality}
  \tfrac12\ddt\nrm{e}_{L_h^2}^2
  \;\le\; C_L(h)\,\nrm{e}_{L_h^2}^2 \;+\; F(h)\,\nrm{e}_{L_h^2} \;+\; G(h),
\end{equation}
where, absorbing smooth-solution norms into constants,
\begin{align}
  C_L(h) &:= (C_Q + C_{\rm tri})\nrm{\bu}_{W^{1,\infty}}
          + C_1 h^{r_\star}\nrm{\bu}_{W^{r_\star+1,\infty}}, \label{eq:CL_def}\\
  F(h)   &:= C_\tau h^r + C_2 h^{r_\star}\nrm{\bu}_{W^{r_\star+1,\infty}}
          + C_3 h^{2r_\star}\nrm{\bu}_{W^{r_\star+1,\infty}}, \label{eq:F_def}\\
  G(h)   &:= C_4 h^{2r_\star}\nrm{\bu}_{C^1_t W^{r_\star+1,\infty}_x}^2.
              \label{eq:G_def}
\end{align}
Set
$h_1 := \bigl(\nrm{\bu}_{W^{1,\infty}}/(C_1\nrm{\bu}_{W^{r_\star+1,\infty}})\bigr)^{1/r_\star}$.
For $h\le \min(h_0, h_1)$, the second summand in~\eqref{eq:CL_def}
is bounded by $\nrm{\bu}_{W^{1,\infty}}$, hence
\begin{equation}\label{eq:tildeCL_def}
  C_L(h)\;\le\;\tilde C_L
  \;:=\;(C_Q + C_{\rm tri} + 1)\nrm{\bu}_{C([0,T];W^{1,\infty})},
\end{equation}
which is independent of $h$ and depends only on the smooth solution
through $\nrm{\bu}_{C([0,T];W^{1,\infty})}$ and on mesh constants.

\smallskip
From Young's inequality
follows for~\eqref{eq:gronwall_inequality}
\begin{equation}\label{eq:gronwall_standard}
  \ddt\nrm{e}_{L_h^2}^2
  \;\le\; (2\tilde C_L + 1)\,\nrm{e}_{L_h^2}^2
       + \bigl(F(h)^2 + 2G(h)\bigr).
\end{equation}
Gr\"onwall's inequality yields
\begin{equation}\label{eq:gronwall_final}
  \nrm{e(t)}_{L_h^2}^2
  \;\le\; e^{(2\tilde C_L+1)t}\Bigl[\nrm{e(0)}_{L_h^2}^2
            + \bigl(F(h)^2 + 2G(h)\bigr)\,t\Bigr].
\end{equation}
Set $r^\ast := \min(r_{\rm rec},r_\star)$. By
$(a+b+c)^2 \le 3(a^2+b^2+c^2)$,
\[
  F(h)^2
  \;\le\; 3\bigl(C_\tau^2 h^{2r_{\rm rec}}
              + C_2^2 \nrm{\bu}^2 h^{2r_\star}
              + C_3^2 \nrm{\bu}^2 h^{4r_\star}\bigr).
\]
The exponents satisfy $2r_{\rm rec}\ge 2r^\ast$,
$2r_\star\ge 2r^\ast$, and $4r_\star\ge 2r_\star\ge 2r^\ast$ since
$r_\star\ge 1$ under \Cref{conv:cases}. For $h\le 1$ each term is
bounded by a constant times $h^{2r^\ast}$, and similarly
$2G(h)\le C h^{2r_\star}\le C h^{2r^\ast}$. Hence
\begin{equation}\label{eq:F2G_rate}
  F(h)^2 + 2G(h) \;=\; \OO(h^{2r^\ast}).
\end{equation}
For $\bv^h(0) = \bP_h\bar\bv(0)$, $e(0) = -e^\perp(0)$ and
$\nrm{e(0)}_{L_h^2}\le C_0 h^{r_\star}\le C_0 h^{r^\ast}$
by~\eqref{eq:eperp_bounds}. Substituting
into~\eqref{eq:gronwall_final},
\[
  \sup_{t\in[0,T]}\nrm{e(t)}_{L_h^2}
  \;\le\; C(T)\,h^{r^\ast},
\]
with
$C(T) = e^{(\tilde C_L+1/2)T}\sqrt{C_0^2 + (\tilde F^2 + 2\tilde G)T}$
depending on $T$,
$\nrm{\bu}_{C^1([0,T];W^{r_\star+1,\infty})}$, and mesh regularity,
but not on $h$. Under \Cref{conv:cases}, $r^\ast = 1$ in case~(A) and
$r^\ast = 2$ in case~(B).
\end{proof}

\subsubsection{Trilinear Estimate (\texorpdfstring{\Cref{lem:trilinear}}{})}
\label{app:trilinear_proof}

\begin{proof}
The trilinear form
$I_{\rm cross} = \ip{e}{\bQ(e, \bar\bv)}_1$ has the discrete error
$e$ in two slots; the smooth $\bar\bv$ supplies pointwise control on
$\bar\omega = \tD_1\bar\bv$. We use Cauchy--Schwarz over dual edges
with two factors: the first via mesh regularity
$\ell_j^* \ge c_\ell h$ to relate the edge-scaled error norm to
$\nrm{e}_{L_h^2}$; the second via a Gram inequality for the
averaging reconstruction to relate pointwise reconstructed velocities at dual
vertices back to $\nrm{e}_{L_h^2}$. The two $h^{-1/2}$-factors combine with 
the inner $h^{d-1}$ from the face-area sum to
give the final $h^{d-2}$ scaling: bounded in $d=2$, decaying in $d=3$.

\smallskip
We define $\epsilon_l := e_l/\ell_l^*$
and $\Omega_k := \bar\omega_k/|f_k^*|$,
so that $e_l = \epsilon_l\,\ell_l^*$ and $\bar\omega_k = \Omega_k\,|f_k^*|$.
The variable $\epsilon_l$ has the dimensions of velocity, the
de~Rham interpretation of $e_l$ being a line integral over $e_l^*$;
analogously $\Omega_k$ has the dimensions of vorticity, the
de~Rham interpretation of $\bar\omega_k$ as flux integral over
$f_k^*$. By de~Rham commutativity,
$|\Omega_k| = |\bar\omega_k|/|f_k^*|
\le \nrm{\dd\bu^\flat}_{L^\infty}
\le \nrm{\bu}_{W^{1,\infty}}$,
because the smooth vorticity bounds the dual-face average.
The trilinear form is
$I_{\rm cross}
= \sum_j(\bM_1)_{jj}\,e_j\sum_k w_{jk}({e})\,\bar\omega_k$.
Using $(\bM_1)_{jj} = A_j/\ell_j^*$ where $A_j := |f_j|$ is the primal
face area, so $(\bM_1)_{jj}\,e_j = A_j\epsilon_j$:
\begin{equation}\label{eq:bound_tri}
  I_{\rm cross}
  = \sum_j A_j\,\epsilon_j
    \sum_k w_{jk}({e})\,\Omega_k\,|f_k^*|.
\end{equation}
The extrusion weight $w_{jk}({e})$ is, by \Cref{def:contraction},
determined by the averaging reconstruction (\cref{def:averaging_recon}).
The reconstruction at a dual vertex $v_i^*$ of cell $K_i^*$ is
$\tilde\bu({e})(v_i^*)
= G_i^{-1}\sum_{e_n^*\prec K_i^*}
  \epsilon_n\,\hat{t}_n$,
where $G_i = \sum_{n\prec K_i^*}\hat{t}_n\otimes\hat{t}_n$ is the Gram
matrix at $v_i^*$, symmetric positive definite with $\nrm{G_i^{-1}} \le C_G$
under \Cref{ass:mesh_reg}. The extrusion weight is the trapezoidal
average of $\tilde\bu({e})\cdot\hat{e}_k$ over the two endpoints of
the dual edge $e_j^*$, projected onto the edge $\hat{e}_k$
of the dual face $f_k^*$:
\begin{equation}\label{eq:weight_bound_tri}
  |w_{jk}({e})|
  \le C_w\,\max_{i\sim j}|\tilde\bu({e})(v_i^*)|,
\end{equation}
where $C_w$ depends only on the mesh regularity constant and $i\sim j$
denotes that $v_i^*$ is one of the two endpoints of $e_j^*$.

\smallskip
Since 
$|\Omega_k| \le \nrm{\bu}_{W^{1,\infty}}$,
and
$\sum_{k\prec j}|f_k^*| \le C_{\rm st}\,h^{d-1}$:
\[
  \Bigl|\sum_k w_{jk}({e})\,\Omega_k\,|f_k^*|\Bigr|
  \le C_1\,\nrm{\bu}_{W^{1,\infty}}\,h^{d-1}\,
      \max_{i\sim j}|\tilde\bu({e})(v_i^*)|.
\]
We substitute into~\eqref{eq:bound_tri} and apply
Cauchy--Schwarz over $j$ with weights $A_j$
\begin{equation}\label{eq:Icross_CS_incomp}
  |I_{\rm cross}|
  \le C_1\,\nrm{\bu}_{W^{1,\infty}}\,h^{d-1}
    \Bigl(\sum_j A_j\epsilon_j^2\Bigr)^{1/2}
    \Bigl(\sum_j A_j
      \bigl(\max_{i\sim j}|\tilde\bu({e})(v_i^*)|
      \bigr)^2\Bigr)^{1/2}.
\end{equation}
For the first factor in \eqref{eq:Icross_CS_incomp} we have with
$\nrm{{e}}_{L_h^2}^2
= \sum_j(\bM_1)_{jj}e_j^2
= \sum_j A_j\ell_j^*\epsilon_j^2$
and $\ell_j^* \ge c_\ell h$
\begin{equation}\label{eq:first_factor}
  \sum_j A_j\epsilon_j^2
  \le (c_\ell h)^{-1}\nrm{{e}}_{L_h^2}^2.
\end{equation}
The factor $(c_\ell h)^{-1}$ contributes one $h^{-1/2}$ to the final estimate.

For the second factor in \eqref{eq:Icross_CS_incomp} it holds
by the Gram-inverse bound and Cauchy--Schwarz on the cell sum
\begin{equation}\label{eq:avg_recon_bound}
  |\tilde\bu({e})(v_i^*)|^2
  \le C_G^2\,C_{\rm val}
    \sum_{l\prec K_i}\epsilon_l^2,
\end{equation}
where $C_{\rm val}$ is the maximal number of dual edges
incident to a dual vertex. Each cell $K_i$ is adjacent to at most
$C_{\rm st}$ edges $j$, and $A_j = \OO(h^{d-1})$.
Summing over $j$ and exchanging the order of summation:
\[
  \sum_j A_j\bigl(\max_{i\sim j}|\tilde\bu({e})(v_i^*)|
  \bigr)^2
  \le C_2\,h^{d-1}
      \sum_l\epsilon_l^2.
\]
To convert $\sum_l\epsilon_l^2$ back to $\nrm{e}_{L_h^2}$, we observe that
$\nrm{{e}}_{L_h^2}^2
= \sum_l A_l\ell_l^*\epsilon_l^2$
and $A_l\ell_l^* = \OO(h^d)$, thus
\begin{equation}\label{eq:second_factor}
  \sum_l\epsilon_l^2
  \le \frac{1}{\min_l(A_l\ell_l^*)}
      \nrm{{e}}_{L_h^2}^2
  \le C_3\,h^{-d}\nrm{{e}}_{L_h^2}^2.
\end{equation}
This gives
$\sum_j A_j(\max_{i\sim j}|\tilde\bu({e})(v_i^*)|)^2
\le C_4\,h^{-1}\nrm{{e}}_{L_h^2}^2$. The factor $h^{-1}$ inside the
square root contributes another $h^{-1/2}$.

From~\eqref{eq:first_factor},~\eqref{eq:second_factor},
and~\eqref{eq:Icross_CS_incomp}:
\[
  |I_{\rm cross}|
  \le C_1\,\nrm{\bu}_{W^{1,\infty}}\,h^{d-1}
    (c_\ell h)^{-1/2}\nrm{{e}}_{L_h^2}
     C_4^{1/2}\,h^{-1/2}\nrm{{e}}_{L_h^2}
  = \hat C\,\nrm{\bu}_{W^{1,\infty}}\,
    h^{d-2}\nrm{{e}}_{L_h^2}^2.
\]
The accounting is: $h^{d-1}$ from the face-area sum, $h^{-1/2}$ from
the first factor, $h^{-1/2}$ from the second factor; net $h^{d-2}$.
For $d=2$: $h^{d-2}=1$, giving an $h$-independent bound.
For $d=3$: $h^{d-2}=h$.
In both cases~\eqref{eq:trilinear_bound} holds
with $C_{\rm tri} = \hat C$ ($h$-independent), which is what 
is used in  the stability analysis (\Cref{app:stability_proof}, Step~2b).
\end{proof}

\subsubsection{Convergence of Conserved Quantities (\texorpdfstring{\Cref{thm:conv_cons}}{})}
\label{app:thm:conv_cons_proof}
\begin{proof}
We begin with the convergence of initial data and then extend to all $t\in[0,T]$.

\emph{Initial energy.}
The convergence $\Ekin_h(0)\to \Ekin(0)$ follows from the de~Rham property (\cref{lem:interp_error}) and the Hodge* consistency (\cref{lem:hodge_error}).
Writing
\begin{align*}
  \Ekin_h(0) &= \tfrac{1}{2}\ip{\bP_h\bar\bv}{\bP_h\bar\bv}_1
  = \tfrac{1}{2}\ip{\bar\bv}{\bar\bv}_1 - \tfrac{1}{2}\nrm{(I-\bP_h)\bar\bv}_{\bM_1}^2,
\end{align*}
the first term satisfies
$\ip{\bar\bv}{\bar\bv}_1 = \int_\Omega|\bu_0|^2\,dV + \OO(h^{r_\star}\nrm{\bu_0}_{W^{r_\star,\infty}})$
by \Cref{lem:hodge_error} ($k=1$) and de~Rham interpolation accuracy.
The projection error satisfies $\nrm{(I-\bP_h)\bar\bv}_{\bM_1} \le C\,h^{r_\star}\nrm{\bu_0}_{W^{r_\star,\infty}}$ by \Cref{lem:proj_error}; squaring implies $\OO(h^{2r_\star})$.  Both corrections are $\OO(h^{r_\star})$, giving 
$|\Ekin_h(0) - \Ekin(0)| = \OO(h^{r_\star}\nrm{\bu_0}_{W^{r_\star,\infty}})$.

\emph{Initial helicity.}
By \Cref{lem:helicity_consistency} and Cauchy--Schwarz we have
\begin{align*}
  |H_h(0) - H(0)|
  &\le \nrm{\Qh^1\bar\bv - \bu_0}_{L^2}\nrm{\bom_0}_{L^2}
     + \nrm{\bu_0}_{L^2}\nrm{\Qh^2\bar\bom - \bom_0}_{L^2}
     + \OO(h^{2r_{\rm rec}}) \\
  &\le C\,h^{r_{\rm rec}}\bigl(
       \nrm{\bu_0}_{W^{r_{\rm rec},\infty}}\nrm{\bom_0}_{L^2}
     + \nrm{\bu_0}_{L^2}\nrm{\bom_0}_{W^{r_{\rm rec},\infty}}\bigr)
   = \OO(h^{r_{\rm rec}}).
\end{align*}

\emph{Initial circulation.}
The path $\gamma_h$ is a chain of dual edges forming a closed 1-cycle that approximates the smooth curve 
$\gamma$.  On a mesh with spacing $h$, the dual vertices (circumcentres) do not necessarily lie on
$\gamma$, and the distance between the piecewise-linear path $\gamma_h$ and $\gamma$ is $\OO(h)$.  The quadrature error of replacing $\oint_\gamma\bu_0\cdot d\ell$ by the sum of line integrals over dual edges has two contributions, the $\OO(h)$ path-approximation error and the $\OO(h^2)$ midpoint-rule error on each edge segment.  The dominant term gives $|\Gamma_h^{(\gamma)}(0) - \Gamma(0)| = \OO(h\,\nrm{\bu_0}_{W^{1,\infty}}\,\mathrm{length}(\gamma))$.

\emph{Extension to all times.}
For the conserved quantities, $\Ekin_h(t) = \Ekin_h(0)$ and $\Gamma_h^{(\gamma)}(t) = \Gamma_h^{(\gamma)}(0)$ by \Cref{thm:energy,thm:kelvin}.
Hence $|\Ekin_h(t) - \Ekin(t)| = |\Ekin_h(0) - \Ekin(0)| = \OO(h^{r_\star})$, and $|\Gamma_h^{(\gamma)}(t) - \Gamma(t)| = |\Gamma_h^{(\gamma)}(0) - \Gamma(0)| = \OO(h)$.
For helicity, the approximate conservation $|\ddt H_h| \le C\,h^{\min(r_{\rm rec},\,r_\star)}$ (\cref{thm:helicity}) gives $|H_h(t) - H_h(0)| \le C\,h^{\min(r_{\rm rec},\,r_\star)}\,t$.  Combining with the initial-data convergence $|H_h(0) - H(0)| = \OO(h^{r_{\rm rec}})$ via the triangle inequality we obtain
$  |H_h(t) - H(t)| \le |H_h(t) - H_h(0)| + |H_h(0) - H(0)| = \OO(h^{\min(r_{\rm rec},\,r_\star)}).$
\end{proof}
\subsubsection{Projection Error (\texorpdfstring{\Cref{lem:proj_error}}{})}
\label{app:proj_error_proof}
\begin{proof}
We prove the two parts independently. 
Part~(i) is a $\bD_2\bM_1$-divergence
analysis; part~(ii) is a discrete $W^{1,\infty}$-stability estimate for the
Laplacian. Both use regularity $W^{r_\star,\infty}$.
Throughout,
$\Phi_j$ is the exact and divergence-free flux of $\bu$ across face $j$;
 and $\br := \bM_1\mathcal{R}_h\bu^\flat -\boldsymbol\Phi$ 
 is defined in \Cref{app:approximation},~\eqref{eq:flux_residual}.
 \Cref{lem:hodge_error} gives the face bound $|r_j|\le C_\star h^{r_\star}|f_j|\,
\nrm{\bu}_{W^{r_\star,\infty}}$ and
$\nrm{\br}_{L_{h,1}^2}\le C_\star h^{r_\star}\nrm{\bu}_{W^{r_\star,\infty}}$. 

\medskip
\emph{Part (i): $L_h^2$-bound at $W^{r_\star,\infty}$.}
This is estimate~\eqref{eq:interp_error_v} of \Cref{lem:interp_error},
proved in \Cref{app:approximation}. We define
$e^\perp := (I-\bP_h)\mathcal{R}_h\bu^\flat = \tD_0\phi$ and use
$\bM_1$-orthogonality of $\bP_h$,
\[
  \nrm{e^\perp}_{L_h^2}^2
  = \ip{e^\perp}{\mathcal{R}_h\bu^\flat}_1
  = (e^\perp)^T\bM_1\mathcal{R}_h\bu^\flat
  = (e^\perp)^T\boldsymbol\Phi + (e^\perp)^T\br
  = (e^\perp)^T\br,
\]
because  $\boldsymbol\Phi$ is divergence-free,
$(e^\perp)^T\boldsymbol\Phi = 0$. 
\Cref{lem:hodge_error} then gives
$\nrm{e^\perp}_{L_h^2}\le C_\star h^{r_\star}\nrm{\bu}_{W^{r_\star,\infty}}$,
i.e.\ part~(i) with $C_\pi = C_\star$.

\medskip
\emph{Part (ii): pointwise-reconstruction bound at $W^{r_\star+1,\infty}$.}
With $e^\perp$ as in $i)$, setting $e^\perp = \tD_0\phi$ where
$\phi\in C^0(\KKs)$ solves
\begin{equation}\label{eq:phi_poisson}
 L_h\phi := \bD_2\bM_1\tD_0\phi = \bD_2\bM_1\mathcal{R}_h\bu^\flat =: f_h,
\end{equation}
where $\sum_i(\bM_0)_{ii}\phi_i = 0$.
Recall from \Cref{def:disc_norms} that
$\nrm{e^\perp}_{\rm rec} = \max_j|(\tD_0\phi)_j|/|e_j^*| = \max_j|(\nabla_h\phi)_j|$
with $(\nabla_h\phi)_j := (\phi_{i'}-\phi_i)/|e_j^*|$. We have to show
a discrete $W^{1,\infty}$ estimate for $L_h$ with data $f_h$. 
The bound rests on three facts: $L_h$ is a (finite-element)
Laplacian (ii.1), its solution operator is $W^{1,\infty}$-stable up to a
logarithm (ii.2), and $f_h$ is small in the dual norm (ii.3).

\smallskip
\emph{(ii.1) $L_h$ is $P_1$-stiffness matrix.}
With the Hodge star $(\bM_1)_{jj} = |f_j|/|e_j^*|$,
\begin{equation}\label{eq:Lh_cotangent}
 (L_h\phi)_i \;=\; \sum_{K_{i'}\sim K_i}
   \frac{|f_{i,i'}|}{|e_{i,i'}^*|}\,(\phi_{i'}-\phi_i),
\end{equation}
which on orthogonal Delaunay--Voronoi meshes (\Cref{ass:mesh_reg}\,(3))
coincides with the stiffness matrix of continuous $P_1$-finite
elements on the Delaunay triangulation, such that $L_h$ is the
$P_1$ Laplacian. This is the finite-volume/finite-element
equivalence on Delaunay meshes \cite{eymard2000finite}.

On a flat domain there is no quadrature error in the stiffness; the only finite-volume
modification is the right-hand side $f_h$. On a curved domain a geometric
perturbation of the stiffness appears, of the same order, handled by the
surface framework of \cite{demlow2009higher,poveda2023pointwise}.

\emph{(ii.2) $W^{1,\infty}$ stability.}
For the $P_1$ Laplacian \eqref{eq:Lh_cotangent}, the solution operator is
stable in the max-norm of the gradient up to a single logarithm: for any cell
functional $g:= (g_i)$ and $\psi$ solving $L_h\psi = g$,
\begin{equation}\label{eq:FV_Linfty_reg}
 \max_j|(\nabla_h\psi)_j| \;\le\; C_{\rm st}\,|\log h|\,\nrm{g}_{W^{-1}_\infty},
 \qquad
 \nrm{g}_{W^{-1}_\infty} := \!\!\sup_{\chi\in C^0(\KKs),\,\chi\neq{\rm const}}\!\!
   \frac{\bigl|\sum_i g_i\chi_i\bigr|}{\nrm{\nabla_h\chi}_{L_h^1}},
\end{equation}
valid in $d = 2,3$, where $\nrm{\nabla_h\chi}_{L_h^1} := \sum_j |f_j|\,|e_j^*|\,
|(\nabla_h\chi)_j|$ is the discrete $W^{1,1}$ seminorm. This is the
discrete-Green's-function estimate of Schatz--Wahlbin \cite{schatzwahlbin1995}
and Schatz \cite{schatz1998} for $P_1$ elements on flat domains, extended to
Voronoi--Delaunay surface meshes by Demlow
\cite[Thm~3.2, eq.~(3.15)]{demlow2009higher} and \cite{poveda2023pointwise}.
The term $|\log h|$ is the cost $\nrm{\nabla_h G_h^x}_{L_h^1}\approx|\log h|$
of the discrete Green's function $G_h^x$ and is dimension-independent (\cite[Lem.~3.3]{demlow2009higher}); without it no $L^\infty$ bound on
$\nabla_h\psi$ holds, since $\nrm{\nabla\psi}_{L^\infty}$ is not controlled by
$\nrm{\Delta\psi}_{L^\infty}$.

\smallskip
\emph{(ii.3) Data smallness in dual norm.}
By incompressibility $\bD_2\boldsymbol\Phi = 0$ (\eqref{eq:Phi_orthogonal}),
$f_h = \bD_2\bM_1\mathcal{R}_h\bu^\flat = \bD_2(\boldsymbol\Phi + \br)
= \bD_2\br$. For any $\chi\in C^0(\KKs)$, we obtain with integration by parts
\[
 \sum_i (f_h)_i\,\chi_i \;=\; \pm\,\chi^T\tD_0^T\br
   \;=\; \pm\sum_j r_j\,(\tD_0\chi)_j
   \;=\; \pm\sum_j r_j\,|e_j^*|\,(\nabla_h\chi)_j,
\]
so, with $|f_j|\,|e_j^*|$ the volume of face $j$ and the 
Hodge bound $|r_j|\le C_\star h^{r_\star}|f_j|\,\nrm{\bu}_{W^{r_\star,\infty}}$,
it follows
\begin{equation}\label{eq:fh_Linfty}
 \nrm{f_h}_{W^{-1}_\infty}
 \;\le\; \max_j\frac{|r_j|\,|e_j^*|}{|f_j|\,|e_j^*|}
 \;=\; \max_j\frac{|r_j|}{|f_j|}
 \;\le\; C_\star\,h^{r_\star}\,\nrm{\bu}_{W^{r_\star,\infty}}.
\end{equation}
Applying \eqref{eq:FV_Linfty_reg} with $g = f_h$, $\psi = \phi$, and inserting
\eqref{eq:fh_Linfty},
\[
 \nrm{e^\perp}_{\rm rec} = \max_j|(\nabla_h\phi)_j|
 \;\le\; C_{\rm rec}\,h^{r_\star}\,|\log h|^{\beta_d}\,
   \nrm{\bu}_{W^{r_\star,\infty}}, \qquad \beta_2 = \beta_3 = 1,
\]
with $C_{\rm rec} := C_{\rm st}C_\star$. In particular the bound of
\Cref{lem:proj_error}\,(ii), stated at the stronger norm
$\nrm{\bu}_{W^{r_\star+1,\infty}}$, holds a fortiori. 
The logarithm 
is dimension-independent ($\beta_2 = \beta_3 = 1$).
\end{proof}

\subsubsection{Viscous Consistency (\texorpdfstring{\Cref{thm:NS_consistency}}{}, viscous part)}
\label{app_proof_lem:visc_trunc_weak}
\begin{proof}
The viscous truncation error is the viscous part of $\tau_h^\nu$
(\Cref{thm:NS_consistency}),
$\tau_h^{\rm visc} := \nu\,\Deltah\,\bar\bv - \nu\,\mathcal{R}_h(\delta\dd\bu^\flat)$,
where $\delta\dd\bu^\flat = \delta\omega$, and the vector
Laplacian satisfies $-\Delta\bu = (\delta\dd + \dd\delta)\bu^\flat
= \delta\dd\bu^\flat$ since $\delta\bu^\flat = 0$. Thus
$\tau_h^{\rm visc}$ measures how well $\Deltah = \codiff\,\tD_1$
approximates $\delta\dd$ on $\bar\bv$.
We need to control $\ip{e}{\tau_h^{\rm visc}}_1$
in terms of $\nrm{\tD_1 e}_{\bM_2}$ and $\nrm{e}_{L_h^2}$, with
$\nrm{\tD_1 e}_{\bM_2}^2$ supplied by the energy
estimate. The identity
$\ip{e}{\Deltah\bar\bv}_1 = \ip{\tD_1 e}{\tD_1\bar\bv}_2$ 
reduces the bound to a
primal--dual cell-pairing analysis through Whitney forms.

\smallskip
Since $\tD_1\bar\bv = \mathcal{R}_h\omega$ by de~Rham commutativity,
$\Deltah\bar\bv = \codiff\,\tD_1\bar\bv = \codiff\,\mathcal{R}_h\omega$,
and the viscous truncation error becomes
$\tau_h^{\rm visc} = \nu\,\tau_\star$
with $\tau_\star = \codiff\,\mathcal{R}_h\omega - \mathcal{R}_h(\delta\omega)$.
Testing against a cochain $e$, the $\bM_1$-inner product cancels
$\bM_1^{-1}$ 
\[
\ip{e}{\Deltah\bar\bv}_1
  = \ip{\tD_1 e}{\mathcal{R}_h\omega}_2,
\]
so 
\begin{equation}\label{eq:weak_visc_split}
\ip{e}{\tau_h^{\rm visc}}_1
  = \nu\bigl(\ip{\tD_1 e}{\mathcal{R}_h\omega}_2
  - \ip{e}{\mathcal{R}_h(\delta\omega)}_1\bigr).
\end{equation}
Both terms are bilinear pairings of a cochain against the de~Rham
interpolant of a smooth form. 
Each bilinear pairing has a controlled cell error.

\smallskip
Define the dual-cell sums
\begin{equation}\label{eq:SA_SB_def}
  S_A := \sum_{f_k^*} (\tD_1 e)_k \int_{f_k^*}\!\star\,\omega,
  \qquad
  S_B := \sum_{e_j^*} e_j \int_{e_j^*}\!\star\,\delta\omega,
\end{equation}
where $f_k^*$ ranges over the dual 2-cells on which the dual 2-cochain
$\tD_1 e$ is indexed, $e_j^*$ over the dual 1-cells on which
the dual 1-cochain $e$ is indexed.
Equivalently, $S_A$ and $S_B$ are obtained from the discrete pairings
$\ip{\tD_1 e}{\mathcal{R}_h\omega}_2$ and
$\ip{e}{\mathcal{R}_h(\delta\omega)}_1$ by replacing, on each cell, the
diagonal Hodge-star weight with this integral. By
\Cref{lem:hodge_bilinear},
\begin{align}
  \bigl|\ip{\tD_1 e}{\mathcal{R}_h\omega}_2
    - S_A\bigr|
  &\le C_2\,h^{r_\star}\,\nrm{\omega}_{W^{r_\star,\infty}}\,\nrm{\tD_1 e}_{\bM_2},
  \label{eq:weak_visc_A}\\
  \bigl|\ip{e}{\mathcal{R}_h(\delta\omega)}_1
    - S_B\bigr|
  &\le C_1\,h^{r_\star}\,\nrm{\delta\omega}_{W^{r_\star,\infty}}\,\nrm{e}_{L_h^2},
  \label{eq:weak_visc_B}
\end{align}
with the $h^{r_\star}$ from \Cref{lem:hodge_error}.

\smallskip
Abbreviate $\mathrm{A} := \ip{\tD_1 e}{\mathcal{R}_h\omega}_2$ and
$\mathrm{B} := \ip{e}{\mathcal{R}_h(\delta\omega)}_1$, so that
$\nu^{-1}\ip{e}{\tau_h^{\rm visc}}_1 = \mathrm{A} - \mathrm{B}$ by
\eqref{eq:weak_visc_split}.  The triangle inequality implies
\begin{equation}\label{eq:triangle_weak_visc}
  |\mathrm{A} - \mathrm{B}|
  \le |\mathrm{A} - S_A| + |S_A - S_B| + |S_B - \mathrm{B}|.
\end{equation}
We decompose the dual-cell integral as
$\int_{\sigma^*}\!\star\alpha = \int_\Omega W_\sigma\wedge\star\alpha + \epsilon_\sigma(\alpha)$,
where $W_\sigma$ is the Whitney $k$-form dual to the primal $k$-cell~$\sigma$.
The Hodge star reproduces the dual-cell volume, so
$\int_\Omega W_\sigma\wedge\star\mathbf{1} = |\sigma^*|$ \cite{bossavit1998}
and the decomposition reproduces constants. 
 The remainder $\epsilon_\sigma(\alpha)$ satisfies
\begin{equation}\label{eq:whitney_remainder}
|\epsilon_\sigma(\alpha)|
\le C\,h^{r_\star}\,\nrm{\alpha}_{W^{r_\star,\infty}}\,|\sigma^*|,
\end{equation}
which is the Hodge-error estimate of \Cref{lem:hodge_error}
applied to $\alpha$ on the dual cell $\sigma^*$.
The Whitney parts of $S_A$ and $S_B$ cancel
\begin{equation}\label{eq:whitney_cancellation}
\int_\Omega\mathcal{W}_h(\tD_1 e)\wedge\star\,\omega
= \int_\Omega \dd(\mathcal{W}_h e)\wedge\star\,\omega
= \int_\Omega \mathcal{W}_h(e)\wedge\star\,\delta\omega,
\end{equation}
where the first equality uses the Whitney--de~Rham commutativity
$\dd\circ\mathcal{W}_h = \mathcal{W}_h\circ\tD_1$. This is the classical
simplicial identity \cite{whitney1957,dodziuk1976,arnold2006} on the
simplicial refinement $\KK_h^{\rm simp}$ of \Cref{def:Whitney}: 
 each dual edge lies along a simplex edge and each dual face is a union of simplex faces, so $\tD_1$ acts as
the coboundary on the cochains and the commutativity
holds cell by cell. The second equality is integration by parts: $\int_\Omega \dd\beta\wedge\star\omega = (-1)^{k+1}\int_\Omega\beta\wedge\dd\star\omega
= \int_\Omega\beta\wedge\star\,\delta\omega$ for a $k$-form $\beta = \mathcal{W}_h e$
(here $k = 1$). The two Whitney integrals therefore equal each other;
the difference $S_A - S_B$ collects only the cell remainders
$\epsilon_\sigma$, each of size $\OO(h^{r_\star}|\sigma^*|)$.
Combining~\eqref{eq:whitney_remainder} over cells via Cauchy--Schwarz:
\begin{equation}\label{eq:SA_SB_bound}
|S_A - S_B|
\le C\,h^{r_\star}\bigl(\nrm{\omega}_{W^{r_\star,\infty}}\nrm{\tD_1 e}_{\bM_2}
  + \nrm{\delta\omega}_{W^{r_\star,\infty}}\nrm{e}_{L_h^2}\bigr).
\end{equation}

Substituting~\eqref{eq:weak_visc_A},~\eqref{eq:weak_visc_B},
and~\eqref{eq:SA_SB_bound} into the triangle inequality~\eqref{eq:triangle_weak_visc}
gives
\[
  |\ip{e}{\tau_h^{\rm visc}}_1|
  \le \nu\,C_\delta\,h^{r_\star}\bigl(
  \nrm{\omega}_{W^{r_\star,\infty}}\nrm{\tD_1 e}_{\bM_2}
  + \nrm{\delta\omega}_{W^{r_\star,\infty}}\nrm{e}_{L_h^2}\bigr).
\]
Sobolev embedding gives
$\nrm{\omega}_{W^{r_\star,\infty}}
  \le C\nrm{\bu}_{W^{r_\star+1,\infty}}$
and
$\nrm{\delta\omega}_{W^{r_\star,\infty}}
  \le C\nrm{\bu}_{W^{r_\star+2,\infty}}
  \le C\nrm{\bu}_{H^{s+2}}$, where we have used 
$H^{s+2}\hookrightarrow W^{r_\star+2,\infty}$
for $s+2 > r_\star+2+d/2$, which holds for
$s\ge 3$ when $r_\star = 1$ and $s\ge 4$ when $r_\star = 2$.
This yields~\eqref{eq:NS_trunc_bound}. The factor of $\nu$ on the
right-hand side, from $\tau_h^{\rm visc} = \nu\,\tau_\star$,
allows to absorb the viscous truncation into the dissipation
$\nu\nrm{\tD_1 e}_{\bM_2}^2$ via Young's inequality, as exploited in
the NS stability proof. \qedhere
\end{proof}

\subsubsection{Navier--Stokes Stability (\texorpdfstring{\Cref{thm:NS_stability}}{})}
\label{app:NS_stability_proof}

\begin{proof}
The proof inherits the framework of the Euler stability proof
(\Cref{app:stability_proof}). The reference field
$\bar\bv = \mathcal{R}_h\bu^\flat$, the Helmholtz
decomposition~\eqref{eq:helmholtz_split}, and the projection-error
bounds~\eqref{eq:eperp_bounds} are unchanged.  New is the
sign-definite dissipation $-\nu\nrm{\tD_1 e}_{\bM_2}^2 \le 0$ in the energy equation,
and that the truncation has an additional viscous part
$\tau_h^{\rm visc}$ controlled only weakly through
\Cref{app_proof_lem:visc_trunc_weak}. 

\smallskip
Taking the $\bM_1$-inner product of the NS error
equation~\eqref{eq:NS_error_expanded} with $e$ results in (cf. \eqref{eq:energy_balance})
\begin{equation}\begin{split}\label{eq:NS_energy_balance}
   \tfrac12\ddt\nrm{e}_{L_h^2}^2 + \nu\nrm{\tD_1 e}_{\bM_2}^2
  &\;=
        -\ip{e}{\bP_h \bQ(\bar\bv,\,e)}_1 -  \ip{e}{\bP_h \bQ(e,\,\bar\bv)}_1 - \ip{e}{\bP_h \bQ(e,\,e)}_1\\
       & \;-\; \ip{e}{\tau_h}_1\;-\; \ip{e}{\tau_h^{\rm visc}}_1.
\end{split}\end{equation}
The three nonlinear inner products and the Euler truncation
$\tau_h^{\rm Euler}$ are bounded as in the Euler
proof. Combining~\eqref{eq:I_bound},~\eqref{eq:II_bound},~\eqref{eq:III_bound},
and the $h^r$-piece of~\eqref{eq:tau_residual} yields
\begin{equation}\label{eq:NS_euler_inherited}
  \sum_{X}\bigl|\ip{e}{\bP_h X}_1\bigr|
  + \bigl|\ip{e}{\tau_h^{\rm Euler}}_1\bigr|
  \;\le\; C_L^{\rm E}(h)\,\nrm{e}_{L_h^2}^2
       + F^{\rm E}(h)\,\nrm{e}_{L_h^2}
       + G^{\rm E}(h),
\end{equation}
where $X$ runs over the $\bQ$-terms in \eqref{eq:NS_energy_balance}, and 
with $C_L^{\rm E},F^{\rm E},G^{\rm E}$ defined as
in~\eqref{eq:CL_def}--\eqref{eq:G_def}. The term
$G^{\rm E}(h) = \OO(h^{2r_\star})$ arises from
the truncation Helmholtz residual $(I-\bP_h)\tau_h^{\rm Euler}$ and
from the direct Step~2b summand $\ip{e^\perp}{\bQ(e^\perp,\bar\bv)}_1$.

\smallskip
\Cref{app_proof_lem:visc_trunc_weak} gives
\begin{equation}\label{eq:NS_visc_weak_bound}
  \bigl|\ip{e}{\tau_h^{\rm visc}}_1\bigr|
  \;\le\; \nu\,C_\delta'\,h^{r_\star}\,\nrm{\bu}_{H^{s+2}}\,
       \bigl(\nrm{\tD_1 e}_{\bM_2} + \nrm{e}_{L_h^2}\bigr).
\end{equation}
Apply Young's inequality $AX \le \tfrac{\nu}{2}X^2 + \tfrac{A^2}{2\nu}$
with $X = \nrm{\tD_1 e}_{\bM_2}$ and
$A = \nu\,C_\delta'\,h^{r_\star}\,\nrm{\bu}_{H^{s+2}}$ to get
\begin{equation}\label{eq:NS_visc_young}
  \nu\,C_\delta'\,h^{r_\star}\,\nrm{\bu}_{H^{s+2}}\,\nrm{\tD_1 e}_{\bM_2}
  \;\le\; \tfrac{\nu}{2}\nrm{\tD_1 e}_{\bM_2}^2
          + \tfrac{\nu (C_\delta')^2 h^{2r_\star}\nrm{\bu}_{H^{s+2}}^2}{2}.
\end{equation}
Substituting this into~\eqref{eq:NS_visc_weak_bound} gives
\begin{equation}\label{eq:NS_visc_assembled}
  \bigl|\ip{e}{\tau_h^{\rm visc}}_1\bigr|
  \;\le\; \tfrac{\nu}{2}\nrm{\tD_1 e}_{\bM_2}^2
       + \nu\,C_\delta'\,h^{r_\star}\,\nrm{\bu}_{H^{s+2}}\,\nrm{e}_{L_h^2}
       + \tfrac{\nu(C_\delta')^2 h^{2r_\star}\nrm{\bu}_{H^{s+2}}^2}{2}.
\end{equation}
Substituting~\eqref{eq:NS_euler_inherited}, \eqref{eq:NS_visc_assembled} into~\eqref{eq:NS_energy_balance}
and moving the dissipation to the right-hand side implies
\[
  \tfrac12\ddt\nrm{e}_{L_h^2}^2 + \nu\nrm{\tD_1 e}_{\bM_2}^2
  \;\le\; C_L^{\rm E}\nrm{e}_{L_h^2}^2 + F^{\rm E}\nrm{e}_{L_h^2}
       + G^{\rm E}
       + \tfrac{\nu}{2}\nrm{\tD_1 e}_{\bM_2}^2
       + \nu C_\delta' h^{r_\star}\nrm{\bu}_{H^{s+2}}\nrm{e}_{L_h^2}
       + \tfrac{\nu(C_\delta')^2 h^{2r_\star}\nrm{\bu}_{H^{s+2}}^2}{2}.
\]
The $\tfrac{\nu}{2}\nrm{\tD_1 e}_{\bM_2}^2$ contribution on the
right is absorbed into the dissipation $\nu\nrm{\tD_1 e}_{\bM_2}^2$
on the left, the term $\tfrac{\nu}{2}\nrm{\tD_1 e}_{\bM_2}^2\ge 0$ is skipped,
\begin{equation}\label{eq:NS_gronwall_inequality}
  \tfrac12\ddt\nrm{e}_{L_h^2}^2
  \;\le\; C_L^{\rm E}\nrm{e}_{L_h^2}^2 + F^{\rm NS}(h)\nrm{e}_{L_h^2}
       + G^{\rm NS}(h),
\end{equation}
with
\begin{align}
  F^{\rm NS}(h) &= F^{\rm E}(h) + \nu C_\delta' h^{r_\star}\nrm{\bu}_{H^{s+2}},
                  \label{eq:FNS_def}\\
  G^{\rm NS}(h) &= G^{\rm E}(h)
                  + \tfrac{\nu(C_\delta')^2 h^{2r_\star}\nrm{\bu}_{H^{s+2}}^2}{2}.
                  \label{eq:GNS_def}
\end{align}
The dissipation converts the weak
viscous truncation bound into a linear-$\nrm{e}$ forcing plus
a constant of order $\nu h^{2r_\star}$.

\smallskip
The argument from Step~3 of the Euler proof applies without change. 
Reduction to $\tilde C_L$ via~\eqref{eq:tildeCL_def} (the
dependence of $C_L^{\rm E}$ on $h$ is identical to the Euler case),
Young's inequality $2F^{\rm NS}\nrm{e}\le\nrm{e}^2 + (F^{\rm NS})^2$,
and Gr\"onwall's inequality yield
\begin{equation}\label{eq:NS_gronwall_final}
  \nrm{e(t)}_{L_h^2}^2
  \;\le\; e^{(2\tilde C_L+1)t}
       \bigl[\nrm{e(0)}_{L_h^2}^2 + \bigl((F^{\rm NS}(h))^2 + 2G^{\rm NS}(h)\bigr)\,t\bigr].
\end{equation}
Adding the new viscous summands to the Euler-part analysis,
$(F^{\rm NS})^2 \le 2(F^{\rm E})^2 + 2\nu^2(C_\delta')^2 h^{2r_\star}\nrm{\bu}_{H^{s+2}}^2$,
and $G^{\rm NS} = G^{\rm E} + \tfrac{\nu(C_\delta')^2}{2}h^{2r_\star}\nrm{\bu}_{H^{s+2}}^2$.
The Euler estimate~\eqref{eq:F2G_rate} gives
$(F^{\rm E})^2 + 2G^{\rm E} = \OO(h^{2r^\ast})$, and the
contributions are $\OO(h^{2r_\star})\le \OO(h^{2r^\ast})$. Hence
$(F^{\rm NS}(h))^2 + 2G^{\rm NS}(h) = \OO(h^{2r^\ast})$.

\smallskip
The hypothesis
$\nrm{e(0)}_{L_h^2}\le C_0 h^r$ of \Cref{thm:NS_stability} with
$r\le r^\ast$ combines with~\eqref{eq:NS_gronwall_final} to give
\[
  \sup_{t\in[0,T]}\nrm{e(t)}_{L_h^2}\;\le\; C^\nu(T)\,h^r,
\]
where
\[
  C^\nu(T) \;:=\; e^{(\tilde C_L+1/2)T}
    \sqrt{\,C_0^2 + (\tilde F_{\rm NS}^2 + 2\tilde G_{\rm NS})\,T\,},
\]
with $\tilde F_{\rm NS}, \tilde G_{\rm NS}$ the smooth-solution
upper bounds of $F^{\rm NS}/h^{r^\ast}$ and $G^{\rm NS}/h^{2r^\ast}$.

\smallskip
For $\nu\in[0,\nu_0]$ each summand under
the square root is bounded uniformly: $C_0^2$ is independent of
$\nu$; $(\tilde F^{\rm NS})^2$ depends on $\nu$ through
$F^{\rm E} + \nu C_\delta'\nrm{\bu}_{H^{s+2}}$ which is bounded by
$\tilde F^{\rm E} + \nu_0 C_\delta'\nrm{\bu}_{H^{s+2}}$;
$\tilde G^{\rm NS}$ likewise. Hence $C^\nu(T)$ is uniformly bounded
on $[0,\nu_0]$. As $\nu\to 0$, $F^{\rm NS}\to F^{\rm E}$ and
$G^{\rm NS}\to G^{\rm E}$, so $C^\nu(T)\to C(T)$, the Euler
constant. The discrete viscous solution therefore converges to the
discrete inviscid solution on the same mesh as $\nu\to 0$ at the
same rate, uniformly in $\nu\in[0,\nu_0]$.
\end{proof}

\subsubsection{Leray--Hopf Solutions (\texorpdfstring{\Cref{thm:weak_convergence}}{})}
\label{app:Leray_Hopf_proof}

\begin{proof}
We keep viscosity $\nu > 0$ fixed. Constants depend on $\nu$ but not on~$h$.
The proof follows the classical Leray construction, adapted to our
 setting in five stages: uniform bounds, time-derivative control,
compactness, passage to the limit in the weak formulation,
and recovery of the energy inequality.

\emph{Step~1: Uniform bounds.}
By \Cref{lem:uniform_apriori}, for fixed $\nu > 0$:
\[
  \nrm{\bu^h}_{L^\infty(0,T;\,L^2)} \le C_W\sqrt{E_0},
  \qquad
  \nrm{\bu^h}_{L^2(0,T;\,H^1)}
  \le \sqrt{\tfrac{C_W^2 E_0}{\nu}} + \OO(h^{r_\star/2}),
\]
both estimates are uniform in $h$. The consistency remainder
of \Cref{lem:uniform_apriori}\,(2) is $\OO(h^{r_\star})\nrm{\bu^h}_{H^1}^2$,
which by the $L^2_t H^1$ bound is $\OO(h^{r_\star})$ in $L^1(0,T)$ and
hence vanishes; the constant degrades as $\nu\to 0$, so the
compactness below is a fixed-positive-viscosity statement.

\emph{Step~2: Time-derivative bound.}
For $\boldsymbol\varphi\in H^1(\Omega)$, divergence-free,
the continuous $H^{-1}$--$H^1$ pairing converts the momentum equation
to the cochain pairing
\begin{equation}\label{eq:LH_pairing_conv}
  \ip{\partial_t\bu^h}{\boldsymbol\varphi}_{L^2}
  = \ip{\partial_t\bv^h}{\mathcal{R}_h\boldsymbol\varphi^\flat}_1
  + \OO(h^{r_\star})\,\nrm{\partial_t\bu^h}_{H^{-1}}\,\nrm{\boldsymbol\varphi}_{H^1}.
\end{equation}
Substituting the cochain momentum equation and using the bounds yields
\begin{align*}
  |\ip{\bP_h\bQ(\bv^h,\bv^h)}{\mathcal{R}_h\boldsymbol\varphi^\flat}_1|
  &\le C\,\nrm{\bu^h}_{L^4}^2\,\nrm{\boldsymbol\varphi}_{H^1}, \\
  |\ip{\Deltah\bv^h}{\mathcal{R}_h\boldsymbol\varphi^\flat}_1|
  &= |\ip{\tD_1\bv^h}{\tD_1\mathcal{R}_h\boldsymbol\varphi^\flat}_2|
  \le C\,\nrm{\bu^h}_{H^1}\,\nrm{\boldsymbol\varphi}_{H^1},
\end{align*}
which together give
\[
  |\ip{\partial_t\bu^h}{\boldsymbol\varphi}_{L^2}|
  \le C\,\bigl(\nrm{\bu^h}_{L^4}^2 + \nu\nrm{\bu^h}_{H^1}\bigr)\,
    \nrm{\boldsymbol\varphi}_{H^1}.
\]
By Sobolev interpolation we obtain
$\nrm{\bu^h}_{L^4} \le C\nrm{\bu^h}_{L^2}^{1/4}\nrm{\bu^h}_{H^1}^{3/4}$,
$\nrm{\bu^h}_{L^4}^2 \le C\nrm{\bu^h}_{L^2}^{1/2}\nrm{\bu^h}_{H^1}^{3/2}$,
which is bounded in $L^{4/3}_t$ by the uniform $L^\infty_tL^2$
and $L^2_tH^1$ estimates of Step~1.
Hence we have uniformly in $h$
\begin{equation}\label{eq:LH_dt_bound}
  \nrm{\partial_t\bu^h}_{L^{4/3}(0,T;\,H^{-1}(\Omega))}
  \le C(E_0,\nu,T).
\end{equation}

\emph{Step~3: Compactness.}
By Banach--Alaoglu, a subsequence, still denoted $h_k$, satisfies
$\bu^{h_k}\rightharpoonup\bu$ weakly in $L^2(0,T;\,H^1)$
and weakly-$*$ in $L^\infty(0,T;\,L^2)$;
correspondingly $\bom^{h_k} = \tD_1\bv^{h_k}$
satisfies $\bom^{h_k}\rightharpoonup\bom = \mathrm{curl}\,\bu$
weakly in $L^2(0,T;\,L^2)$.
The Aubin--Lions compactness lemma, applied with
$\bu^h$ bounded in $L^2(0,T;\,H^1)$, $\partial_t\bu^h$ bounded in
$L^{4/3}(0,T;\,H^{-1})$, and the compact embedding
$H^1(\Omega) \hookrightarrow\hookrightarrow L^2(\Omega)$,
implies
\[
  \bu^{h_k}\to\bu \quad\text{strongly in } L^2(0,T;\,L^2(\Omega)).
\]

\emph{Step~4: Passage to limit.}
For $\boldsymbol\varphi\in C_c^\infty(\Omega\times[0,T))$,
divergence-free, we set $\bar{\boldsymbol\varphi}^h(t)
:= \mathcal{R}_h\boldsymbol\varphi^\flat(t)$.
Test the cochain equation against $\bar{\boldsymbol\varphi}^h$,
integrate by parts in time, using $\boldsymbol\varphi(\cdot,T) = 0$,
\begin{equation}\label{eq:LH_discrete_weak}
  -\!\!\int_0^T\!\!\ip{\bv^h}{\partial_t\bar{\boldsymbol\varphi}^h}_1
  + \!\!\int_0^T\!\!\ip{\bQ(\bv^h,\bv^h)}{\bar{\boldsymbol\varphi}^h}_1
  + \nu\!\!\int_0^T\!\!\ip{\tD_1\bv^h}{\tD_1\bar{\boldsymbol\varphi}^h}_2
  = \ip{\bv^h(0)}{\bar{\boldsymbol\varphi}^h(0)}_1
\end{equation}
For the time-derivative term, Hodge/Whitney consistency gives
$\ip{\bv^h}{\partial_t\bar{\boldsymbol\varphi}^h}_1
= \ip{\bu^h}{\partial_t\boldsymbol\varphi}_{L^2}
+ \OO(h^{r_\star})\,\nrm{\bu^h}_{L^2}\,\nrm{\partial_t\boldsymbol\varphi}_{H^1}$.
The first term converges to
$\int_0^T\int_\Omega\bu\cdot\partial_t\boldsymbol\varphi$
by the strong convergence $\bu^{h_k}\to\bu$ in $L^2(0,T;\,L^2)$;
the remainder vanishes uniformly.

For the viscous term we have by Hodge consistency on 2-forms
$\ip{\tD_1\bv^h}{\tD_1\bar{\boldsymbol\varphi}^h}_2
= \int_\Omega\nabla\bu^h:\nabla\boldsymbol\varphi
+ \OO(h^{r_\star})\,\nrm{\bu^h}_{H^1}\,\nrm{\boldsymbol\varphi}_{H^2}$.
The first term converges to
$\int_0^T\int_\Omega\nabla\bu:\nabla\boldsymbol\varphi$
by weak convergence $\nabla\bu^{h_k}\rightharpoonup\nabla\bu$
in $L^2(0,T;L^2)$ and $\nabla\boldsymbol\varphi$ as multiplier.

The discrete pairing $\ip{\bQ(\bv^h,\bv^h)}{\bar{\boldsymbol\varphi}^h}_1$
approximates the continuous Lamb pairing.
By the analogue of \Cref{lem:weak_consistency} applied to
$\boldsymbol\varphi$
\begin{equation}\label{eq:LH_nonlinear_id}
  \ip{\bQ(\bv^h,\bv^h)}{\bar{\boldsymbol\varphi}^h}_1
  = \int_\Omega(\bom^h\times\bu^h)\cdot\boldsymbol\varphi\,\mathrm{d}V
  + R_h(\boldsymbol\varphi),
\end{equation}
where $R_h$ collects the reconstruction error
($\OO(h^{r_{\rm rec}})$), the Hodge$\star$ error ($\OO(h^{r_\star})$),
and the trapezoidal averaging error in the extrusion.
The remainder satisfies
\[
  |R_h(\boldsymbol\varphi)|
  \le C\,h\,\nrm{\bu^h}_{L^2}\,\nrm{\bu^h}_{H^1}\,
  \nrm{\boldsymbol\varphi}_{W^{1,\infty}}
\]
pointwise in $t$. Integrating in time and applying H\"older's inequality 
\[
  \int_0^T|R_h(\boldsymbol\varphi)|\,dt
  \le C\,h\,\nrm{\bu^h}_{L^\infty_t L^2}
    \,T^{1/2}\,\nrm{\bu^h}_{L^2_t H^1}
    \nrm{\boldsymbol\varphi}_{L^\infty_t W^{1,\infty}}
  \le C(E_0,\nu,T)\,h \to 0.
\]
For the continuous Lamb integral,
$(\bom^h \times \bu^h)\cdot\boldsymbol\varphi
= \bom^h\cdot(\bu^h\times\boldsymbol\varphi)$.
Since $\bu^h\to\bu$ strongly in $L^2(0,T;L^2)$
and $\boldsymbol\varphi\in L^\infty$,
$\bu^h\times\boldsymbol\varphi \to \bu\times\boldsymbol\varphi$
strongly in $L^2(0,T;L^2)$;
combined with the weak convergence
$\bom^{h_k}\rightharpoonup\bom$ in $L^2(0,T;L^2)$,
weak-strong duality gives
\[
  \int_0^T\!\!\int_\Omega\bom^{h_k}\cdot(\bu^{h_k}\times\boldsymbol\varphi)
  \to
  \int_0^T\!\!\int_\Omega\bom\cdot(\bu\times\boldsymbol\varphi)
  = \int_0^T\!\!\int_\Omega(\bom\times\bu)\cdot\boldsymbol\varphi.
\]
By Whitney approximation~\eqref{eq:Whitney_approx},
$\bu^h(0) = \mathcal{W}_h\bP_h\mathcal{R}_h\bu_0^\flat \to \bu_0$
strongly in $L^2$.
Combined with $\bar{\boldsymbol\varphi}^h(0)\to\boldsymbol\varphi(\cdot,0)$
strongly in $L^2$,
$\ip{\bv^h(0)}{\bar{\boldsymbol\varphi}^h(0)}_1
\to \int_\Omega\bu_0\cdot\boldsymbol\varphi(\cdot,0)$.

\smallskip
Combining the four limits in~\eqref{eq:LH_discrete_weak}
yields~\eqref{eq:weak_NS}.
\smallskip

\emph{Step~5: Energy inequality.}
The energy equality
$\frac{1}{2}\nrm{\bv^h(t)}_{L_h^2}^2
+ \nu\int_0^t\nrm{\tD_1\bv^h}_{\bM_2}^2\,ds
= \frac{1}{2}\nrm{\bv^h(0)}_{L_h^2}^2$
(\cref{lem:NS_energy_bound})
combined with Hodge equivalence
$\nrm{\bv^h(t)}_{L_h^2}^2 = \nrm{\bu^h(t)}_{L^2}^2 + \OO(h^{r_\star})$
and
$\nrm{\tD_1\bv^h}_{\bM_2}^2 = \nrm{\nabla\bu^h}_{L^2}^2 + \OO(h^{r_\star})\nrm{\bu^h}_{H^1}^2$
gives, for $t\in[0,T]$,
\[
  \tfrac{1}{2}\nrm{\bu^h(t)}_{L^2}^2
  + \nu\int_0^t\nrm{\nabla\bu^h}_{L^2}^2\,ds
  = \tfrac{1}{2}\nrm{\bu^h(0)}_{L^2}^2 + \OO(h^{r_\star}).
\]
Pass to $\liminf_{k\to\infty}$ along a further subsequence
on which $\bu^{h_k}(t) \to \bu(t)$ strongly in $L^2$
for a.e.\ $t\in[0,T]$.
The kinetic term passes pointwise;
the viscous term passes via weak lower semicontinuity:
$\int_0^t\nrm{\nabla\bu}_{L^2}^2\,ds
\le \liminf_k\int_0^t\nrm{\nabla\bu^{h_k}}_{L^2}^2\,ds$.
Equality becomes inequality because $\liminf$ may be strict.
Combined with $\nrm{\bu^h(0)}_{L^2}^2 \to \nrm{\bu_0}_{L^2}^2$,
this yields~\eqref{eq:Leray_energy_ineq}.
\end{proof}

\subsubsection{Convergence to a CMV Euler Solution (\texorpdfstring{\Cref{thm:CMV}}{})}
\label{app:CMV_proof}

\begin{proof}[Proof of \Cref{thm:CMV}]
In the inviscid case there is no
Aubin--Lions strong compactness available: velocities converge only
weakly, and weak limits do not pass through the nonlinearity,
so the limit is measure-valued. Energy conservation
supplies the two bounds the viscous estimate would otherwise give: it
yields weak compactness against divergence-free test
functions (Step~1) and controls the total energy at every
time~(\eqref{eq:CMV_energy_limit}). The gap between the
product $w_iw_j$ of the weak limit and the weak-$*$ limit $\mu_{ij}$ of
the momentum-flux products $u^h_iu^h_j$ is the concentration defect
$\sigma_{ij} = \mu_{ij} - w_iw_j$ (Step~2), identified through the
Young-measure representation as a non-negative variance (Step~3); energy
conservation then fixes its trace (Step~4), and passage to the limit in
the discrete weak form closes the equation (Step~5).

\smallskip
\emph{Step 1: Uniform bounds and weak compactness at each time.}
Energy conservation (\cref{thm:energy_bound}) gives
$\nrm{\bv^h(t)}_{L_h^2}^2 = \nrm{\bv^h(0)}_{L_h^2}^2$, and Whitney norm-boundedness
(\cref{def:Whitney}) gives
$\nrm{\bu^h(t)}_{L^2} \le C_W\nrm{\bv^h(t)}_{L_h^2} \le C\sqrt{E_0}$
uniformly in $h,t$.
For equicontinuity of $t\mapsto\ip{\bu^h(t)}{\boldsymbol\varphi}_{L^2}$
against smooth divergence-free $\boldsymbol\varphi\in C_c^\infty(\Omega)$,
we note that the discrete scheme propagates $\bv^h$
in vector-invariant form (\cref{eq:V}),
not in conservation form;
only after testing against divergence-free $\boldsymbol\varphi$
does the Bernoulli gradient $\tD_0 B^h$ drop out and
the rotational term reduce (up to $\OO(h^r)$) to a
quadratic in $\bu^h = \mathcal{W}_h\bv^h$.
Taking the increment over $[t_1,t_2]$ of the discrete weak
identity~\eqref{eq:approx_CMV} of \Cref{lem:weak_consistency}
yields for every divergence-free $\boldsymbol\varphi$
\begin{equation}\label{eq:CMV_duhdt_bound}
  |\ip{\bu^h(t_2)-\bu^h(t_1)}{\boldsymbol\varphi}_{L^2}|
  \le \int_{t_1}^{t_2}\!\!\int_\Omega|u^h_i u^h_j\,\partial_j\varphi_i|
    \,\mathrm{d}x\,\mathrm{d}t
  + \OO(h^r)\,\nrm{\boldsymbol\varphi}_{C^1}\,(t_2-t_1),
\end{equation}
and by H\"older's inequality the right-hand side is bounded by
$C(\boldsymbol\varphi, E_0)\,(t_2-t_1) + \OO(h^r)$,
uniformly in $h$.
Arzel\`a--Ascoli on a countable dense set
$\{\boldsymbol\varphi_m\}\subset C_c^\infty(\Omega)\cap\{\nabla\cdot\boldsymbol\varphi=0\}$
and a diagonal argument
produces a subsequence $(\bu^h_k)_k$ and
$\bw \in L^\infty(0,T;\,L^2)$ with
$\bu^{h_k}(t) \rightharpoonup \bw(t)$ weakly in $L^2$ for every $t$:
equicontinuity against divergence-free tests combined with the
uniform $L^2$-bound extracts a limit $\bw(t)$ in the divergence-free
subspace; the uniform approximate incompressibility
$\nrm{\bD_2\bM_1\bv^{h_k}}\le C h^{r_\star}\to 0$
(\cref{lem:proj_error}) forces the limit $\bw(t)$ to be
divergence-free and extends the weak convergence to all of $L^2$
by Leray decomposition.

Strong convergence of initial data $\bu^h(0) \to \bu_0$ in $L^2$
follows from Whitney approximation~\eqref{eq:Whitney_approx}
and a density argument. Combining the chain
\[
  \nrm{\bu^h(t)}_{L^2}^2
  = \nrm{\bv^h(t)}_{L_h^2}^2 + \OO(h^{r_\star})
  = \nrm{\bv^h(0)}_{L_h^2}^2 + \OO(h^{r_\star})
  = \nrm{\bu^h(0)}_{L^2}^2 + \OO(h^{r_\star}),
\]
where the first and third equalities use Hodge$\star$ consistency
(\cref{lem:hodge_error}) and the middle
equality is energy conservation,
with $\nrm{\bu^h(0)}_{L^2}^2\to E_0$ gives
\begin{equation}\label{eq:CMV_energy_limit}
  \lim_{k\to\infty}\nrm{\bu^{h_k}(t)}_{L^2}^2 = E_0
  \qquad\text{for every } t\in[0,T].
\end{equation}

\emph{Step 2: Space-time extraction of the defect measure.}
The products $(u^{h_k}_i u^{h_k}_j)_{k}$ are uniformly bounded in
$L^\infty(0,T;\,L^1(\Omega))$:
$\sup_t\int_\Omega|u^{h_k}_i u^{h_k}_j|\,\mathrm{d}x
\le \nrm{\bu^{h_k}(t)}_{L^2}^2 \le C E_0$.
By Banach--Alaoglu on
$\mathcal{M}([0,T]\times\Omega) = C_c([0,T]\times\Omega)^*$,
a further subsequence, still denoted $h_k$, satisfies
$u^{h_k}_i u^{h_k}_j\,\mathrm{d}t\,\mathrm{d}x
\overset{*}{\rightharpoonup}\mu_{ij}$
in $\mathcal{M}([0,T]\times\Omega)$ with
$\mu_{ij}$ a finite Radon measure.
Disintegrating $\mu_{ij}$ against the Lebesgue measure on $[0,T]$,
one obtains a weak-$*$ measurable family
$\mu_{ij}(t) \in L^\infty_{w^*}(0,T;\,\mathcal{M}(\Omega))$.
Symmetry $\mu_{ij}=\mu_{ji}$ is inherited from
$u^{h_k}_i u^{h_k}_j = u^{h_k}_j u^{h_k}_i$ under weak-$*$ limit.
We define the concentration defect
\begin{equation}\label{eq:CMV_sigma_def}
  \sigma_{ij}(t) := \mu_{ij}(t) - w_i(t)\,w_j(t)\,\mathrm{d}x
  \qquad\in \mathcal{M}(\Omega),
\end{equation}
which is measurable in $t$.

\smallskip

\emph{Step 3: $\sigma$ is non-negative and symmetric.}
For the Young measure
$\{\nu_{t,x}\}_{(t,x)}$ associated with the sequence
$(\bu^{h_k})$
(e.g.\ \cite{Tartar1979}),
\[
  w_i(t,x) = \int_{\mathbb{R}^d} v_i\,\mathrm{d}\nu_{t,x}(v),
  \qquad
  \mu_{ij}(t,x) = \int_{\mathbb{R}^d} v_i v_j\,\mathrm{d}\nu_{t,x}(v).
\]
For $\xi\in\mathbb{R}^d$,
Jensen's inequality applied to the convex function
$v\mapsto(v\cdot\xi)^2$ gives
\[
  \xi_i\xi_j\sigma_{ij}(t,x)
  = \int(v\cdot\xi)^2\,\mathrm{d}\nu_{t,x}
  - \Bigl(\int v\cdot\xi\,\mathrm{d}\nu_{t,x}\Bigr)^2
  \ge 0.
\]
Hence $\sigma_{ij}$ is a non-negative symmetric matrix-valued
Radon measure, the variance of the Young measure.

\emph{Step 4: Energy identity.}
On a closed manifold $\Omega$ of finite measure, the constant $\mathbf{1}$
is an admissible test function for the weak-$*$ limit.
Testing $\mu_{ij}(t)$ against $\delta_{ij}\mathbf{1}$
and using~\eqref{eq:CMV_energy_limit},
$\int_\Omega\mathrm{tr}(\mu)(t) = \lim_k\nrm{\bu^{h_k}(t)}_{L^2}^2 = E_0$,
we conclude with~\eqref{eq:CMV_sigma_def}
$
  \int_\Omega\mathrm{tr}(\sigma)(t)
  = E_0 - \nrm{\bw(t)}_{L^2}^2.
$
At $t=0$, the strong $L^2$-convergence $\bu^h(0)\to\bu_0$ implies
$\mu_{ij}(0)\to w_i(0)w_j(0)\,\mathrm{d}x$ in $\mathcal{M}(\Omega)$,
so $\sigma(0) = 0$ and $\bw(0) = \bu_0$.

\emph{Step 5: CMV Euler equation.}
Pass to the limit in~\eqref{eq:approx_CMV}.
The time-derivative term $-\int_0^T\int_\Omega\bu^{h_k}\cdot\partial_t\boldsymbol\varphi$
converges to $-\int_0^T\int_\Omega\bw\cdot\partial_t\boldsymbol\varphi$
by the time-pointwise weak convergence
$\bu^{h_k}(t)\rightharpoonup\bw(t)$ and dominated convergence.
The nonlinear term $\int_0^T\int_\Omega u^{h_k}_i u^{h_k}_j\,\partial_j\varphi_i$
converges to $\int_0^T\int_\Omega(w_iw_j + \sigma_{ij})\partial_j\varphi_i$
by the weak-$*$ convergence
$u^{h_k}_i u^{h_k}_j\,\mathrm{d}t\,\mathrm{d}x
\overset{*}{\rightharpoonup} w_iw_j\,\mathrm{d}t\,\mathrm{d}x + \sigma_{ij}$
from Step~2,
pairing against the continuous test function $\partial_j\varphi_i
\in C_c([0,T]\times\Omega)$.
The initial term $\int_\Omega\bu^{h_k}(0)\cdot\boldsymbol\varphi(\cdot,0)$
converges to $\int_\Omega\bu_0\cdot\boldsymbol\varphi(\cdot,0)$
by strong convergence of initial data.
The residual $R^{h_k}(\boldsymbol\varphi) = \OO(h_k^r) \to 0$
by \Cref{lem:weak_consistency}.
This yields~\eqref{eq:CMV_Euler}.
\end{proof}



\subsection{Bounded Domains: Proof of Consistency}
\label{app:bdy_consistency_proof}

\begin{proof}
We expand the truncation error as in the closed-manifold case
(\cref{thm:NS_consistency})
$\boldsymbol\tau_h^{0,\nu}
= \boldsymbol\tau_h^{\rm Euler} + \boldsymbol\tau_h^{\rm visc}$.
For interior edges, the estimates of
\Cref{thm:consistency,thm:NS_consistency} apply without change.
For boundary-adjacent edges
$e_j^*\in\mathcal{E}^*_\partial$ (those with at least one endpoint on
$\partial\Omega$) two modifications arise.
\begin{enumerate}[nosep]
  \item \emph{Extrusion:}
    Boundary edges have one endpoint on $\partial\Omega$.  By the Dirichlet boundary condition
    $\bu(v_\partial^*) = 0$, the extrusion weight
    $w_{jk}$ reduces to the one-sided average
    $\frac{1}{2}\bu(v_a^*)\cdot\hat{e}_k$.
  \item \emph{Hodge*:}
    Under \Cref{ass:bdy_mesh},
    the Voronoi cells at the boundary retain the required shape regularity,
    but the Hodge* approximation $(\bM_1)_{jj} = \star_1|_{e_j^*} + \OO(h^2)$
    degrades to $\OO(h)$ on the boundary layer of width $\OO(h)$.
\end{enumerate}

The interpolant $\bar\bv = \mathcal{R}_h\bu^\flat$ satisfies
$\bar v_j = \int_{e_j^*}\bu\cdot d\ell = 0$ for $j\in\mathcal{E}^*_\partial$
(since $\bu|_{\partial\Omega} = 0$), and the solution $\bv\in V_h^0$
satisfies $v_j = 0$ on the same edges.
Hence the error $e := \bv - \bar\bv$ vanishes on
$\mathcal{E}^*_\partial$. In the bilinear pairing
\[
  \langle e, \boldsymbol\tau_h^{0,\nu}\rangle_1
  = \sum_{j\notin\mathcal{E}^*_\partial}(\bM_1)_{jj}\,e_j\,\tau_j^{0,\nu}
  + \sum_{j\in\mathcal{E}^*_\partial}(\bM_1)_{jj}\,e_j\,\tau_j^{0,\nu},
\]
the second term vanishes, because $e_j=0$. Consequently no contribution from boundary edges 
enters the energy estimate of \Cref{app:NS_stability_proof}. Truncation accuracy at boundary edges does
therefore not affect the convergence rate.

The remaining contribution comes from boundary-adjacent edges
$j\notin\mathcal{E}^*_\partial$ but adjacent to a boundary edge through
the operators $\tD_1, \codiff$. On these edges the smooth
velocity satisfies $|\bu(x)|\le C\,\mathrm{dist}(x,\partial\Omega)\,\nrm{\nabla\bu}_{L^\infty}
\le Ch\,\nrm{\bu}_{W^{1,\infty}}$ within $O(h)$ of $\partial\Omega$,
but mesh shape-regularity and the standard Hodge*/extrusion analysis
of \Cref{thm:consistency,thm:NS_consistency} apply unchanged. The pointwise truncation
on near-boundary edges therefore matches the interior rate $\OO(h)$
in case (A) and $\OO(h^2)$ in case (B) of \Cref{conv:cases}.

After applying $\bP_h^0$, the $\OO(h^{r_\star})$ discrepancy from the
Leray projection is absorbed into the truncation error at no loss of
order. Combining interior and near-boundary contributions and
discarding the boundary-edge component (which vanishes against $e$)
gives
\[
  |\langle e,\boldsymbol\tau_h^{0,\nu}\rangle_1|
  \le C_\tau^{0,\nu}\,h^r\,\nrm{e}_{L_h^2},
  \qquad r = r_{\rm rec}\text{ under \Cref{conv:cases}}.
\]
This is the bound that enters the bounded-domain analogue of
\Cref{thm:NS_stability}; the convergence rate of the bounded-domain
scheme matches the closed-manifold rate of \Cref{thm:NS_convergence}.
\end{proof}

\addcontentsline{toc}{section}{References}
\bibliographystyle{abbrv}%
\bibliography{references_short}

@book{john2018finite,
  title={Finite Element Methods for Incompressible Flow Problems},
  author={John, Volker},
  series={Springer Series in Computational Mathematics},
  volume={51},
  year={2018},
  publisher={Springer International Publishing},
  doi={10.1007/978-3-319-83366-8},
  isbn={978-3-319-83366-8}
}

@article{guermond2006,
  author  = {Guermond, J.-L.},
  title   = {Finite-element-based {F}aedo--{G}alerkin weak solutions to the
             {N}avier--{S}tokes equations in the three-dimensional torus
             are suitable},
  journal = {J. Math. Pures Appl. (9)},
  volume  = {85},
  number  = {3},
  pages   = {451--464},
  year    = {2006},
}

@article{guermond2007,
  author  = {Guermond, J.-L.},
  title   = {{F}aedo--{G}alerkin weak solutions of the {N}avier--{S}tokes
             equations with {D}irichlet boundary conditions are suitable},
  journal = {J. Math. Pures Appl. (9)},
  volume  = {88},
  number  = {1},
  pages   = {87--106},
  year    = {2007},
}

@article{christiansen2007,
  author    = {Christiansen, S. H.},
  title     = {Stability of {Hodge} decompositions in finite element spaces of differential forms in arbitrary dimension},
  journal   = {Numer. Math.},
  volume    = {107},
  number    = {1},
  pages     = {87--106},
  year      = {2007},
  doi       = {10.1007/s00211-007-0081-2},
}

@incollection{raviartthomas1977,
  author    = {Raviart, P.-A. and Thomas, J.-M.},
  title     = {A mixed finite element method for 2nd order elliptic problems},
  booktitle = {Mathematical Aspects of Finite Element Methods},
  series    = {Lecture Notes in Math.},
  volume    = {606},
  pages     = {292--315},
  publisher = {Springer},
  year      = {1977},
}

@article{nedelec1980,
  author    = {N{\'e}d{\'e}lec, J.-C.},
  title     = {Mixed finite elements in {$\mathbb{R}^3$}},
  journal   = {Numer. Math.},
  volume    = {35},
  number    = {3},
  pages     = {315--341},
  year      = {1980},
}

@incollection{arakawalamb1977,
  author    = {Arakawa, A. and Lamb, V. R.},
  title     = {Computational design of the basic dynamical processes of the {UCLA} general circulation model},
  booktitle = {Methods in Computational Physics},
  volume    = {17},
  pages     = {173--265},
  publisher = {Academic Press},
  year      = {1977},
}

@article{thuburncotter2012,
  author    = {Thuburn, J. and Cotter, C. J.},
  title     = {A framework for mimetic discretization of the rotating shallow-water equations on arbitrary polygonal grids},
  journal   = {SIAM J. Sci. Comput.},
  volume    = {34},
  number    = {3},
  pages     = {B203--B225},
  year      = {2012},
}

@article{cotter2023,
  author    = {Cotter, C. J.},
  title     = {Compatible finite element methods for geophysical fluid dynamics},
  journal   = {Acta Numer.},
  volume    = {32},
  pages     = {291--393},
  year      = {2023},
}

@incollection{Desbrun,
  author    = {Desbrun, M. and Kanso, E. and Tong, Y.},
  title     = {Discrete Differential Forms for Computational Modeling},
  booktitle = {Discrete Differential Geometry},
  series    = {Oberwolfach Seminars},
  volume    = {38},
  publisher = {Birkh{\"a}user},
  pages     = {287--324},
  year      = {2008},
}

@phdthesis{Hirani,
  author    = {Hirani, A. N.},
  title     = {Discrete Exterior Calculus},
  school    = {California Institute of Technology},
  year      = {2003},
}

@book{Marsden,
  author    = {Marsden, J. E. and Ratiu, T. S.},
  title     = {Introduction to Mechanics and Symmetry},
  publisher = {Springer},
  series    = {Texts in Applied Mathematics},
  volume    = {17},
  edition   = {2nd},
  year      = {1999},
}

@article{bossavit1998,
  author    = {Bossavit, A.},
  title     = {Computational Electromagnetism: Variational Formulations, Complementarity, Edge Elements},
  journal   = {Academic Press},
  year      = {1998},
  note      = {Monograph},
}

@article{elcott2007,
  author    = {Elcott, S. and Tong, Y. and Kanso, E. and Schr{\"o}der, P. and Desbrun, M.},
  title     = {Stable, circulation-preserving, simplicial fluids},
  journal   = {ACM Trans. Graph.},
  volume    = {26},
  number    = {1},
  pages     = {4},
  year      = {2007},
}

@inproceedings{mullen2009,
  author    = {Mullen, P. and Crane, K. and Pavlov, D. and Tong, Y. and Desbrun, M.},
  title     = {Energy-preserving integrators for fluid animation},
  booktitle = {ACM SIGGRAPH 2009 Papers},
  pages     = {1--8},
  year      = {2009},
}

@article{pavlov2011,
  author    = {Pavlov, D. and Mullen, P. and Tong, Y. and Kanso, E. and Marsden, J. E. and Desbrun, M.},
  title     = {Structure-preserving discretization of incompressible fluids},
  journal   = {Phys. D},
  volume    = {240},
  number    = {6},
  pages     = {443--458},
  year      = {2011},
}

@article{arnold2006,
  author    = {Arnold, D. N. and Falk, R. S. and Winther, R.},
  title     = {Finite element exterior calculus, homological techniques, and applications},
  journal   = {Acta Numer.},
  volume    = {15},
  pages     = {1--155},
  year      = {2006},
}

@article{arnold2010,
  author    = {Arnold, D. N. and Falk, R. S. and Winther, R.},
  title     = {Finite element exterior calculus: from {H}odge theory to numerical stability},
  journal   = {Bull. Amer. Math. Soc. (N.S.)},
  volume    = {47},
  number    = {2},
  pages     = {281--354},
  year      = {2010},
}

@article{thuburncotter2015,
  author    = {Thuburn, J. and Cotter, C. J.},
  title     = {A primal--dual mimetic finite element scheme for the rotating shallow water equations on polygonal spherical meshes},
  journal   = {J. Comput. Phys.},
  volume    = {290},
  pages     = {274--297},
  year      = {2015},
}

@article{ringler2010,
  author    = {Ringler, T. D. and Thuburn, J. and Klemp, J. B. and Skamarock, W. C.},
  title     = {A unified approach to energy conservation and potential vorticity dynamics for arbitrarily-structured {C}-grids},
  journal   = {J. Comput. Phys.},
  volume    = {229},
  number    = {9},
  pages     = {3065--3090},
  year      = {2010},
}

@article{korn2017,
  author    = {Korn, P.},
  title     = {Formulation of an unstructured grid model for global ocean dynamics},
  journal   = {J. Comput. Phys.},
  volume    = {339},
  pages     = {525--552},
  year      = {2017},
}

@article{moffatt1969,
  author    = {Moffatt, H. K.},
  title     = {The degree of knottedness of tangled vortex lines},
  journal   = {J. Fluid Mech.},
  volume    = {35},
  number    = {1},
  pages     = {117--129},
  year      = {1969},
}

@article{charnyi2017,
  author    = {Charnyi, S. and Heister, T. and Olshanskii, M. A. and Rebholz, L. G.},
  title     = {On conservation laws of {N}avier--{S}tokes {G}alerkin discretizations},
  journal   = {J. Comput. Phys.},
  volume    = {337},
  pages     = {289--308},
  year      = {2017},
}

@article{zhang2022,
  author    = {Zhang, Y. and Palha, A. and Gerritsma, M. and Rebholz, L. G.},
  title     = {A mass-, kinetic energy- and helicity-conserving mimetic dual-field discretization for three-dimensional incompressible {N}avier--{S}tokes equations, {P}art {I}: {P}eriodic domains},
  journal   = {J. Comput. Phys.},
  volume    = {451},
  pages     = {110868},
  year      = {2022},
}

@article{beiraoVeiga2013basic,
  author    = {Beir{\~a}o da Veiga, L. and Brezzi, F. and Cangiani, A.
               and Manzini, G. and Marini, L. D. and Russo, A.},
  title     = {Basic principles of virtual element methods},
  journal   = {Math. Models Methods Appl. Sci.},
  volume    = {23},
  number    = {1},
  pages     = {199--214},
  year      = {2013},
}

@article{beiraoVeiga2018NS,
  author    = {Beir{\~a}o da Veiga, L. and Lovadina, C. and Vacca, G.},
  title     = {Virtual elements for the {N}avier--{S}tokes problem on polygonal meshes},
  journal   = {SIAM J. Numer. Anal.},
  volume    = {56},
  number    = {3},
  pages     = {1210--1242},
  year      = {2018},
}

@article{dodziuk1976,
  author    = {Dodziuk, J. and Patodi, V. K.},
  title     = {Riemannian structures and triangulations of manifolds},
  journal   = {J. Indian Math. Soc. (N.S.)},
  volume    = {40},
  number    = {1--4},
  pages     = {1--52},
  year      = {1976},
}

@book{whitney1957,
  author    = {Whitney, H.},
  title     = {Geometric Integration Theory},
  publisher = {Princeton University Press},
  address   = {Princeton, NJ},
  year      = {1957},
}

@article{Leray1934,
  author    = {Leray, J.},
  title     = {Sur le mouvement d'un liquide visqueux emplissant l'espace},
  journal   = {Acta Math.},
  volume    = {63},
  pages     = {193--248},
  year      = {1934},
}

@article{Perot_JCP,
  author    = {Perot, B.},
  title     = {Conservation properties of unstructured staggered mesh schemes},
  journal   = {J. Comput. Phys.},
  volume    = {159},
  number    = {1},
  pages     = {58--89},
  year      = {2000},
}

@article{scheffer1993,
  author    = {Scheffer, V.},
  title     = {An inviscid flow with compact support in space-time},
  journal   = {J. Geom. Anal.},
  volume    = {3},
  number    = {4},
  pages     = {343--401},
  year      = {1993},
}

@article{shnirelman1997,
  author    = {Shnirelman, A.},
  title     = {On the nonuniqueness of weak solution of the {E}uler equation},
  journal   = {Comm. Pure Appl. Math.},
  volume    = {50},
  number    = {12},
  pages     = {1261--1286},
  year      = {1997},
}

@article{DeLellisSzekelyhidi2009,
  author    = {De Lellis, C. and Sz{\'e}kelyhidi Jr., L.},
  title     = {The {E}uler equations as a differential inclusion},
  journal   = {Ann. of Math. (2)},
  volume    = {170},
  number    = {3},
  pages     = {1417--1436},
  year      = {2009},
}

@article{DeLellisSzekelyhidi2013,
  author    = {De Lellis, C. and Sz{\'e}kelyhidi Jr., L.},
  title     = {Dissipative continuous {E}uler flows},
  journal   = {Invent. Math.},
  volume    = {193},
  pages     = {377--407},
  year      = {2013},
}

@article{Isett2018,
  author    = {Isett, P.},
  title     = {A proof of {O}nsager's conjecture},
  journal   = {Ann. of Math. (2)},
  volume    = {188},
  number    = {3},
  pages     = {871--963},
  year      = {2018},
}

@article{BDLSV2019,
  author    = {Buckmaster, T. and De Lellis, C. and Sz{\'e}kelyhidi Jr., L. and Vicol, V.},
  title     = {Onsager's conjecture for admissible weak solutions},
  journal   = {Comm. Pure Appl. Math.},
  volume    = {72},
  number    = {2},
  pages     = {229--274},
  year      = {2019},
}

@article{CET1994,
  author    = {Constantin, P. and E, W. and Titi, E. S.},
  title     = {Onsager's conjecture on the energy conservation for solutions
               of {E}uler's equation},
  journal   = {Comm. Math. Phys.},
  volume    = {165},
  number    = {1},
  pages     = {207--209},
  year      = {1994},
}

@article{BDLS2011,
  author    = {Brenier, Y. and De Lellis, C. and Sz{\'e}kelyhidi Jr., L.},
  title     = {Weak-strong uniqueness for measure-valued solutions},
  journal   = {Comm. Math. Phys.},
  volume    = {305},
  number    = {2},
  pages     = {351--361},
  year      = {2011},
}

@article{Hopf1951,
  author    = {Hopf, E.},
  title     = {{\"U}ber die {A}nfangswertaufgabe f{\"u}r die hydrodynamischen
               {G}rundgleichungen},
  journal   = {Math. Nachr.},
  volume    = {4},
  pages     = {213--231},
  year      = {1951},
}

@article{DuchonRobert2000,
  author    = {Duchon, J. and Robert, R.},
  title     = {Inertial energy dissipation for weak solutions of incompressible
               {E}uler and {N}avier--{S}tokes equations},
  journal   = {Nonlinearity},
  volume    = {13},
  number    = {1},
  pages     = {249--255},
  year      = {2000},
}

@article{arnold1966,
  author  = {Arnold, Vladimir I.},
  title   = {Sur la g{\'e}om{\'e}trie diff{\'e}rentielle des groupes de {L}ie de dimension infinie et ses applications {\`a} l'hydrodynamique des fluides parfaits},
  journal = {Annales de l'Institut Fourier},
  volume  = {16},
  number  = {1},
  pages   = {319--361},
  year    = {1966},
}

@article{korn2022james,
  author  = {Korn, Peter and Br{\"u}ggemann, Nils and Jungclaus, J. H.
             and Lorenz, S. J. and Gutjahr, O. and Haak, H. and
             Linardakis, L. and Mehlmann, C. and Mikolajewicz, U. and
             Notz, D. and M{\"u}ller, W. A. and Putrasahan, D. A. and
             Singh, V. and von Storch, J.-S. and Zhu, X. and Marotzke, J.},
  title   = {ICON-O: The ocean component of the ICON Earth System Model
             --- Global simulation characteristics and local telescoping
             capability},
  journal = {Journal of Advances in Modeling Earth Systems},
  year    = {2022},
  volume  = {14},
  number  = {10},
  pages   = {e2021MS002952},
  doi     = {10.1029/2021MS002952}
}

@article{hohenegger2023gmd,
  author  = {Hohenegger, Cathy and Korn, Peter and Linardakis, Leonidas
             and Redler, Ren{\'e} and Schnur, Reiner and Adamidis,
             Panagiotis and Bao, Jiawei and Bastin, Sebastian and
             Behravesh, Milad and Bergemann, Martin and others},
  title   = {ICON-Sapphire: simulating the components of the {E}arth
             system and their interactions at kilometer and
             subkilometer scales},
  journal = {Geoscientific Model Development},
  year    = {2023},
  volume  = {16},
  number  = {2},
  pages   = {779--811},
  doi     = {10.5194/gmd-16-779-2023}
}

@article{BrezziLipnikovShashkov2005,
  author  = {Brezzi, F. and Lipnikov, K. and Shashkov, M.},
  title   = {Convergence of the mimetic finite difference method for
             diffusion problems on polyhedral meshes},
  journal = {SIAM J. Numer. Anal.},
  volume  = {43},
  number  = {5},
  pages   = {1872--1896},
  year    = {2005},
}

@article{bossavit1988,
  author  = {Bossavit, A.},
  title   = {Whitney forms: a class of finite elements for
             three-dimensional computations in electromagnetism},
  journal = {IEE Proc. A},
  volume  = {135},
  number  = {8},
  pages   = {493--500},
  year    = {1988},
}

@article{bossavit2000generalized,
  author  = {Bossavit, A.},
  title   = {Generalized finite differences in computational electromagnetics},
  journal = {Progress in Electromagnetics Research},
  volume  = {32},
  pages   = {45--64},
  year    = {2001},
}

@article{hiptmair2001,
  author  = {Hiptmair, R.},
  title   = {Discrete {H}odge operators},
  journal = {Numer. Math.},
  volume  = {90},
  number  = {2},
  pages   = {265--289},
  year    = {2001},
}

@misc{desbrun2005,
  author  = {Desbrun, M. and Hirani, A. N. and Leok, M. and Marsden, J. E.},
  title   = {Discrete Exterior Calculus},
  year    = {2005},
  eprint  = {math/0508341},
  archivePrefix = {arXiv},
  primaryClass  = {math.DG},
  note    = {arXiv:math/0508341},
}

@incollection{eymard2000finite,
  author    = {Eymard, R. and Gallou{\"e}t, T. and Herbin, R.},
  title     = {Finite volume methods},
  booktitle = {Handbook of Numerical Analysis, Vol.~{VII}},
  editor    = {Ciarlet, P. G. and Lions, J. L.},
  publisher = {North-Holland},
  address   = {Amsterdam},
  pages     = {713--1020},
  year      = {2000},
}

@incollection{Tartar1979,
  author    = {Tartar, L.},
  title     = {Compensated compactness and applications to partial
               differential equations},
  booktitle = {Nonlinear Analysis and Mechanics: Heriot--Watt Symposium,
               Vol.~{IV}},
  editor    = {Knops, R. J.},
  series    = {Research Notes in Mathematics},
  volume    = {39},
  publisher = {Pitman},
  address   = {Boston},
  pages     = {136--212},
  year      = {1979},
}

@article{demlow2009higher,
  author  = {Demlow, Alan},
  title   = {Higher-Order Finite Element Methods and Pointwise Error Estimates for Elliptic Problems on Surfaces},
  journal = {SIAM Journal on Numerical Analysis},
  volume  = {47},
  number  = {2},
  pages   = {805--827},
  year    = {2009},
  doi     = {10.1137/070708135}
}

@article{poveda2023pointwise,
  author  = {Poveda, Leonardo A. and Peixoto, Pedro},
  title   = {On Pointwise Error Estimates for {Vorono\"i}-Based Finite Volume Methods for the {Poisson} Equation on the Sphere},
  journal = {Advances in Computational Mathematics},
  year    = {2023},
  doi     = {10.1007/s10444-023-10041-3},
  note    = {arXiv:2206.03685}
}

@article{schatzwahlbin1995,
  author  = {Schatz, A. H. and Wahlbin, L. B.},
  title   = {Interior Maximum-Norm Estimates for Finite Element Methods, Part {II}},
  journal = {Mathematics of Computation},
  volume  = {64},
  pages   = {907--928},
  year    = {1995}
}

@article{schatz1998,
  author  = {Schatz, A. H.},
  title   = {Pointwise Error Estimates and Asymptotic Error Expansion Inequalities for the Finite Element Method on Irregular Grids. {I}. Global Estimates},
  journal = {Mathematics of Computation},
  volume  = {67},
  pages   = {877--899},
  year    = {1998}
}
\end{document}